\documentclass[a4paper,11pt,reqno,oneside]{amsart}
\input{declarations}
\pdfoutput=1 

\title{Higher operations in string topology of classifying spaces}
\date{\today}

\author{Anssi Lahtinen}
\address{%
Universit\"{a}t Hamburg\\
Fachbereich Mathematik/AZ\\
Bundesstrasse 55\\ 
20146 Hamburg\\
Germany}
\email{anssi.lahtinen@uni-hamburg.de}


\begin{document}

\begin{abstract}
Examples of non-trivial higher string topology operations have
been regrettably rare in the literature. 
In this paper, working in the context of string topology 
of classifying spaces, we provide explicit calculations of 
a wealth of non-trivial higher string topology operations
associated to a number of different Lie groups.
As an application of these calculations, we obtain 
an abundance of interesting homology classes in the 
twisted homology groups of automorphism groups of 
free groups, the ordinary homology groups of 
holomorphs of free groups, and
the ordinary homology groups of 
affine groups over the integers and 
the field of two elements.
\end{abstract}

\maketitle

\setcounter{tocdepth}{2}
\tableofcontents

\section{Introduction}

Since Chas and Sullivan's original observation that the 
homology of the free loop space of a manifold
admits the structure of a Batalin--Vilkovisky algebra
\cite{ChasSullivan},
string topology has been extended in many ways,
both by replacing Batalin--Vilkovisky algebras
with more elaborate algebraic structures, and by
constructing similar structures from objects other than manifolds.
In modern formulations, string topology of a space $E$ is 
typically expressed as a field theory consisting (roughly speaking)
of compatible maps 
\[
	H_\ast (\calM) \tensor H_\ast (E^X) \longto H_\ast (E^Y)
\]
where $X$ and $Y$ are (typically) compact 1-manifolds, 
$E^X$ and $E^Y$ denote the
spaces of maps from $X$ and $Y$ into $E$, 
respectively, and $\calM$ is some kind of
space of cobordisms between $X$ and $Y$.
The space $E$ could be, as in Chas and Sullivan's original work,
a closed oriented manifold \cite{Godin},
or it could be
the classifying space of a compact Lie group \cite{ChataurMenichi}
or even a stack \cite{BGNX}. 
The space $\calM$, on the other hand, might be 
a model for the moduli space of Riemann surfaces
modelled on a cobordism $\Sigma\colon X \to Y$
\cite{Godin,ChataurMenichi,BGNX}
or a space of suitable Sullivan chord diagrams 
\cite{Chataur-bordism};
see also \cite{Kaufmann,PoirierRounds,WahlWesterland}.
Typically, the degree 0 part of the structure, 
that is, the part corresponding 
to $H_0$ of the spaces $\calM$, is relatively accessible. 
On the other hand, \emph{higher string topology operations}, that is, 
maps $H_\ast (E^X) \to H_\ast (E^Y)$ corresponding to 
positive-degree homology classes of $\calM$, 
have proved much harder to understand: 
while the Batalin--Vilkovisky $\Delta$-operator has 
been computed in a number of cases (see eg.\
\cite{TamanoiStiefel, 
MenichiSpheres, 
HepworthProjective, 
HepworthLie, 
Yang,
BB}),
examples of other non-trivial higher operations 
remain rare.
Indeed, Tamanoi \cite{TamanoiStable}
has shown that string topology operations
associated to the best-understood classes
in the homology of moduli spaces of Riemann surfaces,
namely those in the stable range, vanish.
In the context of string topology of manifolds,
Wahl \cite{Wahl}
has recently succeeded in constructing families of 
non-trivial higher operations not arising 
from the $\Delta$-operator, and it appears that 
these are the only examples of such operations 
so far (with the possible exception of the 
Goresky--Hingston $\circledast$-product \cite{GoreskyHingston},
which could be interpreted as a higher operation
associated to a degree 1 homology class 
of a compactified moduli space).
On the other hand, in the context of string topology
of classifying spaces of compact Lie groups,
it seems that the first and only 
examples of non-trivial higher operations
are given by the computation in 
\cite[section 9]{HL} for the group $\Z/2$.

The present paper can be viewed as an argument for 
the following three theses: first, that non-trivial higher string 
topology operations exist in abundance; second, that
many examples of such operations can be computed
explicitly; and third, that such computations have 
applications to mathematics outside string topology.
To support the first two theses, we will provide 
a wealth of explicit calculations of non-trivial
higher operations in the string topology 
of classifying spaces of 
a number of different compact Lie groups.
We will work within the context of the recent extension \cite{HL}
(in characteristic 2)
of Chataur and Menichi's
string topology of classifying spaces \cite{ChataurMenichi} into 
a novel kind of field theory called 
a Homological H-Graph Field Theory,
or an HHGFT. 
Recall from \cite{HL} that an \emph{h-graph}
means any space homotopy equivalent
to a finite graph, and that an 
\emph{h-graph cobordism} $S\colon X \hto Y$
from an h-graph $X$ to another h-graph $Y$ is a diagram 
$X \incl S \hookleftarrow Y$
of h-graphs satisfying certain conditions
\cite[Definition~2.5]{HL}.
A \emph{family of h-graph cobordisms}
$S/B \colon X \hto Y$ from $X$ to $Y$ over a space $B$
consists of a fibration $S \to B$ and maps 
\begin{equation}
\label{maps:family-of-h-graph-cobs}
	X\times B \longto S \longleftarrow Y\times B
\end{equation}
over $B$ satisfying certain conditions \cite[Definition~2.8]{HL}
which in particular imply that for every $b\in B$,
the restriction to fibres over $b$ in 
\eqref{maps:family-of-h-graph-cobs}
gives an h-graph cobordism $S_b \colon X \hto Y$.
An h-graph cobordism $S \colon X \hto Y$  is 
called \emph{positive} if $X$ meets every component of $S$,
and a family of h-graph cobordism $S/B \colon X \hto Y$
is called positive if all its fibres are positive.
Given a compact Lie group $G$, the HHGFT $\Phi^G$
constructed in \cite{HL}  consists of maps
\[
	\Phi^G(S/B) 
	\colon 
	H_{\ast + \dim(G) \chi(S,X)}(B) \tensor H_\ast (BG^X)
	\longto
	H_\ast (BG^Y),
\]
one for every positive family of h-graph cobordisms 
$S/B \colon X \hto Y$, satisfying various 
compatibility axioms \cite[subsection 3.1]{HL}. 
Here $\chi(S,X)$ denotes the locally constant
function $B \to \Z$ sending each point $b\in B$
to the Euler characteristic $\chi(S_b,X)$.

The string topology operations we compute
are associated to families of h-graph cobordisms
$S_n/B\Sigma_n\colon \pt \hto \pt$ 
from the one-point space to itself
constructed as follows.
Let $\hat S_n \colon \pt \hto \pt$ be the h-graph
cobordism depicted below 
consisting of $n$ strings joining the two points.
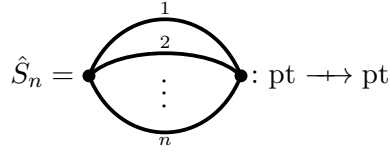
\[
	\hat S_n 
	= 
	\begin{tikzpicture}[scale=0.02,baseline=-2.7]
		\path[ARC, fill=black] (0,0) circle (3);
		\path[ARC, fill=black] (100,0) circle (3);
		\path[ARC] (0,0) .. controls (20,50) and (80,50) .. (100,0) 
				node [midway, above, yshift=-2.5] {\tiny $1$};
		\path[ARC] (0,0) .. controls (20,20) and (80,20) .. (100,0) 
				node [midway, above, yshift=-2.5] {\tiny $2$};
		\path (50,30) -- (50,-50) node [black,midway,rotate=90]  {$\cdots$};
		
		\path[ARC] (0,0) .. controls (20,-50) and (80,-50) .. (100,0) 
				node [midway, below, yshift=2.5] {\tiny $n$};
	\end{tikzpicture}
	\colon
	\pt \longhto \pt
\]
The symmetric group $\Sigma_n$ acts on 
$\hat S_n$ by permuting the $n$ strings,
and performing the Borel construction gives us the
family h-graph cobordisms
\[
	S_n/B\Sigma_n 
	= 
	E\Sigma_n \times_{\Sigma_n} \hat S_n / B\Sigma_n 
	\colon 
	\pt \longhto \pt.
\]
We will compute explicitly the operations 
\[
	\Phi^G(S_n/B\Sigma_n)
	\colon 
	H_{\ast-\dim(G)(n-1)}(B\Sigma_n)\tensor H_\ast(BG)
	\longto
	H_\ast(BG)	
\]
for all $n\geq 1$ when $G$ is an elementary abelian 2-group
or a dihedral group of order $2$ mod $4$,
for $1 \leq n \leq 7$ when $G$ is a torus  $\T^l = (S^1)^l$,
and for $1\leq n \leq 3$ when $G=SU(2)$. 
The computation done in \cite[section 9]{HL} amounts 
to the calculation of $\Phi^{\Z/2} (S_2/B\Sigma_2)$.
It turns out that for any compact Lie group $G$, 
all higher operations associated to $S_n/B\Sigma_n$ 
vanish when $n$ is not a power of 2. 
On the other hand, when $n$ is a power of 
2, our calculations provide a wealth of non-trivial 
operations for all of the groups $G$ above.
In particular, these calculations give the first 
examples of non-trivial higher operations 
in string topology of classifying spaces 
where the group $G$ is non-abelian
or a positive-dimensional compact Lie group.

To support the third thesis, we show 
that our calculations detect a wealth 
of non-trivial elements in the homology groups of 
certain highly interesting groups,
namely in the twisted homology groups  
$H_\ast (B\Aut(F_n);\, \tilde \F_2^n)$
of $\Aut(F_n)$
and in the ordinary mod 2 homology groups 
of the holomorph $\Hol(F_n) = F_n \rtimes \Aut(F_n)$
and the affine groups 
$\Aff_n(\Z) = \Z^n \rtimes GL_n(\Z)$
and
$\Aff_n(\F_2) = \F_2^n \rtimes GL_n(\F_2)$.
Here $F_n$ denotes the free group on $n$ generators,
and $\tilde \F_2^n$ denotes $\F_2^n$ equipped 
with the $\Aut(F_n)$-action given by the 
composite $\Aut(F_n) \to GL_n(\Z) \to GL_n(\F_2)$,
where the first homomorphism is given by abelianization
and the second one by reduction mod~2.
The elements we detect are exhibited in 
Corollaries~\ref{cor:hol-app},
\ref{cor:non-stab},
\ref{cor:aut-app},
\ref{cor:hol-fam2},
\ref{cor:aff-app},
and
\ref{cor:aff-app2} below.
For each of the aforementioned groups, 
the homology groups are known to stabilize
in the sense that the homology groups for successive
values of $n$ are connected by a map,
the \emph{stabilization map}, which is 
known to be an isomorphism in degrees small enough
relative to $n$, the \emph{stable range}. 
In each case, the homology groups are 
well understood in the stable range, 
but remain mysterious outside it; 
see section~\ref{subsec:stability-results}
and the discussion following
Corollary~\ref{cor:aut-app}.
It is in this poorly understood unstable 
range where the elements we 
detect live. Indeed, many of the elements 
are \emph{unstable} in the sense that they 
are annihilated by an iteration of the 
stabilization map; 
see Corollaries~\ref{cor:hol-fam2} and
\ref{cor:aff-app2}
as well as Corollary~\ref{cor:aut-app}
together with the discussion following it.

\begin{example}
The reader may appreciate an example of 
the non-trivial homology classes we construct.
Let $u$ be an integer of the form $u=\sum_{i=1}^k u_i$
where $u_1,\ldots, u_k$ are positive integers with the
property that no two of them
have a $1$ in common in their binary expansions.
Let $f \colon \{1,\ldots,k\} \to \{1,\ldots,r\}$
be a surjective function, and denote
\[
	N = \sum_{i=1}^r  (2^{|f^{-1}(i)|}-1).
\]
Then, in view of Proposition~\ref{prop:example-of-non-triv-op},
our computations for the group $G=\Z/2$ combine with:
\begin{enumerate}
\item Corollary~\ref{cor:hol-app} 
	to produce a non-trivial element in 
	$H_u(B\Hol(F_N))$.
	This class is stable,
	but not in the image of the stabilization map
	$H_u(B\Hol(F_{N-1})) \to H_u(B\Hol(F_{N}))$.
	See Remark~\ref{rk:stab} and Corollary~\ref{cor:non-stab}.
\item Corollary~\ref{cor:aut-app}
	to produce a non-trivial element in 
	$H_{u-1} (B\Aut(F_N);\,\tilde \F_2^N)$.
	Like all elements in 
	$H_* (B\Aut(F_N);\,\tilde \F_2^N)$,
	this class is unstable. See 
	the discussion following
	Corollary~\ref{cor:aut-app}.
\item Corollary~\ref{cor:hol-fam2}
	to produce a non-trivial unstable element in 
	$H_u(B\Hol(F_N))$ which is not in 
	the image of the stabilization map
	$H_u(B\Hol(F_{N-1})) \to H_u(B\Hol(F_{N}))$.
\item Corollary~\ref{cor:aff-app} 
	to produce a non-trivial element in 
	$H_u(B\Aff_N(\Z))$ and another one in 
	$H_u(B\Aff_N(\F_2))$.
	Like all positive-dimensional classes in the 
	homology of $\Aff_N(\F_2)$, the
	element in $H_u(B\Aff_N(\F_2))$
	is unstable. See Theorems~\ref{thm:glstab} 
	and	\ref{thm:quillen-calc}.
\item Corollary~\ref{cor:aff-app2} 
	to produce a non-trivial unstable element in 
	$H_u(B\Aff_N(\Z))$ and another one in 
	$H_u(B\Aff_N(\F_2))$.
\end{enumerate}
\end{example}

The paper is organized as follows. 
In section~\ref{sec:product-thm}, we prove 
a general result,
Theorem~\ref{thm:product-thm},
which expresses $\Phi^{G_1\times G_2}$
in terms of $\Phi^{G_1}$ and $\Phi^{G_2}$.
This result is used to reduce the 
computation of $\Phi^{(\Z/2)^l}$  and $\Phi^{\T^l}$
to the calculation of 
$\Phi^{(\Z/2)^1}$  and $\Phi^{\T^1}$, respectively,
and could be used to further extend the list 
of groups $G$ for which the operations
$\Phi^G(S_n/B\Sigma_n)$
are known by taking products of 
the groups we consider.
In section~\ref{sec:vanishing-result}, we  prove
a vanishing result, Theorem~\ref{thm:vanishing},
which implies the vanishing of 
the higher operations 
associated to $S_n/B\Sigma_n$
when $n$ is not a power of 2, and 
reduces the calculation of 
the operations $\Phi^G(S_{2^k}/B\Sigma_{2^k})$
to the computation of certain maps
\[
	\alpha^G_k
	\colon 
	H_{\ast-\dim(G)(2^k-1)}(BV_k)\tensor H_\ast(BG)
	\longto 
	H_\ast(BG).
\]
where $V_k$ denotes the elementary abelian 2-group 
of dimension $k$. See Definition~\ref{def:alphagk}.
In section~\ref{sec:preliminary-calculations},
we develop descriptions of the maps $\alpha^G_k$
suitable for computations.
Sections~\ref{sec:e-and-dcalc},
\ref{sec:tcalc} and \ref{sec:su2calc}
are then devoted to the computation of the 
maps $\alpha^G_k$
when $G$ is an elementary abelian 2-group
or a dihedral group of order 2 mod 4,
a torus,
or $SU(2)$,
respectively. 
Finally, in section~\ref{sec:applications}
we discuss the aforementioned applications
of our computations to the homology of 
$\Aut(F_n)$, $\Hol(F_n)$,
$\Aff_n(\Z)$ and $\Aff_n(\F_2)$.

\subsection*{Notation and conventions}

Throughout the paper, we will work over the 
field $\F_2$ of two elements. Unless stated
otherwise, homology and cohomology groups 
are taken with $\F_2$ coefficients.
If $S/B \colon X \hto Y$ is a positive family of 
h-graph cobordisms, we denote by 
$\Phi^G(S/B)^\sharp$ the adjoint
\[
	\Phi^G(S/B)^\sharp \colon H_{\ast + \dim(G) \chi(S,X)}(B) 
	\longto 
	\Hom_\ast(H_\ast(BG^X), H_\ast(BG^Y))
\]
of the map 
\[
	\Phi^G(S/B) 
	\colon
	H_{\ast + \dim(G) \chi(S,X)}(B) \tensor H_\ast(BG^X)
	\longto
	H_\ast(BG^Y).
\]
Here $\Hom_\ast$ denotes the internal hom in the 
category of graded $\F_2$-vector spaces.
If $S\colon X \hto Y$ is an h-graph cobordism,
we write $\hAut(S)$ for the topological monoid
of self homotopy equivalences of $S$ fixing $X$ and $Y$
pointwise. 
The symbols $\ucirc$ and $\usqcup$ 
refer to external composition and external disjoint union
of families of h-graph cobordisms, respectively; 
see \cite[Definition~2.8]{HL}.
If $H$ is a group or monoid, possibly with a topology,
$BH$ and $EH$ always stand for the bar constructions 
$B(\pt,H,\pt)$ and $B(\pt,H,H)$, respectively. 
See for example~\cite[section 7]{MayClassifying}.
As in \cite{HL}, when working with topological spaces,
we will work within the category of $k$-spaces,
and when working with fibred spaces,
the base spaces are in addition assumed to be
weak Hausdorff spaces.

\subsection*{Acknowledgements}

The author would like to thank 
Tilman Bauer,
Gregory Ginot, 
Jesper Grodal,
Richard Hepworth,
Ib Madsen,
Luc Menichi,
Christine Vespa,
Karen Vogtmann
and 
Nathalie Wahl
for useful conversations
and 
the editor 
and 
the anonymous referee 
for valuable comments and suggestions.

\section{The product theorem}
\label{sec:product-thm}

Let $G_1$ and $G_2$ be compact Lie groups. For any space
$X$, the projections of $G_1\times G_2$ onto $G_1$ and $G_2$ induce 
a natural homeomorphism
\[
	B(G_1\times G_2)^X
	\xto{\ \homeom\ }
	BG_1^X \times BG_2^X.
\]
Combining this with the homology cross product, we obtain a 
natural isomorphism
\[
	\kappa_X \colon H_\ast(BG_1^X) \tensor H_\ast(BG_2^X) 
						\xto{\ \homeom\ }
					H_\ast(B(G_1\times G_2)^X)
\]
The purpose of this section is to prove the following theorem
expressing $\Phi^{G_1 \times G_2}$ in terms of 
$\Phi^{G_1}$ and $\Phi^{G_2}$.
\begin{theorem}[Product theorem]
\label{thm:product-thm}
Suppose $S/B\colon X \hto Y$ is a positive family of h-graph cobordisms,
and let $d_1 = \dim G_1$ and $d_2 = \dim G_2$.
Then the following diagram commutes.
\[\xymatrix@C+4em{
	H_{\ast + (d_1+d_2)\chi(S,X)} (B) \tensor H_\ast(B(G_1 \times G_2)^X)
	\ar[r]^-{\Phi^{G_1 \times G_2}(S/B)}
	&
	H_\ast(B(G_1 \times G_2)^Y)	
	\\
	\ar[u]^-{1\tensor\kappa_X}_-\isom
	{\left[\begin{array}{c}
		H_{\ast + (d_1+d_2)\chi(S,X)} (B) \\
			\tensor \\
		H_\ast(BG_1^X)\tensor H_\ast(BG_2^X)
	\end{array}\right]}
	\ar[d]_-{\Delta_\ast \tensor 1 \tensor 1}
	\\
	{\left[\begin{array}{c}
		H_{\ast + d_1\chi(S\times B,X) + d_2\chi(B \times S,X)} (B\times B) \\
			\tensor \\
		H_\ast(BG_1^X)\tensor H_\ast(BG_2^X)
	\end{array}\right]}
	\\
	\ar@{<-}[u]^-{(\times \tensor 1 \tensor 1)^{-1}}_-\isom
	{\left[\begin{array}{c}
		H_{\ast + d_1\chi(S,X)} (B) \tensor H_{\ast + d_2\chi(S,X)}(B) \\
			\tensor \\
		H_\ast(BG_1^X) \tensor H_\ast(BG_2^X)	
	\end{array}\right]}
	\ar[d]_-{1\tensor \tau \tensor 1}^-\isom
	\\
	{\left[\begin{array}{c}
		H_{\ast + d_1\chi(S,X)} (B) 
			\tensor 
		H_\ast(BG_1^X) \\
			\tensor \\
		H_{\ast + d_2\chi(S,X)}(B) 
			\tensor
		H_\ast(BG_2^X)	
	\end{array}\right]}
	\ar[r]^-{\Phi^{G_1} (S/B) \tensor \Phi^{G_2} (S/B)}
	&
	H_\ast(BG_1^Y) \tensor H_\ast(BG_2^Y)
	\ar[uuuu]_{\kappa_Y}^\isom
}\]
Here $\Delta$ denotes the diagonal map and $\tau$ denotes the twist map.
\end{theorem}

Recall from \cite[Definition~5.2]{HL} that a base space $B$
is called \emph{good} if it 
can be embedded as an open subset of some CW complex.
Let $G$ be a compact Lie group and let $S/B\colon X \hto Y$ be 
a positive family of h-graph cobordisms over a good base space.
In \cite[section 7.6 and 7.8]{HL},
it is shown 
that for an h-graph with basepoints $(Z,R)$, there
is a natural zigzag of homotopy equivalences
\begin{equation}
\label{eq:zigzag}
\xymatrix@C+1.4em{
	B(G^{\Pi_1(Z,R)})
	\ar[r]^{\theta^G_{\Pi_1(Z,R)}}_\homot 
	&
	BG^{B\Pi_1(Z,R)} 
	\ar[r]^{\tilde\alpha^G_{Z,R}}_\homot
	&
	BG^{WQ(Z,R)}
	&
	\ar[l]_-{\tilde\beta^G_{Z,R}}^-\homot
	BG^Z.
}
\end{equation}
Let $\eta^G_{Z,R}$ denote the induced isomorphism
\[
	\eta^G_{Z,R} 
	\colon 
	H_\ast B (G^{\Pi_1(Z,R)}) 
	\xto{\ \isom\ }
	H_\ast (BG^Z).
\]
Unrolling the construction of $\Phi^G(S/B)$ 
given in \cite[sections 6 and 7]{HL},
we obtain the following description of the operation $\Phi^G(S/B)$.
\begin{proposition}
\label{prop:opdesc}
Let $S/B \colon X \hto Y$ be a positive family of h-graph cobordisms
over a good base space.
Choose basepoints $P\to X $ and $Q \to Y$ for $X$ and $Y$, respectively.
Then $\Phi^G(S/B)$ is given by the composite
\begin{equation*}
\vcenter{\xymatrix@1@!0@C=3.4em{
	*!L{H_{\ast+\dim(G)\chi(S,X)} (B) \tensor H_\ast(BG^X)}
	\\
	&
	\ar[rr]^{1\tensor (\eta_{X,P}^G)^{-1}}_\isom
	&&
	*!L{\;H_{\ast+\dim(G)\chi(S,X)} (B) \tensor H_\ast B(G^{\Pi_1(X,P)})}
	\\
	&
	\ar[rr]^\times_\isom
	&&
	*!L{\;H_{\ast+\dim(G)\chi(S,X)} (B\times B(G^{\Pi_1(X,P)}))}
	\\
	&
	\ar[rr]^{!}
	&&
	*!L{\;H_\ast B(G^{\Pi_1(S,P)})}
	\\
	&
	\ar[rr]^{(a)}_\isom
	&&
	*!L{\;H_\ast B(G^{\Pi_1(S,P\sqcup Q)})}
	\\
	&
	\ar[rr]^{(b)}
	&&
	*!L{\;H_\ast B(G^{\Pi_1(Y Q)})}
	\\
	&
	\ar[rr]^{\eta_{Y,Q}^G}_\isom
	&&
	*!L{\;H_\ast ( BG^Y )}
}}
\end{equation*}
where the map labeled by \textup{$!$}
is the umkehr map \cite[section 7.2]{HL} associated to the map 
\begin{equation}
\label{map:umkehr-data}
\vcenter{\xymatrix{
	B \times B(G^{\Pi_1(X,P)})
	\ar[dr]
	&&
	B(G^{\Pi_1(S,P)})
	\ar[ll]
	\ar[dl]
	\\
	&
	B\times BG^P
}}
\end{equation}
of fibrewise manifolds induced by the map $B\times X \to S$;
where the map \textup{(a)} is induced by the inverse of the 
homotopy equivalence
\[
	B(G^{\Pi_1(S,P\sqcup Q)}) 
	\xto{\ \homot\ } 
	B(G^{\Pi_1(S,P)}) 
\]
induced by the inclusion $P\incl P\sqcup Q$;
and where the map \textup{(b)}
is induced by the map $(S\to B,P\sqcup Q) \to (Y\to \pt,Q)$
in the category $\calHbp^\fop$ of 
\cite[subsection 7.3 and Definition 5.10]{HL} given by the diagram
\[\vcenter{\xymatrix{
	(S, P\sqcup Q) 
	\ar[d]
	&
	\ar[l]
	(B\times Y,Q)
	\ar[r]^-{\pr}
	\ar[d]
	&
	(Y,Q)
	\ar[d]
	\\
	B
	\ar@{=}[r]
	&
	B
	\ar[r]
	&
	\pt
}}\]
where the top left map is induced by the inclusions $B\times Y \to S$ and
$Q \to P\sqcup Q$. \qed
\end{proposition}

\begin{proof}[Proof of Theorem \ref{thm:product-thm}]
By picking a CW approximation to $B$ and using the 
base change axiom of HHGFTs, we may assume that $B$ is
a CW complex and hence a good base space. 
A further application of the base change axiom shows that 
without loss of generality we may restrict to the case where
$B$ is connected. This simplifies notation, since in this
case $\chi(S,X)$ is an integer rather than an integer-valued
locally constant function on $B$.
The theorem follows by a lengthy but straightforward 
diagram chase using the description for $\Phi^G(S/B)$
given in Proposition~\ref{prop:opdesc} and the 
following observations:
\begin{itemize}
\item For any family $(T,R)$ of h-graphs with basepoints over 
	a good base space $C$, the projections of $G_1\times G_2$ onto
	$G_1$ and $G_2$
	induce a natural isomorphism
	\[\xymatrix{
		B\big((G_1\times G_2)^{\Pi_1(T,R)} \big)
		\ar[r]_-\isom
		\ar[d]
		&
		B\big(G_1^{\Pi_1(T,R)} \big)
			\times_C
		B\big(G_2^{\Pi_1(T,R)} \big)
		\ar[d]
		\\
		C\times B(G_1\times G_2)^R
		\ar[r]_-\isom
		&
		C \times BG_1^R \times BG_2^R
	}\]
	of fibrewise manifolds.
\item The zigzag \eqref{eq:zigzag} is natural with respect
	to maps of Lie groups. It follows that the diagram
	\[\xymatrix@C+2.5em{
		H_\ast B(G_1\times G_2)^{\Pi_1(X,P)})
		\ar[r]^{\eta^{G_1\times G_2}_{X,P}}_\isom
		\ar[d]^\isom
		&
		H_\ast (B(G_1 \times G_2)^X)
		\ar[d]_\isom
		\\
		H_\ast \big( B(G_1^{\Pi_1(X,P)}) \times B(G_2^{\Pi_1(X,P)}) \big)
		\ar[r]
		&
		H_\ast (BG_1^X\times BG_2^X)
		\\
		\ar[u]^\times_\isom
		H_\ast B(G_1^{\Pi_1(X,P)}) \tensor H_\ast B(G_2^{\Pi_1(X,P)})
		\ar[r]^-{\eta^{G_1}_{X,P} \tensor \eta^{G_2}_{X,P} }_-\isom
		&
		\ar[u]_\times^\isom
		H_\ast (BG_1^X) \tensor H_\ast(BG_2^X)
		\ar `r[u]`[uu]_{\kappa_X}^\isom [uu]
	}\]
	commutes, and similarly for $(Y,Q)$.
	Here the middle horizontal arrow is 
	\[
		(\tilde\beta^{G_1} \times \tilde\beta^{G_2})^{-1}_\ast
			\circ
		(\tilde\alpha^{G_1} \times \tilde\alpha^{G_2})_\ast
			\circ 
		(\theta^{G_1} \times \theta^{G_2})_\ast
	\]
	and the vertical maps in the top square are
	induced by the projections of $G_1\times G_2$ onto $G_1$ and $G_2$.
\item Consider the commutative diagram
	\begin{equation}
	\label{diag:sdsm}
	\vcenter{\xymatrix@!0@C=9em@R=9ex{
		B\times B(G_1^{\Pi_1(X,P)})\times B(G_2^{\Pi_1(X,P)})
		\ar[ddd]
		\ar[dr]
		&&
		B(G_1^{\Pi_1(S,P)})\times_B B(G_2^{\Pi_1(S,P)})
		\ar[ddd]
		\ar[ll]
		\ar[dl]
		\\
		&
		B\times BG_1^P \times BG_2^P
		\ar[d]
		\\
		&
		B\times BG_1^P \times B\times BG_2^P
		\\
		B\times B(G_1^{\Pi_1(X,P)})\times B\times B(G_2^{\Pi_1(X,P)})
		\ar[ur]	
		&&
		\ar[ll]
		B(G_1^{\Pi_1(S,P)})\times B(G_2^{\Pi_1(S,P)})
		\ar[ul]
	}}
	\end{equation}
	where the horizontal arrows are induced by the 
	maps \eqref{map:umkehr-data} of fibrewise manifolds
	for $G=G_1,G_2$;
	the right-hand vertical map is the inclusion;
	and the middle and left-hand vertical maps
	are induced by the diagonal map $B\to B\times B$.
	Decorating the left-hand corners of the diagram with degree shift
	$(d_1+d_2)\chi(S,X)$ and the right-hand corners
	with degree shift $0$, we obtain a 
	2-cell in the double category $\bbS^\ds(\calM)$ of
	\cite[subsection~7.1]{HL}. Applying the umkehr functor 
	$U_\mfld$ \cite[subsection~7.2]{HL}, 
	we therefore obtain a commutative square
	{\small
	\[\xymatrix@C-1.2em{
		H_{\ast + (d_1+d_2)\chi(S,X)}
			\big(B\times B(G_1^{\Pi_1(X,P)})\times B(G_2^{\Pi_1(X,P)})\big)
		\ar[r]^-{!}
		\ar[d]
		&
		H_\ast \big(B(G_1^{\Pi_1(S,P)})\times_B B(G_2^{\Pi_1(S,P)})\big)
		\ar[d]
		\\
		H_{\ast + (d_1+d_2)\chi(S,X)}
			\big(B\times B(G_1^{\Pi_1(X,P)})
					\times 
				 B\times B(G_2^{\Pi_1(X,P)})\big)
		\ar[r]^-{!}
		&
		H_\ast \big(B(G_1^{\Pi_1(S,P)})\times B(G_2^{\Pi_1(S,P)})\big)
	}\]}%
	where the horizontal maps are the umkehr maps associated
	with the respective horizontal maps in \eqref{diag:sdsm}.
\item The square
	{\small
	\[\xymatrix{
		{\left[\begin{array}{c}
	 		H_{\ast+ d_1\chi(S,X)} \big( B \times B(G_1^{\Pi_1(X,P)})\big)\\
				\tensor \\
			H_{\ast+ d_2\chi(S,X)} \big( B \times B(G_2^{\Pi_1(X,P)})\big)
		\end{array}\right]}
		\ar[d]_\times^\isom
		\ar[r]^{!\tensor !}
		&
		H_\ast \big(B(G_1^{\Pi_1(S,P)})\big)
			\tensor
		H_\ast \big(B(G_2^{\Pi_1(S,P)})\big)
		\ar[d]^\times_\isom
		\\		
		H_{\ast + (d_1+d_2)\chi(S,X)} 
		{\left(\begin{array}{c}
			  B \times B(G_1^{\Pi_1(X,P)})\\
				\times\\
			  B \times B(G_2^{\Pi_1(X,P)})
		\end{array}\right)}
		\ar[r]^-{!}
		&
		H_\ast \big(
			B(G_1^{\Pi_1(S,P)})
				\times
			B(G_2^{\Pi_1(S,P)})
		\big)
	}\]}%
	commutes, as it is an instance of the
	monoidality constraint $U_{\mfld,\tensor,1}$.
	Here the top horizontal map is 
	the tensor product of the umkehr maps associated
	with the 
	maps \eqref{map:umkehr-data} of fibrewise manifolds
	for $G=G_1,G_2$, while the bottom horizontal map
	is the umkehr map associated with the 
	direct product of these maps.\qedhere
\end{itemize}
\end{proof}

\section{A vanishing result}
\label{sec:vanishing-result}

Our aim in this section is to prove a vanishing result,
Theorem~\ref{thm:vanishing} below. Combined with 
a sufficient understanding of the homology of 
symmetric groups, this result in particular 
shows that the operations
$\Phi^G(S_n/B\Sigma_n)$ 
vanish whenever $n$ is not a power of 2,
and reduces the computation of the 
operations $\Phi^G(S_{2^k}/B\Sigma_{2^k})$ 
to the computation of certain maps
\[
	\alpha^G_k
	\colon 
	H_{\ast-\dim(G)(2^k-1)}(BV_k)\tensor H_\ast(BG)
	\longto 
	H_\ast(BG)
\]
where $V_k$ is an elementary abelian 2-group 
of dimension $k$.
See Definition~\ref{def:alphagk}.

We start by recalling facts about the homology of symmetric groups.
The standard inclusions 
$\Sigma_n \times \Sigma_m \to \Sigma_{m+n}$ induce on
$\bigsqcup_{n\geq 0} B\Sigma_n$ the structure of a topological monoid,
making the homology groups $H_\ast(\bigsqcup_{n\geq 0} B\Sigma_n)$
into a ring bigraded by degree and \emph{weight}, the homogeneous
elements of degree $d$ and weight $n$ being precisely the elements
of $H_d(B\Sigma_n)$. In fact, $\bigsqcup_{n\geq 0} B\Sigma_n$
is not just a monoid, but an $E_\infty$-space. The following 
description of its homology in terms of the concomitant Dyer--Lashof
operations $Q^s$ can be found in \cite[\S I.4]{IteratedLoopSpaces}.
\begin{theorem}
\label{thm:dldesc}
We have
\[\pushQED{\qed} 
	H_\ast\Big(\bigsqcup_{n\geq 0} B\Sigma_n\Big)
		=
	\F_2\big[Q^I [1]\, |\, I \text{ is admissible and } e(I) > 0\big] . \qedhere
\popQED
\]
\end{theorem}
Here a multi-index $I = (s_1,\ldots, s_k)$ with $s_j \geq 0$ is 
called \emph{admissible}
if $s_j \leq 2s_{j+1}$ for all $1 \leq j < k$;
is said to have \emph{excess}
$e(I) = s_1 - \sum_{j=2}^k s_j$;
and has associated operation 
$Q^I = Q^{s_1} \cdots Q^{s_k}$.
(By convention, 
the empty multi-index $I=\emptyset$
is admissible, has excess $\infty$, and has associated
operation $Q^{\emptyset} = 1$.)
The element $[1]$ is the generator of $H_0(B\Sigma_1)$,
and $Q^I[1]$ for $I = (s_1,\ldots, s_k)$ has degree
$s_1+\cdots+s_k$ and weight $2^k$.

For our purposes, a slightly different description of 
$H_\ast(\bigsqcup_{n\geq 0} B\Sigma_n)$
is more convenient. A choice of
bijection between the product $\{1,\ldots,n\} \times \{1,\ldots, m\}$
and $\{1,2,\ldots,nm\}$ induces a homomorphism 
$\Sigma_n\times \Sigma_m \to \Sigma_{nm}$,
with different choices leading to conjugate homomorphisms.
Thus we obtain a map
\[
	H_\ast(B\Sigma_n)\tensor H_\ast (B\Sigma_m) \longto H_\ast(B\Sigma_{nm})
\]
which does not depend on the bijection chosen.
These maps for varying $n$ and $m$ fit together to equip
$H_\ast(\bigsqcup_{n\geq 0} B\Sigma_n)$
with another commutative associative product which we 
denote by $\circ$. The element $[1]$ is the identity element for 
this product. For each $i\geq 0$, let
$E_i$ denote the non-trivial element of 
$H_i(B\Sigma_2)$. From Theorem~\ref{thm:dldesc} 
and the computation of $H_\ast(QS^0)$ 
given in \cite{Turner}, we can easily read off the 
following result.
\begin{theorem}
\label{thm:turnerdesc}
We have
\[
	H_\ast\Big(\bigsqcup_{n\geq 0} B\Sigma_n\Big) 
		= 
	\F_2[E_{i_1} \circ E_{2i_2}\circ \cdots \circ E_{2^{k-1}i_k} 
		\,|\, k\geq0, 1\leq i_1 \leq \cdots \leq i_k]
\]
where the empty $\circ$-product is to be interpreted as the element
 $[1] \in H_0(B\Sigma_1)$.
\end{theorem}
Notice that the element 
$E_{i_1} \circ E_{2i_2}\circ \cdots \circ E_{2^{k-1}i_k}$
has degree 
$i_1 + 2i_2+\cdots + 2^{k-1}i_k$ 
and weight $2^k$.
\begin{proof}[Proof of Theorem \ref{thm:turnerdesc}]
Let 
$i\colon \bigsqcup_{n\geq 0} B\Sigma_n \to QS^0$
be the group completion map afforded by the Barratt--Quillen--Priddy
theorem, so that 
\[
	i_\ast
	\colon 
	H_\ast\Big(\bigsqcup_{n\geq 0} B\Sigma_n\Big)
	\longto
	H_\ast(QS^0)
\]
is a ring map identifying $H_\ast(QS^0)$ 
with the localization 
$H_\ast\big(\bigsqcup_{n\geq 0} B\Sigma_n\big)\big[[1]^{-1}\big]$.
Since by Theorem~\ref{thm:dldesc} 
multiplication by $[1]$ is injective in 
$H_\ast(\bigsqcup_{n\geq 0} B\Sigma_n)$,
the map $i_\ast$ is injective. 
The map $i_\ast$ sends each element
$E_i \in H_\ast(B\Sigma_2)$
to the element of $H_\ast(QS^0)$ of the same name
defined on p.\ 213 in \cite{Turner}; to see this,
use the factorization of the transfer 
map $\mathrm{tr^{(n)}}$
given on p.\ 212 of \cite{Turner}. Moreover, 
\cite[Theorem 3.10]{MadsenMilgram} implies that
the map $i_\ast$ preserves the $\circ$-product. 
Thus $i_\ast$ takes each element 
\begin{equation}
\label{eq:elts} 
	E_{i_1} \circ E_{2i_2} \circ \cdots \circ E_{2^{k-1}i_k} 
		\in 
	H_\ast\Big(\bigsqcup_{n\geq 0} B\Sigma_n\Big),
	\quad
	1\leq i_1 \leq \cdots \leq i_k
\end{equation}
to the element of the same name in $H_\ast(QS^0)$.
By \cite[Theorem 4.15]{Turner}, these
elements of $H_\ast(QS^0)$ are algebraically 
independent, whence so are the elements \eqref{eq:elts}.
It remains to show that the elements \eqref{eq:elts} generate 
all of $H_\ast(\bigsqcup_{n\geq 0} B\Sigma_n)$.
But this follows from the description of 
$H_\ast(\bigsqcup_{n\geq 0} B\Sigma_n)$ 
given in Theorem~\ref{thm:dldesc}
by comparison of bigraded Poincar\'{e} series
using the observation  that for each $k$,
setting $i_t = s_t - \sum_{j={t+1}}^k s_j$ for $1\leq t \leq k$ 
defines a bijection between
\[
	\{I=(s_1,\ldots,s_k)\,|\, I \text{ admissible}, e(I)>0\}
	\quad
	\text{and}
	\quad
	\{(i_1,\ldots,i_k)\,|\, 1\leq i_1 \leq \cdots \leq i_k \}
\]
with $s_1+\cdots + s_k = i_1 + 2 i_2 + \cdots + 2^{k-1}i_k$.
\end{proof}

\begin{theorem}
\label{thm:vanishing}
Let $G$ be a positive-dimensional compact Lie group or a 
finite group of even order.
Then the map 
\[
	\Phi^G(S_n/B\Sigma_n)^\sharp(a)
	 \colon 
	H_\ast(BG) 
	\longto
	H_{\ast + |a| + \dim(G)(n-1)} (BG)
\]
is zero whenever $a \in H_\ast(B\Sigma_n)$ is decomposable in 
the ring $H_\ast(\bigsqcup_{n\geq0} B\Sigma_n)$.
\end{theorem}

\begin{remark}
The case where $G$ is a finite group of odd order is 
uninteresting, since in that case the homology of $BG$
is concentrated in degree 0.
\end{remark}

Before proving Theorem~\ref{thm:vanishing}, we
need to establish an auxiliary vanishing result.
Let $I \colon \pt \to \pt$ and $\mu \colon \pt\sqcup\pt \to \pt$
and $\delta \colon \pt \to \pt \sqcup \pt$
and $\varepsilon \colon \pt \to \emptyset$
be the h-graph cobordisms depicted below.
\begin{equation}
\label{eq:basic-cobs}
I = 
\begin{tikzpicture}[scale=0.03,baseline=-2.5]
	\path[fill=black] (0,0) circle (2);
	\path[fill=black] (30,0) circle (2);
	\path[ARC] (0,0) -- (30,0);
\end{tikzpicture}
\qquad\qquad
\mu = 
\begin{tikzpicture}[scale=0.03,baseline=-2.5]
	\path[fill=black] (0,12) circle (2);
	\path[fill=black] (0,-12) circle (2);
	\path[fill=black] (40,0) circle (2);
	\path[ARC] (0,12) .. controls (20,12) and (22,5) .. (22,0);
	\path[ARC] (0,-12) .. controls (20,-12) and (22,-5) .. (22,0);
	\path[ARC] (22,0) -- (40,0);
\end{tikzpicture}
\qquad\qquad
\delta=
\begin{tikzpicture}[scale=0.03,baseline=-2.5, xscale=-1]
	\path[fill=black] (0,12) circle (2);
	\path[fill=black] (0,-12) circle (2);
	\path[fill=black] (40,0) circle (2);
	\path[ARC] (0,12) .. controls (20,12) and (22,5) .. (22,0);
	\path[ARC] (0,-12) .. controls (20,-12) and (22,-5) .. (22,0);
	\path[ARC] (22,0) -- (40,0);
\end{tikzpicture}
\qquad\qquad
\varepsilon=
\begin{tikzpicture}[scale=0.03,baseline=-2.5]
	\path[fill=black] (0,0) circle (2);
	\path[ARC] (0,0) -- (10,0);
\end{tikzpicture}
\end{equation}
While in the context of the present discussion our interest in 
the following result is motivated by the role it plays in the
proof of Theorem~\ref{thm:vanishing}, the result is 
of some interest of its own,
since it implies that the product on $H_\ast(BG)$ induced by 
$\mu$ is uninteresting.

\begin{proposition}
\label{prop:vanishing}
Let $G$ be a compact Lie group of positive dimension or a finite
group of even order. Then the operation 
\[
	\Phi^G(\mu/\pt)
	\colon 
	H_{\ast-\dim(G)}(\pt)\tensor H_\ast(BG\times BG) 
		\longto 
	H_\ast(BG)
\]
is zero. 
\end{proposition}
\begin{proof}
The h-graph cobordism $\mu$ decomposes as
$\mu \homot (I \sqcup (\varepsilon \circ\mu) )\circ (\delta \sqcup I)$
as illustrated below. 
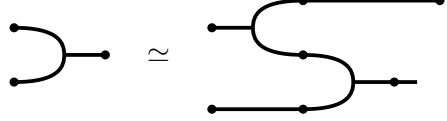
\[
\begin{tikzpicture}[scale=0.03,baseline=-2.5]
	\path[fill=black] (0,12) circle (2);
	\path[fill=black] (0,-12) circle (2);
	\path[fill=black] (40,0) circle (2);
	\path[ARC] (0,12) .. controls (20,12) and (22,5) .. (22,0);
	\path[ARC] (0,-12) .. controls (20,-12) and (22,-5) .. (22,0);
	\path[ARC] (22,0) -- (40,0);
\end{tikzpicture}
\quad
\homot
\quad
\begin{tikzpicture}[scale=0.03,baseline=-2.5]
	\path[fill=black] (0,12) circle (2);
	\path[fill=black] (0,-24) circle (2);
	\path[fill=black] (40,-24) circle (2);
	\path[fill=black] (40,0) circle (2);
	\path[fill=black] (40,24) circle (2);
	\path[fill=black] (80,-12) circle (2);
	\path[fill=black] (100,24) circle (2);
	\begin{scope}[shift={(40,12)},xscale=-1]
	\path[ARC] (0,12) .. controls (20,12) and (22,5) .. (22,0);
	\path[ARC] (0,-12) .. controls (20,-12) and (22,-5) .. (22,0);
	\path[ARC] (22,0) -- (40,0);
	\end{scope}
	\begin{scope}[shift={(40,-12)}]
	\path[ARC] (0,12) .. controls (20,12) and (22,5) .. (22,0);
	\path[ARC] (0,-12) .. controls (20,-12) and (22,-5) .. (22,0);
	\path[ARC] (22,0) -- (40,0);
	\end{scope}
	\path[ARC] (0,-24) -- (40,-24);
	\path[ARC] (40,24) -- (100,24);
	\path[ARC] (80,-12) -- (90,-12);
\end{tikzpicture}
\]
When $G$ is positive dimensional, the
claim now follows from the observation that the operation 
\[
	\Phi^G((\varepsilon\circ\mu)/\pt)
	\colon 
	H_{\ast-\dim(G)}(\pt) \tensor H_\ast(BG\times BG)
	\longto
	H_\ast(\pt)
\]
must vanish for degree reasons. 
Let us assume that $G$ is finite of even order.
Let $P$ denote the two incoming points in $\mu$.
Then by Proposition~\ref{prop:opdesc}, 
the operation $\Phi^G(\mu/\pt)$ factors through
the umkehr map associated with the map
\begin{equation}
\label{eq:proj} 
	B(G^{\Pi_1(\mu,P)}) \longto B(G^{\Pi_1(P,P)})
\end{equation}
of fibrewise manifolds 
(or, in this case, simply covering spaces)
over $BG^P$
induced by the inclusion $P\incl \mu$.
To prove the claim, it is therefore
enough to show that this umkehr map is zero.
By \cite[Lemma 8.6]{HL}, the 
umkehr map in this case 
is just the transfer map.
Through the zigzag \eqref{eq:zigzag},
the map \eqref{eq:proj} 
corresponds to the map $BG^\mu \to BG^P=BG\times BG$
which, up to a homotopy equivalence of the source,
is just the diagonal map $\Delta \colon BG \to BG \times BG$.
By standard properties of the transfer, the composite
\[\xymatrix{
	H_\ast(BG\times BG) 
	\ar[r]^-{\Delta^!} 
	&
	H_\ast(BG)
	\ar[r]^-{\Delta_\ast}
	&
	H_\ast(BG\times BG)
}\]
of the transfer map and the induced map associated with $\Delta$
is multiplication by the index $[G\times G : G] = |G|$,
and hence zero by the assumption that the order of $G$  is even.
Since $\Delta_\ast$ is injective, having a left inverse 
induced by the projection map $\pr_1\colon BG\times BG \to BG$,
it follows that the transfer map $\Delta^!$ is zero,
giving the claim.
\end{proof}

Theorem~\ref{thm:vanishing} now follows easily
from Proposition~\ref{prop:vanishing}.

\begin{proof}[Proof of Theorem~\ref{thm:vanishing}]
Suppose $a\in H_\ast(B\Sigma_n)$ decomposes as a product 
$a = bc$ for some 
$b \in H_\ast(B\Sigma_k)$
and
$c \in H_\ast(B\Sigma_l)$
where $k,l \geq 1$, $k+l = n$.
Then $a = (Bm_{k,l})_\ast(b\times c)$
where $m_{k,l} \colon \Sigma_k\times \Sigma_l \to \Sigma_n$ 
is the standard inclusion. By the base change axiom,
we have
\[
	\Phi^G(S_n/B\Sigma_n)^\sharp(a)
	 = 
	\Phi^G\big((Bm_{k,l})^\ast S_n/(B\Sigma_k \times B\Sigma_l)\big)^\sharp(b\times c).
\]
But $(Bm_{k,l})^\ast S_n/(B\Sigma_k \times B\Sigma_l)$
decomposes (up to a 2-cell \cite[Definition 2.10]{HL}
which is a homeomorphism on base spaces)
as the composite
\[
	(\mu/\pt) 
	\ucirc 
	\big((S_k/B\Sigma_k) \usqcup (S_l/B\Sigma_l)\big)
	\ucirc 
	(\delta/\pt).
\]
Thus the claim follows from Proposition~\ref{prop:vanishing}.
\end{proof}

Combining Theorem~\ref{thm:vanishing} with Theorem~\ref{thm:turnerdesc}
(or Theorem~\ref{thm:dldesc}), we obtain the following result.
\begin{corollary}
Let $G$ be a positive-dimensional compact Lie group or a 
finite group of even order.
Then the operation $\Phi^G(S_n/B\Sigma_n)$ vanishes
unless $n$ is a power of 2.\qed
\end{corollary}

Theorem~\ref{thm:vanishing} and Theorem~\ref{thm:turnerdesc}
also imply that to understand $\Phi^G(S_n/B\Sigma_n)$
for $n=2^k$, it is enough to understand its behaviour on 
classes in the image of the map 
$H_\ast B\Sigma_2^k \to H_\ast B\Sigma_{2^k}$
induced by the map $\iota \colon\Sigma_2^k \to \Sigma_{2^k}$
(well defined up to conjugacy)
obtained by iterating the maps inducing the $\circ$-product
on $H_\ast(\bigsqcup_{n\geq 0} B\Sigma_n)$.
In the following sections, we will therefore focus on 
computing the maps $\alpha_k$ defined below.
Observing that $\Sigma_2^k$ is simply the $k$-dimensional
elementary abelian 2-group, we write $V_k$ for this group
and switch to additive notation for the group operation.
\begin{definition}
\label{def:alphagk}
We define $\alpha^G_k$ to be the composite map
\[\xymatrix@1@!0@C=6em{
	*!R{\alpha^G_k\colon H_{\ast-d(2^k-1)}(BV_k)\tensor H_\ast(BG)\,}
	\ar[r]^{(B\iota)_\ast \tensor 1}
	&
	*!L{\,H_{\ast-d(2^k-1)}(B\Sigma_{2^k}) \tensor H_\ast(BG)}
	\\
	*!R{\phantom{\alpha^G_k\colon H_{\ast-d(2^k-1)}(B\Sigma_2^k)\tensor H_\ast(BG)\,}}
	\ar[r]^{\Phi^G(S_{2^k}/B\Sigma_{2^k})}
	&
	*!L{\,H_\ast(BG)}
}\]
where $d=\dim(G)$. We write $(\alpha^G_k)^\sharp$
for the adjoint map
\[
	(\alpha^G_k)^\sharp 
	\colon 
	H_{\ast-d(2^k-1)}(BV_k)
	\longto 
	\Hom_\ast(H_\ast(BG),H_\ast(BG)).
\]
Notice that we may alternatively describe $\alpha^G_k$ as the operation
\[
	\alpha^G_k = \Phi^G((B\iota)^\ast S_{2^k}/BV_k).
\]
\end{definition}

The identity axiom of HHGFTs 
\cite[subsection~3.1]{HL}
implies the following computation of the map $\alpha^G_0$.
\begin{proposition}
For any compact Lie group $G$, the map $(\alpha_0^G)^\sharp$
sends the generator of $H_\ast(BV_0) \isom \F_2$ to 
the identity map of $H_\ast (BG)$. \qed
\end{proposition}

To conclude the section, let us elaborate
on the homology of $V_k$ and the map $\iota$.
Recall that the homology $H_\ast(BV_k)$ is  
isomorphic as a ring to 
the divided power algebra generated by $V_k$
where $V_k$ is identified with $H_1(BV_k)$.
We write $v^{[n]}$ for the $n$-th divided power
of $v\in H_{>0}(BV_k)$, and recall the formula
$v^{[n]} v^{[m]} = {{n+m}\choose {m}} v^{[n+m]}$
and the convention $v^{[0]}=1$.
Let us write $x_1,\ldots, x_k$ for the basis of 
$V_k$ corresponding to the product decomposition 
$V_k = \Sigma_2^k$. By the choice of $\iota$,
we then have the following 
pleasant formula for the map induced by $\iota$ 
on homology. This formula makes it easy to read off the 
operation $\Phi^G(S_{2^k}/B\Sigma_{2^k})$ from 
formulas describing the map $\alpha^G_k$,
and it is the main motivation for our preference of 
the description of 
$H_\ast\big(\bigsqcup_{n\geq 0} B\Sigma_n\big)$
given in Theorem~\ref{thm:turnerdesc}
over that given in Theorem~\ref{thm:dldesc}.
\begin{proposition}
\label{prop:iotaformula} 
The map $\iota_\ast$ is given by 
\begin{equation*}
\pushQED{\qed} 
	\iota_\ast
	\colon
	H_\ast(BV_k) \longto H_\ast(B\Sigma_{2^k}),
	\quad
	x_{{1}}^{[n_1]} \cdots x_k^{[n_k]} 
	\longmapsto
	E_{n_1} \circ \cdots \circ E_{n_k}.\qedhere
\popQED
\end{equation*}
\end{proposition}
 
\begin{remark}
Observe that the map $\iota$ can alternatively be
understood as the Cayley embedding obtained from some
identification of $V_k$ with the set $\{1,2,3,\ldots,2^k\}$.
Since the composite of a Cayley embedding
$\Gamma \to \Sigma_{|\Gamma|}$  of a group $\Gamma$ 
with an automorphism of $\Gamma$ is conjugate to the 
original embedding, we see that the map $\iota_\ast$
factors through $H_\ast(BV_k)_{GL_k(\F_2)}$.
Thus the formula 
for $\iota_\ast$ 
given in Proposition~\ref{prop:iotaformula}
is in fact valid for \emph{any}
basis $x_1,\ldots,x_k$ of $V_k$.
It can be shown that the map 
$H_\ast(BV_k)_{GL_k(\F_2)} \to H_\ast(B\Sigma_{2^k})$
induced by $\iota_\ast$ 
is injective. See eg.\ \cite[p.\ 61]{MadsenMilgram}.
\end{remark}

\section{Preliminary calculations}
\label{sec:preliminary-calculations}

The main purpose of this section is to develop 
a description of the maps $\alpha^G_k$
amenable to making calculations.
The main results are 
Propositions~\ref{prop:alphagk-comp}
and \ref{prop:compsum}.

For $\Gamma$ a topological group and $X$ 
a space with a $\Gamma$-action, we write
$X \sslash \Gamma$ for the Borel construction $E\Gamma \times_\Gamma X$.
Equivalently, $X \sslash \Gamma$
can be described as the bar construction $B(\pt,\Gamma,X)$
or the classifying space of the action groupoid
of the $\Gamma$-action on $X$.
Recall from \cite[proof of Proposition 7.34]{HL}
that if  $(T,R)$ is an
h-graph with basepoints, then for any compact Lie group $G$
there is a natural homeomorphism 
\begin{equation}
\label{iso:borel-interpretation} 
	B(G^{\Pi_1(T,R)}) \isom \fun(\Pi_1(T,R),G) \sslash G^R
\end{equation}
where $\fun(\Pi_1(T,R),G)$ denotes the space of 
functors from $\Pi_1(T,R)$ to $G$ (with $G$
interpreted
as a one-object topological category) and where the action of $G^R$
on the space $\fun(\Pi_1(T,R),G)$ is given by the formula
\[
	(\delta\cdot f)(\alpha)
	=
	\delta(y)\cdot f(\alpha)\cdot\delta(x)^{-1}
\]
for $f\in \fun(\Pi_1(T,R),G)$, $\delta\in G^R$, and $\alpha\colon x\to y$
a morphism of $\Pi_1(T,R)$. 
If a discrete group $\Gamma$ acts on $T$ in such a fashion that the image of 
the basepoint map $R \to T$ is kept pointwise fixed, we obtain a 
family of h-graphs with basepoints $(E\Gamma\times_\Gamma T, R)$
over $B\Gamma$, and have natural isomorphisms
\begin{multline}
\label{isos:borel-interpretation2}
	B(G^{\Pi_1(E\Gamma\times_\Gamma T, R)}) 
	\isom 
	E\Gamma\times_\Gamma B(G^{\Pi_1(T, R)}) 
	\isom
	E\Gamma\times_\Gamma (\fun(\Pi_1(T,R),G)\sslash G^R)
	\\
	\isom
	\fun(\Pi_1(T,R),G)\sslash \Gamma \times G^R
\end{multline}
of fibrewise manifolds over $B\Gamma \times BG^R$
where the action of $\Gamma \times G^R$ on 
$\fun(\Pi_1(T,R),G)$ is given by
\[
	((\gamma,\delta)\cdot f)(\alpha)
	=
	\delta(y)\cdot f(\gamma_\ast^{-1}(\alpha))\cdot\delta(x)^{-1}
\]
for $f\in \fun(\Pi_1(T,R),G)$, $\gamma \in \Gamma$, $\delta\in G^R$,
and $\alpha\colon x\to y$ a morphism of $\Pi_1(T,R)$. 

\begin{definition}
\label{def:families-from-borel-construction}
Suppose $\hat S\colon X \hto Y$ is an h-graph cobordism 
equipped with an action of a discrete group $\Gamma$ in which  
$X$ and $Y$ are kept pointwise fixed. 
We say that a family of h-graph cobordisms $S/B\Gamma \colon X\hto Y$
\emph{is obtained from} $\hat S$ \emph{by the Borel construction}
if $S/B\Gamma$ is equipped with a homotopy equivalence 
$S \to E\Gamma \times_\Gamma \hat S$ which is a map
over $B\Gamma$ and under $(X \sqcup Y) \times B\Gamma$.
Here $E\Gamma \times_\Gamma \hat S$ is made into a 
space under $(X \sqcup Y) \times B\Gamma$ by the map 
$(X \sqcup Y) \times B\Gamma \isom E\Gamma \times_\Gamma (X\sqcup Y)
\to E\Gamma \times_\Gamma \hat S$.
\end{definition}

\begin{remark}
\label{rk:families-from-borel-construction}
The motivation for Definition~\ref{def:families-from-borel-construction}
comes from the fact that it is not always clear that the map 
$(X \sqcup Y) \times B\Gamma \to E\Gamma \times_\Gamma \hat S$
is a closed fibrewise cofibration, as would be required for
$E\Gamma \times_\Gamma \hat S$ itself to give a family
of h-graph cobordisms from $X$ to $Y$ over $B\Gamma$.
However, this problem can always be fixed simply by replacing 
$E\Gamma \times_\Gamma \hat S$ by the mapping cylinder
$(E\Gamma \times_\Gamma \hat S)'$ of the 
map $(X \sqcup Y) \times B\Gamma \to E\Gamma \times_\Gamma \hat S$.
Then $(E\Gamma \times_\Gamma \hat S)'/B\Gamma\colon X \hto Y$ 
is obtained from 
$\hat S$ by the Borel construction in the above sense.
(The map $(E\Gamma \times_\Gamma \hat S)'\to B\Gamma$
is a fibration by \cite[Proposition 1.3]{Clapp}.)
Moreover, it is easy to show that any family 
$S/B\Gamma \colon X \hto Y$
obtained from $\hat S$ by the Borel construction
is homotopy equivalent over $B\Gamma$ and under 
$(X \sqcup Y) \times B\Gamma$ to 
$(E\Gamma \times_\Gamma \hat S)'$.

In the case of the families
$S_n = E\Sigma_n \times_{\Sigma_n} \hat S_n$ over $B\Sigma_n$,
the map
$(\pt \sqcup \pt) \times B\Sigma_n 
\to E\Sigma_n \times_{\Sigma_n} \hat S_n$
\emph{is} a closed fibrewise cofibration over $B\Sigma_n$,
as can be seen using the criterion 
\cite[Lemma~5.2.4]{MaySigurdsson}. 
Thus there is no need to replace 
$E\Sigma_n \times_{\Sigma_n} \hat S_n$ 
by a mapping cylinder in this case.
\end{remark}

Using the isomorphisms \eqref{isos:borel-interpretation2},
we obtain the following reinterpretation of 
Proposition~\ref{prop:opdesc} in the case where the 
family of h-graph cobordisms $S/B$ is obtained by 
the Borel construction from 
an h-graph cobordism equipped with a group action.

\begin{proposition}
\label{prop:borel-opdesc}
Suppose a discrete group $\Gamma$ acts 
on a positive h-graph cobordism $\hat S \colon X \hto Y$
in such a way that $X$ and $Y$ are kept
pointwise fixed,
and suppose $S/B\Gamma \colon X\hto Y$ is obtained
from $\hat S$ by the Borel construction.
Choose basepoints $P\to X$ and $Q\to Y$
for $X$ and $Y$, respectively. Then for any compact Lie group $G$,
the operation $\Phi^G(S/B\Gamma)$ agrees with the composite
\begin{equation*}
\vcenter{\xymatrix@1@!0@C=3.4em{
	*!L{H_{\ast+\dim(G)\chi(\hat S,X)} (B\Gamma) \tensor H_\ast(BG^X)}
	\\
	&
	\ar[rr]^{1\tensor (\tilde\eta_{X,P}^G)^{-1}}_\isom
	&&
	*!L{\;H_{\ast+\dim(G)\chi(\hat S,X)} (\pt\sslash\Gamma) \tensor H_\ast (\fun(\Pi_1(X,P),G)\sslash G^P)}
	\\
	&
	\ar[rr]^\times_\isom
	&&
	*!L{\;H_{\ast+\dim(G)\chi(\hat S,X)} \big((\pt\sslash\Gamma) \times (\fun(\Pi_1(X,P),G)\sslash G^P)\big) }
	\\
	&
	\ar[rr]^{(a)}_\isom
	&&
	*!L{\;H_{\ast+\dim(G)\chi(\hat S,X)} (\fun(\Pi_1(X,P),G)\sslash \Gamma \times G^P) }
	\\
	&
	\ar[rr]^{!}
	&&
	*!L{\;H_\ast ( \fun(\Pi_1(\hat S,P),G)\sslash \Gamma \times G^P) }
	\\
	&
	\ar[rr]^{(b)}_\isom
	&&
	*!L{\;H_\ast ( \fun(\Pi_1(\hat S,P\sqcup Q),G)\sslash \Gamma \times G^{P\sqcup Q}) }
	\\
	&
	\ar[rr]^{(c)}
	&&
	*!L{\;H_\ast ( \fun(\Pi_1(Y,Q),G)\sslash G^{Q}) }
	\\
	&
	\ar[rr]^{\tilde\eta_{Y,Q}^G}_\isom
	&&
	*!L{\;H_\ast ( BG^Y )}
}}
\end{equation*}
where 
$\tilde\eta_{X,P}^G$ and $\tilde\eta_{Y,Q}^G$ 
are obtained from the maps
$\eta_{X,P}^G$ and $\eta_{Y,Q}^G$ 
of section~\ref{sec:product-thm}
and the homeomorphism \eqref{iso:borel-interpretation};
where the map \textup{(a)} is induced by the homeomorphism
\[
	(\pt\sslash\Gamma) \times (\fun(\Pi_1(X,P),G)\sslash G^P)
 	\isom 
	\fun(\Pi_1(X,P),G)\sslash \Gamma \times G^P;
\]
where the map labeled by \textup{$!$}
is the umkehr map \cite[section 7.2]{HL} associated with the map 
\begin{equation}
\label{map:to-umkehr}
\vcenter{\xymatrix@C-2em{
	\fun(\Pi_1(X,P),G)\sslash \Gamma \times G^P
	\ar[dr]
	&&
	\fun(\Pi_1(\hat S,P),G)\sslash \Gamma \times G^P
	\ar[ll]
	\ar[dl]
	\\
	&
	\pt \sslash \Gamma \times G^P
}}
\end{equation}
of fibrewise manifolds induced by the inclusion $X \incl \hat S$;
where the map \textup{(b)} is induced by the inverse of the 
homotopy equivalence
\begin{equation}
\label{map:pqeq}
	\fun(\Pi_1(\hat S,P\sqcup Q),G) \sslash \Gamma \times G^{P\sqcup Q} 
	\xto{\ \homot\ } 
	\fun(\Pi_1(\hat S,P),G) \sslash \Gamma \times G^P
\end{equation}
induced by the inclusion $P\incl P\sqcup Q$;
and where the map \textup{(c)}
is induced by the map $(Y,Q) \to (\hat S, P\sqcup Q)$
of h-graphs with basepoints and the 
projection $\Gamma \times G^{P\sqcup Q}\to G^Q$.
\qed
\end{proposition}

\begin{remark}
\label{rk:stabilizer}
For any $f \in \fun(\Pi_1(\hat S,P\sqcup Q),G)$,
the projection map
$\Gamma \times G^{P\sqcup Q} \to \Gamma \times G^P$
maps the stabilizer of $f$ in the 
$\Gamma \times G^{P\sqcup Q}$-action
injectively into $\Gamma \times G^P$.
One way to see this is as follows.
Choose for each $q \in Q$ 
a morphism $\alpha_q$ in $\Pi_1(\hat S, P\sqcup Q)$
from an element of $P$ to $q$, and observe
that a functor $\Pi_1(\hat S,P\sqcup Q) \to G$
is then precisely the same data as a functor 
$\Pi_1(\hat S,P) \to G$ and the assignment of an
element of $G$ to each $\alpha_q$, $q\in Q$.
Thus we get a homeomorphism
\begin{equation}
\label{isom:qsplit}
	\fun(\Pi_1(\hat S,P\sqcup Q),G) \isom \fun(\Pi_1(\hat S,P),G)\times G^Q.
\end{equation}
Under this homeomorphism, the restriction of the 
$\Gamma \times G^{P\sqcup Q}$-action to $G^Q$ on 
the left hand side 
corresponds to multiplication action on the
$G^Q$-factor on the right hand side, showing that the 
$G^Q$-action on $\fun(\Pi_1(\hat S,P\sqcup Q),G)$
is free.
\end{remark}

Let us now specialize to the computation of 
the operations
$\alpha^G_k$ of Definition~\ref{def:alphagk}.
The family of h-graph cobordisms 
$(B\iota)^\ast S_{2^k}/BV_k \colon \pt \hto \pt$
inducing the map $\alpha^G_k$ is 
isomorphic to the family
$EV_k \times_{V_k} \hat S_{2^k}/BV_k \colon \pt \hto \pt$
obtained by the Borel construction from 
the h-graph cobordism $\hat S_{2^k}\colon \pt \hto \pt$;
the action of $V_k$ on $\hat S_{2^k}$ is the one
obtained by restricting the $\Sigma_{2^k}$-action 
on $\hat S_{2^k}$ along the inclusion
$\iota\colon V_k \to \Sigma_{2^k}$.
Let $p$ and $q$ denote the incoming and 
outgoing points of $\hat S_{2^k}$, 
respectively, and recall that we may 
interpret $\iota \colon V_k\to \Sigma_{2^k}$
as the Cayley embedding associated to
some bijection between $V_k$ and $\{1,2,3,\ldots,2^k\}$.
For $v\in V_k$, let $s_v$ be the path from $q$ to $p$ that traces
the string of $\hat S_{2^k}$ corresponding to $v$ under this bijection.
Then the homotopy classes of the paths $s_v$ determine a
basis for the finite free groupoid $\Pi_1(\hat S_{2^k},\{p,q\})$,
and we obtain a homeomorphism
\begin{equation}
\label{isom:fungv}
	\fun(\Pi_1(\hat S_{2^k},\{p,q\}),G) \xto{\ \isom\ } G^{V_k},
	\quad 
	f \longmapsto \big(f([s_v])\big)_{v\in V_k}
 \end{equation}
under which the $V_k \times G^{\{p,q\}}$-action on 
$\fun(\Pi_1(\hat S_{2^k},\{p,q\}),G)$ corresponds to the 
$V_k \times G^{\{p,q\}}$-action on $G^{V_k}$
given by 
\begin{equation}
\label{eq:action}
	(u,g_p,g_q) \cdot (g_v)_{v\in V_k}  
	= 
	(g_p^{} g_{u+v}^{} g_q^{-1})_{v\in V_k}.
\end{equation}
From the homeomorphism \eqref{isom:qsplit} of 
Remark~\ref{rk:stabilizer}, we can further deduce that 
\eqref{isom:fungv} induces a homeomorphism
\[
	\fun(\Pi_1(\hat S_{2^k},\{p\}),G) \xto{\ \isom\ } G^{V_k}/\Delta G
\]
where $G^{V_k}/\Delta G$ denotes the space of left cosets of 
the diagonal subgroup $\Delta G$ of the product group $G^{V_k}$.
Under this homeomorphism, the 
the $V_k \times G^{\{p\}}$-action on the space
$\fun(\Pi_1(\hat S_{2^k},\{p\}),G)$ corresponds to the 
$V_k \times G^{\{p\}}$-action on $G^{V_k}/\Delta G$
given by 
\begin{equation}
\label{eq:action2}
	(u,g_p) \cdot (g_v)_{v\in V_k} \Delta G 
	= 
	(g_p g_{u+v})_{v\in V_k} \Delta G.
\end{equation}
Proposition~\ref{prop:borel-opdesc} now implies that 
we may compute 
$\alpha^G_k$ by a push-pull construction
in the diagram
\begin{equation}
\label{diag:push-pull-for-alphagk}
\vcenter{\xymatrix@!0@C=3.5em@R=8ex{
	&
	G^{V_k}/\Delta G \sslash V_k \times G^{\{p\}}
	\ar[dl]_{!}
	&&&
	G^{V_k}\sslash V_k \times G^{\{p,q\}}
	\ar[lll]_-\homot
	\ar[dr]
	\\
	\pt \sslash V_k \times G^{\{p\}}
	&&&&&
	\pt \sslash G^{\{q\}}
}} 
\end{equation}
where the arrows are induced by the evident
quotient maps of spaces and projection homomorphisms
of groups. More precisely, we have the following 
result.
\begin{proposition}
\label{prop:alphagk-comp}
Let $G$ be a compact Lie group. Then the map $\alpha^G_k$ 
equals the composite
\newcommand{\entry}{H_\ast(BV_k) \tensor H_\ast(BG)\;}
\begin{equation*}
\vcenter{\xymatrix@1@!0@C=4em{
	*!R{\entry}
	\ar[r]^\times_\isom
	&
	*!L{\;H_\ast (BV_k \times BG) }
	\\
	*!R{\phantom{\entry}}
	\ar[r]^{(a)}_\isom
	&
	*!L{\;H_\ast (\pt \sslash V_k \times G^{\{p\}} ) }
	\\
	*!R{\phantom{\entry}}
	\ar[r]^{!}
	&
	*!L{\;H_\ast (G^{V_k}/\Delta G \sslash V_k \times G^{\{p\}}) }
	\\
	*!R{\phantom{\entry}}
	\ar[r]^{(b)}_\isom
	&
	*!L{\;H_\ast (G^{V_k} \sslash V_k \times G^{\{p,q\}}) }
	\\
	*!R{\phantom{\entry}}
	\ar[r]^{(c)}
	&
	*!L{\;H_\ast ( \pt \sslash G^{\{q\}}) }
	\\
	*!R{\phantom{\entry}}
	\ar[r]^{(d)}_\isom
	&
	*!L{\;H_\ast ( BG ) }
}}
\end{equation*}
where \textup{(a)} and \textup{(d)} are
induced by the evident homeomorphisms
\[
	BV_k\times BG \isom \pt \sslash V_k\times G^{\{p\}}
	\qquad\text{and}\qquad 
	\pt \sslash G^{\{q\}} \isom BG,
\]
respectively; 
where the map labeled by \textup{$!$} is the umkehr
map \cite[section 7.2]{HL} associated to the map
\textup{$!$} of diagram \eqref{diag:push-pull-for-alphagk}
considered as a map of fibrewise manifolds over 
$\pt\sslash V_k \times G^{\{p\}}$;
where the map \textup{(b)} is induced by 
the homotopy inverse of the horizontal map in 
\eqref{diag:push-pull-for-alphagk};
and where the map \textup{(c)} is 
induced by the right-hand diagonal map 
in \eqref{diag:push-pull-for-alphagk}.
\qed
\end{proposition}

In the case of a finite group $G$, 
we may compute the composite of the maps $!$, (b) and (c)
of  Proposition~\ref{prop:alphagk-comp} by decomposing
$G^{V_k}$ into $V_k\times G^{\{p,q\}}$-orbits.
\begin{proposition}
\label{prop:compsum}
Let $G$ be a finite group, and let $O \subset G^{V_k}$
be a set of orbit representatives
for the $V_k\times G^{\{p,q\}}$-action on $G^{V_k}$.
Then the composite of the maps~\textup{$!$}, \textup{(b)} and \textup{(c)}
in Proposition~\ref{prop:alphagk-comp}
is equal to the sum over all
$\bar g \in O$ of the composite maps
\[\xymatrix{
	H_\ast (\pt \sslash V_k \times G^{\{p\}})
	\ar[r]^-{!}
	&
	H_\ast (\pt \sslash (V_k \times G^{\{p,q\}})_{\bar g}) 
	\ar[r]
	&
	H_\ast ( \pt \sslash G^{\{q\}})
}\]
where the first map is the transfer map associated with the map
\[
	\pt \sslash (V_k \times G^{\{p,q\}})_{\bar g} 
	\longto  
	\pt \sslash V_k \times G^{\{p\}}
\]
induced by the inclusion of the stabilizer 
$(V_k \times G^{\{p,q\}})_{\bar g}$ into 
$V_k \times G^{\{p,q\}}$ and the projection  
$V_k \times G^{\{p,q\}} \to V_k \times G^{\{p\}}$;
and where the second map is induced by the 
map
\[
	\pt \sslash (V_k \times G^{\{p,q\}})_{\bar g} 
	\longto  
	\pt \sslash G^{\{q\}}
\]
given by the composite of the inclusion  
$(V_k \times G^{\{p,q\}})_{\bar g}\incl V_k \times G^{\{p,q\}}$ 
and the projection 
$V_k \times G^{\{p,q\}} \to G^{\{q\}}$.
\end{proposition}
\begin{remark}
\label{rk:injectivity}
From Remark~\ref{rk:stabilizer} or a direct calculation
using \eqref{eq:action}, we know that the 
composite map
\[
	(V_k \times G^{\{p,q\}})_{\bar g}
	\longincl 
	V_k \times G^{\{p,q\}}
	\xto{\ \pr\ }
	V_k \times G^{\{p\}}
\]
is injective for every $\bar g \in G^{V_k}$,
and hence admits a transfer map.
\end{remark}
\begin{proof}[Proof of Proposition~\ref{prop:compsum}]
When $G$ is finite,
the map labeled by \textup{$!$} in 
Proposition~\ref{prop:alphagk-comp}
is just the transfer map; see \cite[Lemma 8.6]{HL}.
By the homotopy invariance of transfers, in this case
the composite of the maps \textup{$!$} and (b) in 
Proposition~\ref{prop:alphagk-comp} is 
simply the transfer map associated with the composite
\[
	G^{V_k} \sslash V_k \times G^{\{p,q\}}
	\longto 
	\pt\sslash V_k \times G^{\{p\}}
\]
of the left-hand diagonal map and the horizontal map 
of \eqref{diag:push-pull-for-alphagk}.
Now observe that for any group $\Gamma$ and transitive 
$\Gamma$-set $X$ and element $x\in X$, the 
map 
\[
	\pt\sslash \Gamma_x \longto X \sslash \Gamma
\]
given by the map $\pt \mapsto x$ and the inclusion 
$\Gamma_x \incl \Gamma$ is a homotopy equivalence.
Indeed, it suffices to consider the case where
$X = \Gamma/H$ for a subgroup $H \leq \Gamma$ and 
$x = eH \in \Gamma/H$. It follows that we have a homotopy
equivalence
\[
	\bigsqcup_{\bar g \in O} \pt \sslash (V_k\times G^{\{p,q\}})_{\bar g}
	\xto{\ \homot\ }
	G^{V_k} \sslash	V_k\times G^{\{p,q\}}
\]
which on the summand corresponding to $\bar g\in O$ is 
given by the map $\pt \mapsto \bar g$ and 
the inclusion  
$(V_k\times G^{\{p,q\}})_{\bar g} \incl V_k\times G^{\{p,q\}}$.
The claim now follows from standard properties of transfer maps.
\end{proof}

We conclude the section with the following simple observation 
which dramatically reduces the number of orbits we need to take into
account in our applications of Proposition~\ref{prop:compsum}.

\begin{lemma}
\label{lm:trivial-transfers}
Let $\Gamma$ be a group and let $H \leq \Gamma$ be an 
even-index subgroup such that the 
map $i_\ast \colon H_\ast(\pt\sslash H) \to H_\ast(\pt\sslash\Gamma)$
induced by the inclusion $i\colon H\incl \Gamma$ 
is injective. Then the transfer map
$i^! \colon H_\ast(\pt\sslash \Gamma) \to H_\ast(\pt\sslash H)$
is zero.
\end{lemma}
\begin{proof}
By standard properties of the transfer map, the composite
$i_\ast i^!$ is multiplication by the index of $H$ in $\Gamma$,
and hence zero as we are working with $\F_2$ coefficients. 
The claim now follows from the injectivity of $i_\ast$.
\end{proof}

\section{Computations for elementary abelian 2-groups and dihedral groups}
\label{sec:e-and-dcalc}

The purpose of this section is to compute the maps $\alpha_k^G$
when $G$ is an elementary abelian 2-group $E$ or a dihedral group 
\[
	D_{4n+2} = \langle r,s \,|\, r^{2n+1}= s^2 = e, sr = r^{-1}s\rangle.
\]
We allow the case $n=0$, in which case $D_{4n+2}$ reduces to a 
copy of $\Z/2$.
Recall that the quotient homomorphism
$D_{4n+2} \to D_{4n+2}/\langle r \rangle \isom \Z/2$
induces an isomorphism on mod 2 homology,
as can be seen for example by considering the 
Hochschild--Serre spectral sequence of the 
short exact sequence
\[
	1 \longto \langle r \rangle \longto D_{4n+2} \longto \Z/2 \longto 1.
\]
It follows that the inclusion $\langle s \rangle \incl D_{4n+2}$
also induces an isomorphism on mod 2 homology,
since the composite $\Z/2 \isom \langle s \rangle \incl D_{4n+2}$
is a right inverse to the quotient map $D_{4n+2} \to \Z/2$.
Our aim is to prove the following two theorems, 
the second one of which reduces the computation of $\alpha_k^{D_{4n+2}}$
to the first.

\begin{theorem}
\label{thm:ecalc}
Let $E$ be an elementary abelian 2-group. Then the map 
\[
	\alpha^{E}_k \colon H_\ast(BV_k) \tensor H_\ast(BE)
	\longto 
	H_\ast(BE)
\]
is given by 
\[
	a \tensor b \longmapsto \sum_{K} K_\ast(a)b
\]
where the sum is over all homomorphisms $K\colon V_k \to E$.
\end{theorem}

\begin{theorem}
\label{thm:dcalc}
The diagram
\[\xymatrix@C+2em{
	H_\ast(BV_k) \tensor H_\ast (BD_{4n+2})
	\ar[r]^-{\alpha^{D_{4n+2}}_k}
	&
	H_\ast (BD_{4n+2})
	\\
	H_\ast(BV_k) \tensor H_\ast (B\langle s \rangle)
	\ar[r]^-{\alpha^{\langle s \rangle }_k }
	\ar[u]^\isom
	&
	H_\ast (B\langle s \rangle)
	\ar[u]_\isom
}\]
where the vertical maps are induced by the 
inclusion $\langle s\rangle \incl D_{4n+2}$ commutes
for all $k\geq 1$.
\end{theorem}

To prove Theorems~\ref{thm:ecalc} and \ref{thm:dcalc},
we will use Propositions~\ref{prop:alphagk-comp}
and \ref{prop:compsum},
making use of Lemma~\ref{lm:trivial-transfers}
to limit the number of orbits we need to
take into account when applying 
Proposition~\ref{prop:compsum}.
When computing $\alpha_k^{D_{4n+2}}$,
Lemma~\ref{lm:trivial-transfers} 
together with the following lemma show that 
we may restrict attention to the orbits
of those elements $\bar g \in D_{4n+2}^{V_k}$
with the property that the image of 
the stabilizer $(V_k \times D_{4n+2}^{\{p,q\}})_{\bar g}$
in $V_k \times D_{4n+2}^{\{p\}}$ is an 
odd-index subgroup.

\begin{lemma}
\label{lm:subgroup-inj}
Let $H$  be a subgroup of $V_k \times D_{4n+2}$.
Then the inclusion of $H$ into $V_k\times D_{4n+2}$
induces an injection on homology.
\end{lemma}
\begin{proof}
Let $W = H\cap V_k$, and let $Q=H/W$.
We then have the commutative diagram 
\[\xymatrix{
	1
	\ar[r]
	&
	W
	\ar[r]
	\ar[d]_{i'}
	&
	H
	\ar[r]
	\ar[d]_{i}
	&
	Q
	\ar[r]
	\ar[d]^{i''}
	&
	1
	\\
	1
	\ar[r]
	&
	V_k
	\ar[r]
	&
	V_k \times D_{4n+2}
	\ar[r]
	&
	D_{4n+2}
	\ar[r]
	&
	1
}\]
where the rows are exact and 
$i$ and $i'$ are the inclusions and
$i''$ is the map induced by $i$.
Notice that $i''$ is also 
injective. There results a map between the
Hochschild--Serre spectral sequences associated 
with the two rows. On the $E^2$-page, this map is 
given by the tensor product
\[
	i''_\ast \tensor i'_\ast 
	\colon 
	H_\ast(BQ) \tensor H_\ast(BW)
	\longto
	H_\ast(BD_{4n+2}) \tensor H_\ast(BV_k).
\]
Recalling every subgroup of $D_{4n+2}$ is 
cyclic of the form $\langle r^d\rangle $
for some $d|(2n+1)$ or dihedral of the form
$\langle r^d, r^i s \rangle$ for some 
$d | (2n+1)$, $0\leq i < 2n+1$,
and observing that in each case the
map induced by the inclusion of the 
subgroup into $D_{4n+2}$ is injective
on homology, we see that 
the map 
$i''_\ast \colon H_\ast(BQ) \to H_\ast(BD_{4n+2})$
is injective.
Since the map $i'$ admits a left inverse,
the map $i'_\ast \colon H_\ast(BW) \to H_\ast(BV_k)$
is injective as well. 
Thus the map between the two
spectral sequences 
is injective on the $E^2$-page.
Since the spectral sequence of the 
second row collapses on the $E^2$-page,
it follows that  the spectral sequence
of the first row also does. Thus
the map between spectral sequences
is also injective on the $E^\infty$-page,
which suffices to imply the claim.
\end{proof}

The next two lemmas provide an analysis of the 
orbits containing an element
$\bar g \in D_{4n+2}^{V_k}$ with the 
aforementioned property as well as the
associated stabilizers.

\begin{lemma}
\label{lm:homom-stabs}
Suppose $K\colon V_k \to \langle s \rangle$ 
is a homomorphism. Then the stabilizer of the
element $(K(v))_{v\in V_k} \in D_{4n+2}^{V_k}$
in the $V_k \times D_{4n+2}^{\{p,q\}}$-action on $D_{4n+2}^{V_k}$
is
\[
	\big\{(u,g_p,g_q)\in V_k \times D_{4n+2}^{\{p,q\}}
		\,\big|\, g_q=g_p \big\} 
\]
if $K$ is the trivial homomorphism and 
\[
	\big\{(u,g_p,g_q)\in V_k \times D_{4n+2}^{\{p,q\}}
		\,\big|\, g_p \in \langle s\rangle, \,g_q=g_pK(u)\big\}
\]
otherwise.
\end{lemma}
\begin{proof}
In both cases, it is trivial to check that the 
subgroup given is contained in the stabilizer.
To prove the reverse containment, 
suppose $(u,g_p,g_q)$ is in the 
stabilizer of $(K(v))_{v\in V_k}$.
Then we have
\begin{equation}
\label{eq:stabcond}
	g_p K(u+v) g_q^{-1} = K(v) 
\end{equation}
for all $v$. Setting $v=0$, we see that $g_q = g_p K(u)$,
which in the case of trivial $K$ reduces to 
the equation $g_q = g_q$.
Substituting $g_q = g_p K(u)$ back to \eqref{eq:stabcond},
we see that we must in addition have $g_p K(v) g_p^{-1} = K(v)$
for all $v$. If $K$ is trivial, this condition is 
satisfied for all $g_p \in D_{4n+2}$. On the other hand,
if $K$ is non-trivial, that is, $s$ is in the image of $K$,
then $g_p$ must belong to the centralizer of $s$ in $D_{4n+2}$,
which is $\langle s\rangle$. The claim follows.
\end{proof}

\begin{lemma}
\label{lm:orbit-reps}
The map
\begin{equation}
\label{map:k-to-orbit}
	(K\colon V_k\to \langle s \rangle)
	\longmapsto
	\textup{the orbit of $(K(v))_{v \in V_k} \in D_{4n+2}^{V_k}$}
\end{equation}
gives a bijection from the set of homomorphisms 
$K\colon V_k \to \langle s \rangle$
onto the set of orbits of the $(V_k\times D_{4n+2}^{\{p,q\}})$-action
on $D_{4n+2}^{V_k}$ having the property that
for some (and hence every) element $\bar g$ in the orbit, 
the image of the stabilizer $(V_k\times D_{4n+2}^{\{p,q\}})_{\bar g}$
under the projection 
$V_k\times D_{4n+2}^{\{p,q\}} \to V_k\times D_{4n+2}^{\{p\}}$
is an odd-index subgroup of $V_k\times D_{4n+2}^{\{p\}}$.
\end{lemma}
\begin{proof}
Lemma~\ref{lm:homom-stabs} implies that the map 
\eqref{map:k-to-orbit} takes values in the 
claimed subset of orbits.
To see that the map is injective,
suppose $K_1,K_2\colon V_k \to \langle s \rangle$
are two homomorphisms such that 
$(u,g_p,g_q)\cdot (K_1(v))_{v\in V_k} = (K_2(v))_{v\in V_k}$
for some $(u,g_p,g_q)\in V_k\times D_{4n+2}^{\{p,q\}}$.
Then 
\begin{equation}
\label{eq:k1k2}
	g_p K_1(u+v) g_q^{-1} = K_2(v)
\end{equation}
for all $v\in V_k$. Taking $v=0$ gives $g_q = g_p K_1(u)$,
and substituting this back to \eqref{eq:k1k2}, we see that
$K_2(v) = g_p K_1(v) g_p^{-1}$ for all $v\in V_k$. 
In particular, $K_1(v)$ and $K_2(v)$ are non-trivial
for precisely same $v \in V_k$. Since $K_1$ and $K_2$ 
both take values in the group $\langle s \rangle$ which has
only one non-trivial element, it follows that $K_1 = K_2$.

It remains to show that the map \eqref{map:k-to-orbit} 
is surjective.
Suppose $\bar g = (g_v)_{v\in V_k} \in D_{4n+2}^{V_k}$ 
is an element such that the image of 
the stabilizer of $\bar g$ is an odd-index subgroup 
$H$ of $V_k \times D_{4n+2}^{\{p\}}$.
By Remark~\ref{rk:injectivity},
the stabilizer of $\bar g$ is of the form
\[
	\big\{ (u,g_p,\varphi(u,g_p)) \in V_k \times D_{4n+2}^{\{p,q\}}
		 \,\big|\, (u,g_p) \in H \big\}
\]
for some homomorphism $\varphi \colon H \to D_{4n+2}^{\{q\}}$.
Since $H$ has odd index in $V_k \times D_{4n+2}^{\{p\}}$,
we must have $V_k \subset H$. Thus for every $u\in V_k$,
the element $(u,e,\varphi(u,e))$ is in the stabilizer of 
$\bar g$, so that $g_{u+v} \varphi(u,e)^{-1} = g_v$
for all $v\in V_k$. Taking $v=0$, we see that
\[
	g_u = g_0 \varphi(u,e)
\]
for all $u\in V_k$. The image of $V_k$ under the homomorphism
$\varphi$ is a 2-subgroup of $D_{4n+2}$,
and hence conjugate by some $\gamma \in D_{4n+2}$ 
to a subgroup of $\langle s \rangle$. Now
\[
	(0,\gamma g_0^{-1},\gamma) \cdot \bar g 
	= 
	(\gamma g_0^{-1} g^{}_v \gamma^{-1})_{v \in V_k}
	= 
	(\gamma g_0^{-1} g^{}_0\varphi(v,e) \gamma^{-1})_{v \in V_k}
	=
	(\gamma \varphi(v,e)\gamma^{-1})_{v \in V_k}	
\]
is a representative of the desired form for the orbit of $\bar g$.
\end{proof}

We are now ready to prove Theorems~\ref{thm:ecalc} and \ref{thm:dcalc}.

\begin{proof}[Proof of Theorems~\ref{thm:ecalc} and \ref{thm:dcalc}]
Our strategy is to prove Theorem~\ref{thm:dcalc}
simultaneously with the case $E=\Z/2$ of Theorem~\ref{thm:ecalc},
and then use Theorem~\ref{thm:product-thm} to complete
the proof of Theorem~\ref{thm:ecalc}.

By Proposition~\ref{prop:alphagk-comp}
together with Proposition~\ref{prop:compsum}
and Lemmas~\ref{lm:trivial-transfers} and 
\ref{lm:subgroup-inj} and \ref{lm:orbit-reps},
the map $\alpha^{D_{4n+2}}_k$ equals the composite
\newcommand{\entry}{H_\ast(BV_k) \tensor H_\ast(BD_{4n+2})\;}
\[\vcenter{\xymatrix@!0@C=4em{
	*!R{\entry}
	\ar[r]^\times_\isom
	&
	*!L{\;H_\ast (BV_k \times BD_{4n+2}) }
	\\
	*!R{\phantom{\entry}}
	\ar[r]^{(a)}_\isom
	&
	*!L{\;H_\ast (\pt \sslash V_k \times D_{4n+2}^{\{p\}} ) }
	\\
	*!R{\phantom{\entry}}
	\ar[r]^{(b)}
	&
	*!L{\;H_\ast ( \pt \sslash D_{4n+2}^{\{q\}}) }
	\\
	*!R{\phantom{\entry}}
	\ar[r]^{(c)}_\isom
	&
	*!L{\;H_\ast ( BD_{4n+2} ) }
}}\]
where (a) and (c) are induced by the evident homeomorphisms
and where (b) is the sum of the composite maps
\begin{equation}
\label{eq:dcomp-composite}
\xymatrix@C+1em{
	H_\ast (\pt \sslash V_k \times D_{4n+2}^{\{p\}})
	\ar[r]^-{i_K^!}
	&
	H_\ast (\pt \sslash H_K ) 
	\ar[r]^-{(p_K)_\ast}
	&
	H_\ast ( \pt \sslash D_{4n+2}^{\{q\}}) 
} 
\end{equation}
as $K$ runs through all homomorphisms $V_k \to \langle s \rangle$.
Here $H_K \leq V_k \times D_{4n+2}^{\{p,q\}}$ 
is the stabilizer of
$(K(v))_{v \in V_k} \in D_{4n+2}^{V_k}$ computed in 
Lemma~\ref{lm:homom-stabs},
the first map is the transfer map associated to the 
injective homomorphism 
\[
	i_K\colon H_K \longto V_k \times D_{4n+2}^{\{p\}}, 
	\quad 
	(v,\varepsilon_p,\varepsilon_q) \longmapsto (v,\varepsilon_p),
\]
and the second map is induced by the map
\[
	p_K \colon H_K \longto D_{4n+2}^{\{q\}},
	\quad
	(v,\varepsilon_p,\varepsilon_q) \longmapsto \varepsilon_q.
\]
For every $K\colon V_k \to \langle s \rangle$,
the maps \eqref{eq:dcomp-composite} fit into a commutative diagram
\begin{equation}
\label{diag:pf-of-dcalc}
\vcenter{\xymatrix@+1em{
	H_\ast (\pt \sslash V_k \times D_{4n+2}^{\{p\}})
	\ar[r]^-{i^!_K}
	&
	H_\ast (\pt \sslash H_K ) 
	\ar[r]^-{(p_K)_\ast}
	&
	H_\ast ( \pt \sslash D_{4n+2}^{\{q\}}) 
	\\
	H_\ast(\pt \sslash V_k\times \langle s \rangle)
	\ar[rr]^{(q_K)_\ast}
	\ar[u]^{(1\times j)_\ast}_\isom
	\ar[ur]_{d_\ast}
	&
	&
	H_\ast(\pt\sslash \langle s\rangle)
	\ar[u]_{j_\ast}^\isom
}}
\end{equation}
where the vertical maps are induced by the inclusion 
$j\colon \langle s \rangle \incl D_{4n+2}$,
the diagonal arrow is induced by the map
\[
	d\colon V_K\times \langle s \rangle \longto H_K,
	\quad 
	(v,g) \longmapsto (v,g,gK(v)),
\]
and the bottom horizontal arrow is induced by the map
\[
	q_K\colon V_K\times \langle s \rangle \longto \langle s \rangle,
	\quad 
	(v,g) \longmapsto gK(v).
\]
Indeed, the trapezoid in \eqref{diag:pf-of-dcalc}
commutes since $p_K d = j q_K$ on the level of 
homomorphisms. Furthermore, the maps $i_K^!$ and 
\[
	(i_K)_\ast 
	\colon 
	H_\ast(\pt \sslash H_K)
	\longto
	H_\ast(\pt \sslash V_k \times D_{4n+2}^{\{p\}})
\]
are inverse isomorphisms, as follows from
the observation that the composite map 
$(i_K)_\ast \circ i_K^!$ is the identity, being equal
to multiplication by the index of $i_K H_K$ in 
$V_k \times D_{4n+2}^{\{p\}}$, and 
Lemma~\ref{lm:subgroup-inj}, which implies that 
the map $(i_K)_\ast$ is injective.
Thus the commutativity of the triangle in
\eqref{diag:pf-of-dcalc} follows from the 
identity $i_K \circ d = 1\times j$.
We conclude that the map $\alpha_k^{D_{4n+2}}$
fits into the commutative diagram
\begin{equation}
\label{diag:alphakdred}
\vcenter{\xymatrix@C+2em{
	H_\ast(BV_k) \tensor H_\ast (BD_{4n+2})
	\ar[r]^-{\alpha^{D_{4n+2}}_k}
	&
	H_\ast (BD_{4n+2})
	\\
	H_\ast(BV_k) \tensor H_\ast (B\langle s \rangle)
	\ar[r]
	\ar[u]^{1\tensor j_\ast}_\isom
	&
	H_\ast (B\langle s \rangle)
	\ar[u]_{j_\ast}^\isom
}}
\end{equation}
where the bottom horizontal map is given by 
$a\tensor b \mapsto \sum_K K_\ast(a)b$,
where the sum is over all homomorphisms 
$K\colon V_k \to \langle s \rangle$.

Specializing to the case $n=0$, we 
have obtained a proof of Theorem~\ref{thm:ecalc}
in the special case $E=\Z/2$.
In particular, the bottom horizontal arrow in
\eqref{diag:alphakdred}
is equal to $\alpha_k^{\langle s\rangle}$,
proving Theorem~\ref{thm:dcalc}.
Finally, Theorem~\ref{thm:ecalc} 
in the general case follows at once from 
the special case $E=\Z/2$ and 
Theorem~\ref{thm:product-thm}.
\end{proof}

In practice, it is useful to have a more explicit formula 
for the maps $\alpha^E_k$
than the one given in Theorem~\ref{thm:ecalc}.
Let $x_1,\ldots,x_k$ be a basis for
$V_k$ and let $x$ be the generator of $\Z/2$.
Then we have
\begin{equation}
\label{eq:alphaek-formula}
\begin{aligned}
   	\alpha^{E}_k \big((x_1^{[n_1]} \cdots {}& x_k^{[n_k]}) \tensor b\big)
	\\
	=&\, 
	\sum_K K_\ast \big(x_1^{[n_1]} \cdots x_k^{[n_k]}\big)b
	\\
	=&\,
	\sum_K K_\ast \big(x_1^{[n_1]}\big) \cdots K_\ast\big(x_k^{[n_k]}\big)b
	\\
	=&\,
	\sum_K K_\ast(i_1)_\ast \big(x^{[n_1]}\big) 
		\cdots K_\ast(i_k)_\ast \big(x^{[n_k]}\big) b
	\\
	=&\,
	\bigg(\sum_{L\colon \Z/2 \to E} L_\ast\big(x^{[n_1]}\big)\bigg) 
		\cdots
	\bigg(\sum_{L\colon \Z/2 \to E} L_\ast\big(x^{[n_k]}\big)\bigg) b
	\\
	=&\, 	\begin{cases}
 			\big(\sum_{\varepsilon \in E} \varepsilon^{[n_1]}\big)
			 \cdots 
			\big(\sum_{\varepsilon \in E} \varepsilon^{[n_k]}\big) b
			&
			\text{if $n_1,\ldots, n_k >0$}
			\\
			0 & \text{otherwise}
		\end{cases}
\end{aligned}
\end{equation}
Here $i_j \colon \Z/2 \to V_k$ is the map sending $x$ to $x_j$,
and the indices $L$ run over all homomorphisms $\Z/2 \to E$.
In particular, in the case $E=\Z/2$ we get the following 
formula.
\[
	\alpha^{\Z/2}_k \big((x_1^{[n_1]} \cdots x_k^{[n_k]}) \tensor b\big)
	= 
	\begin{cases}
		x^{[n_1]} \cdots x^{[n_k]} b
		&
		\text{if $n_1,\ldots, n_k >0$}
		\\
		0 & \text{otherwise}
	\end{cases}
\]
Recalling that $x^{[n]} x^{[m]} = {{n+m}\choose {m}} x^{[n+m]}$
in $H_\ast B(\Z/2)$ and that 
${{n+m}\choose {m}}$ 
is odd if the binary expansions of $n$ and $m$ have no 1's in common
and even otherwise, we obtain the following result.
\begin{proposition}
\label{prop:z2-nontrivial-ops}
The operation 
\[
	\big(\alpha^{\Z/2}_k\big)^\sharp \big( x_1^{[n_1]} \cdots x_k^{[n_k]}\big)
	\colon 
	H_{\ast} B(\Z/2) 
	\longto 
	H_{\ast+n_1+\cdots+n_k} B(\Z/2)
\]
is non-trivial precisely when the numbers $n_1,\ldots,n_k$ are positive
and no two of them have a 1 in common in their binary expansions.
Moreover, in this case the operation amounts to multiplication by 
$x^{[n_1+\cdots+n_k]}$.\qed
\end{proposition}

Let us now consider the case of a general elementary abelian 2-group $E$
of dimension $l\geq 1$.
Choose a basis $t_1,\ldots,t_l$ for $E$. Then 
\begin{equation}
\label{eq:divpowsum}
	\sum_{\varepsilon \in E} \varepsilon^{[n]}
	=
	\sum_{\substack{
			i_1,\ldots, i_l \geq 1 
			\\
			i_1 + \cdots + i_l = n}}
		t_1^{[i_1]} \cdots t_l^{[i_l]}
\end{equation}
for every $n>0$ as follows by induction 
from the equation
\[
	\sum_{\varepsilon \in E} \varepsilon^{[n]} 
	=
	\sum_{\varepsilon \in \langle t_1,\ldots,t_{l-1}\rangle }
		\big(\varepsilon^{[n]}+ (\varepsilon+t_l)^{[n]}\big)
	=
	\sum_{\varepsilon \in \langle t_1,\ldots,t_{l-1}\rangle }
		\sum_{i=1}^{n-1} \varepsilon^{[n-i]}_{\phantom{l}}t_l^{[i]}
\]	
where $\langle t_1,\ldots,t_{l-1}\rangle$ denotes the span
of $t_1,\ldots,t_l$ and we have assumed that $l \geq 2$.
Substituting \eqref{eq:divpowsum} into 
\eqref{eq:alphaek-formula}, we see
that for $n_1,\ldots,n_k >0$, we have
\[
\begin{aligned}
   	\alpha^{E}_k \big((x_1^{[n_1]} \cdots {}& x_k^{[n_k]}) \tensor b\big)
	\\
	&=
	\Big(\sum_{\substack{
			i_{11},\ldots, i_{1l} \geq 1 
			\\
			i_{11} + \cdots + i_{1l} = n_1}}
		t_1^{[i_{11}]} \cdots t_l^{[i_{1l}]}
	\Big) 
	\cdots
	\Big(\sum_{\substack{
			i_{k1},\ldots, i_{kl} \geq 1 
			\\
			i_{k1} + \cdots + i_{kl} = n_k}}
		t_1^{[i_{k1}]} \cdots t_l^{[i_{kl}]}
	\Big) b
	\\
	&= 
	\sum_I \big(t_1^{[i_{11}]}\cdots t_1^{[i_{k1}]}\big) 
			\cdots
		   \big(t_l^{[i_{1l}]}\cdots t_l^{[i_{kl}]}\big) b
	\\
	&= 
	\sum_{e_1,\ldots,e_l} A(n_1,\ldots,n_k; e_1,\ldots,e_l) 
			t_1^{[e_1]}\cdots t_l^{[e_l]} b.
\end{aligned}
\]
Here the index $I$ runs over all $(k\times l)$-matrices 
$I=(i_{cd})$
of positive integers with row sums $n_1,\ldots,n_k$, and 
$A(n_1,\ldots,n_k; e_1,\ldots,e_l)$ denotes the 
number of such matrices $I$ having column sums $e_1,\ldots, e_l$
and satisfying the additional property that no two numbers 
in the same column have a common 1 in their binary expansions.
It is useful to think of 
$A(n_1,\ldots,n_k; e_1,\ldots,e_l)$ as counting the 
number of ways to distribute the powers of 2 occurring 
in the binary expansions of $e_1,\ldots,e_l$ among the 
rows of a $(k\times l)$-matrix in such a way that the
result has the prescribed row sums and every entry of the
matrix is positive.

For the operation 
\[
	\big(\alpha^E_k\big)^\sharp\big(x_1^{[n_1]} \cdots  x_k^{[n_k]}\big)
	\colon
	H_\ast(BE)
	\longto
	H_{\ast + n_1+\cdots +n_k} (BE)
\]
to be non-trivial, we must clearly have $n_i \geq l$ for 
all $i$.  Moreover, all the $n_i$'s have to be different,
since the square of any positive-degree element in 
$H_\ast (BE)$ is trivial. 
The problem of determining precisely when the operation
$\big(\alpha^E_k\big)^\sharp \big(x_1^{[n_1]} \cdots  x_k^{[n_k]}\big)$
is non-zero appears complicated, and we will not 
attempt a complete solution here. Instead, we 
content ourselves with observing that it is
easy to find examples where the operation
is non-zero. 
Coupled with the observation that 
$\big(\alpha_1\big)^\sharp \big(x_1^{[n]}\big)$ 
is non-trivial for all $n\geq l$,
the following result suffices to 
construct infinitely many examples of non-trivial
operations for every $k\geq 1$.

\begin{proposition}
\label{prop:non-trivial-ops-e}
Suppose $\big(\alpha^E_k\big)^\sharp\big(x_1^{[n_1]} \cdots  x_k^{[n_k]}\big)$
is non-trivial and let $n_{k+1}$ be a number of the form
$n_{k+1}=2^s r$ where $r\geq l$ and 
$2^s > n_1 +\cdots +n_k$.
Then the operation
$\big(\alpha^E_{k+1}\big)^\sharp \big(x_1^{[n_1]}
	 \cdots  
x_{k+1}^{[n_{k+1}]}\big)$
is non-trivial. 
\end{proposition}
\begin{proof}
Since $\big(\alpha^E_k\big)^\sharp_k\big(x_1^{[n_1]} \cdots  x_k^{[n_k]}\big)$
is non-trivial, we can find exponents $e_1,\ldots, e_l$
such that $A(n_1,\ldots,n_k; e_1,\ldots,e_l)$ is odd.
By the assumption on $n_{k+1}$, we can write it as
a sum $n_{k+1} = r_1 +\cdots + r_l$ where
each $r_j$ is some positive multiple of $2^s$.
Let $I$ be a matrix of the type counted by 
$A(n_1,\ldots,n_{k+1}; e_1+r_1,\ldots,e_l+r_l)$.
Since $r_j$ is divisible by $2^s$ and 
\[
	2^{s} > n_1+\cdots+n_k = e_1+\cdots+e_l \geq e_j
\]
for all $j$, we see that for each $j$,
the set of powers of 2 occurring in the 
binary expansion of 
$e_j + r_j$ 
is the disjoint union of 
those occurring in the binary expansions of 
$e_j$ and $r_j$.
Since the powers of 2 coming from $r_j$ 
are all greater or equal to $2^s$ and $2^s > n_1+\cdots +n_k \geq n_i$
for $1\leq i \leq k$, the powers of 2 coming 
from the $r_j$'s
must all have been assigned to the last row of $I$,
for otherwise the row sum of some other row of $I$ 
would be too large.
On the other hand, no powers of 2 occurring 
in the binary expansions of the $e_j$'s 
may have been assigned to the last row,
for then the row sum of the last row of $I$ 
would exceed $n_{k+1} = r_1 + \cdots + r_l$. 
Thus forgetting the last row provides a bijection
from the set of matrices counted by 
$A(n_1,\ldots,n_{k+1}; e_1+r_1,\ldots,e_l+r_l)$
to the set of matrices counted by 
$A(n_1,\ldots,n_k; e_1,\ldots,e_l)$.
Therefore
$A(n_1,\ldots,n_{k+1}; e_1+r_1,\ldots,e_l+r_l) 
=
A(n_1,\ldots,n_k; e_1,\ldots,e_l)$,
and the claim follows.
\end{proof}

We also have the following result, which is 
an immediate corollary of the observation that
$A(n_1,\ldots,n_k; e_1,\ldots,e_l) 
	= 
A(2n_1,\ldots,2n_k; 2e_1,\ldots,2e_l)$
for all $n_1,\ldots,n_k$ and $e_1,\ldots,e_l$.	
\begin{proposition}
If the operation
$\big(\alpha^E_k\big)^\sharp\big(x_1^{[n_1]} \cdots  x_k^{[n_k]}\big)$
is non-trivial, so is the operation
$\big(\alpha^E_k\big)^\sharp\big(x_1^{[2n_1]} \cdots  x_k^{[2n_k]}\big)$. 
\qed
\end{proposition}

\section{Computations for tori}
\label{sec:tcalc}

The purpose of this section is to compute the 
operations $\alpha^G_1$ and $\alpha^G_2$
when $G$ is a torus $\T^l = (S^1)^l$.
Recall that $H_\ast(B\T^1) \isom H_\ast (\C P^\infty)$ is a 
ring, the divided power algebra on a single generator of
degree 2. The inclusion $\beta$ of $V_1$ into $\T^1$ as $\pm 1$ 
induces a ring homomorphism 
$\beta_\ast \colon H_\ast BV_1 \to H_\ast B\T^1$
which is zero in odd degrees and an isomorphism
in even degrees, as can be seen for example by considering
the Serre spectral sequence of the fibre sequence 
$\T^1 \to BV_1 \to B\T^1$. 
Write $x$ for the generator of $V_1$, and 
let $x_1$ and $x_2$ form a basis of $V_2$.
We will prove the following results.
\begin{theorem}
\label{thm:tcalc}
The map 
\[
	\alpha^{\T^1}_1 
	\colon
	H_{\ast-1}( BV_1) \tensor H_\ast(B\T^1) \longto H_\ast (B\T^1)
\]
is given by 
\[
	x^{[n]} \tensor b \longmapsto \beta_\ast(x^{[n+1]}) b.
\]
\end{theorem}

\begin{theorem}
\label{thm:tcalc2}
The map 
\[
	\alpha^{\T^1}_2
	\colon
	H_{\ast-3}( BV_2) \tensor H_\ast(B\T^1) \longto H_\ast (B\T^1)
\]
is given by 
\[
	x_1^{[n_1]}x_2^{[n_2]} \tensor b 
	\longmapsto 
	\left(1+{{n_1+n_2+2}\choose{n_1+1}}\right)
	\beta_\ast(x^{[n_1+n_2+3]}) b.
\]
\end{theorem}

Remembering that the maps
$H_\ast(BV_1) \to H_\ast(BV_1 \times BV_1)$
and 
$H_\ast(BV_2) \to H_\ast(BV_2 \times BV_2)$
induced by the diagonal maps are given by 
\[
	x^{[n]} \longmapsto \sum_{i=0}^n (x^{[i]} \times x^{[n-i]})
	\quad
	\text{and}
	\quad
	x_1^{[n_1]} x_2^{[n_2]} 
	\longmapsto 
	\sum_{i=0}^{n_1} \sum_{j=0}^{n_2}
	(x_1^{[i]}x_2^{[j]} \times x_1^{[n_1-i]}x_2^{[n_2-j]}),
\]
respectively,
Theorem~\ref{thm:product-thm} allows us to 
read off the following following formulas for $\alpha^{\T^l}_1$
and $\alpha^{\T^l}_2$
from Theorems~\ref{thm:tcalc} and \ref{thm:tcalc2}.
Observe that $H_\ast (B\T^l) \isom H_\ast (B\T^1)^{\tensor l}$ 
is also a ring,
the divided power algebra on $l$ generators of degree 2.
\begin{theorem}
The map 
\[
	\alpha^{\T^l}_1 
	\colon
	H_{\ast-l}( BV_1) \tensor H_\ast(B\T^l) \longto H_\ast (B\T^l)
\]
is given by 
\[
	x^{[n]} \tensor b
	\longmapsto\;
	\sum_{\mathclap{
		\substack{i_1,\ldots,i_l \geq 0 \\ i_1+\cdots+i_l = n} }}
		\;\big(\beta_\ast(x^{[i_1+1]})
			\times \cdots \times
		 \beta_\ast(x^{[i_l+1]}) \big)b
\]
while the map 
\[
	\alpha^{\T^l}_2 
	\colon
	H_{\ast-3l}( BV_2) \tensor H_\ast(B\T^l) \longto H_\ast (B\T^l)
\]
is given by 
\[
	x_1^{[n_1]}x_2^{[n_2]} \tensor b
	\mapsto 
	\sum\prod_{a=1}^l \left(1+{{i_a+j_a+2}\choose {i_a+1}}\right)
		\big(\beta_\ast(x^{[i_1+j_1+3]}) 
			\times \cdots \times
		 \beta_\ast(x^{[i_l+j_l+3]}) \big)b
\]
where the sum is over all non-negative integers
$i_1,\ldots,i_l$ and $j_1,\ldots,j_l$
such that $i_1+\cdots+i_l = n_1$ and $j_1+\cdots+j_l = n_2$.
\qed
\end{theorem}
\noindent
In particular, the operation 
$\big(\alpha^{\T^l}_1\big)^\sharp \big(x^{[n]}\big)$
is non-trivial precisely when $n=2k-l$ for some $k\geq l$.

\begin{proof}[Proof of Theorem~\ref{thm:tcalc}]
Let us interpret $V_1 = \{\pm 1\}$ and write the group 
operation multiplicatively.
From Proposition~\ref{prop:alphagk-comp} we deduce that
we may compute $\alpha^{\T^1}_1$ by a push-pull construction
in the diagram
\begin{equation}
\label{diag:push-pull-for-alphat1}
\vcenter{\xymatrix@!0@C=3.5em@R=9.3ex{
	&
	\T^1 \sslash V_1 \times (\T^1)^{\{p\}}
	\ar[dl]_{!}^{\pi_1}
	&&&
	(\T^1)^{V_1}\sslash V_1 \times (\T^1)^{\{p,q\}}
	\ar[lll]^-\homot_-\eta
	\ar[dr]_{\pi_2}
	\\
	\pt \sslash V_1 \times (\T^1)^{\{p\}}
	&&&&&
	\pt \sslash (\T^1)^{\{q\}}
}} 
\end{equation}
where the $V_1\times (\T^1)^{\{p\}}$-action on $\T^1$ is given by
\[
	(\varepsilon, z_p) \cdot z = z^\varepsilon,
\]
where the $V_1\times (\T^1)^{\{p,q\}}$-action on $(\T^1)^{V_1}$
is given by 
\[
	(\varepsilon,z_p,z_q) \cdot (z_{-1},z_{1}) 
	= 
	(z_p^{} z_q^{-1} z_{-\varepsilon}^{},
	 z_p^{} z_q^{-1} z_\varepsilon^{}),
\]
where $\eta$ is induced by the map
\[
	(\T^1)^{V_1} \longto \T^1,
	\quad
	(z_{-1},z_1) \longmapsto z_{-1}^{} z_1^{-1}
\]
and the projection
$V_1\times (\T^1)^{\{p,q\}} \to V_1\times (\T^1)^{\{p\}}$,
and where $\pi_1$ and $\pi_2$ are induced by the 
evident projections of spaces and groups.
Diagram~\eqref{diag:push-pull-for-alphat1}
is simply diagram~\eqref{diag:push-pull-for-alphagk}
with $G = \T^1$, $k=1$ and the space $(\T^1)^{V_1}/\Delta \T^1$
replaced with $\T^1$ using the homeomorphism
\[
	(\T^1)^{V_1}/\Delta \T^1 \xto{\ \isom\ } \T^1,
	\quad
	(z_{-1},z_1)\Delta \T^1 \longmapsto z_{-1}^{} z_1^{-1}.
\]

Observe that the source of $\pi_1$ splits as a product
\[
	\T^1 \sslash V_1 \times (\T^1)^{\{p\}} 
	\isom 
	(\T^1 \sslash V_1) \times (\pt \sslash (\T^1)^{\{p\}} )
\]
where the $V_1$-action on $\T^1$ is given by
$\varepsilon \cdot z = z^{\varepsilon}$,
while the target of $\pi_1$ splits as a product
\[
	\pt \sslash V_1 \times (\T^1)^{\{p\}}
	\isom 
	(\pt \sslash V_1) \times (\pt\sslash (\T^1)^{\{p\}}).
\]
Under these homeomorphisms, the map $\pi_1$
corresponds to the map
\[
	\pi_1'\times 1 
	\colon 
	(\T^1 \sslash V_1) \times (\pt \sslash (\T^1)^{\{p\}} )
	\longto
	(\pt \sslash V_1) \times (\pt\sslash (\T^1)^{\{p\}}),
\]
where $\pi_1'$ is induced by the projection $\T^1 \to \pt$.
By the compatibility of umkehr maps with direct products,
under the isomorphisms induced by the above homeomorphisms 
and homology cross products,
the umkehr map $\pi_1^!$ corresponds to the map 
\[
	(\pi_1')^! \tensor 1
	\colon 
	H_\ast(\pt\sslash V_1) \tensor H_\ast(\pt\sslash (\T^1)^{\{p\}})
	\longto
	H_{\ast+1}(\T^1\sslash V_1) \tensor H_\ast(\pt\sslash (\T^1)^{\{p\}}).
\]

By \cite[Lemma~8.4]{HL}, we may compute the map
$(\pi_1')^!$ using the Serre spectral sequence 
for the fibration $\pi_1' \colon \T^1\sslash V_1 \to \pt \sslash V_1$.
More precisely, $(\pi_1')^!$ is given by the composite
\begin{equation}
\label{comp:pidash-umkehr}
	H_n(\pt\sslash V_1) 
	\xto{\ \isom\ }
	H_n(\pt\sslash V_1;\,H_1 \T^1)
	=
	E^2_{n,1}
	\longtwoheadto
	E^\infty_{n,1}
	\longincl
	H_{n+1} (\T^1\sslash V_1)	
\end{equation}
where $E^2$ and $E^\infty$ refer to pages in the Serre
spectral sequence of $\pi_1'$, the first map 
is induced by the fundamental class of $\T^1$,
and the last two maps arise from the fact
that we are working with the top row of the
spectral sequence.

The assignments $\pt \mapsto \pm 1\in \T^1$ define sections
$s_{\pm 1} \colon \pt\sslash V_1 \to \T^1\sslash V_1$
of $\pi_1'$.
Since the spectral sequence of $\pi_1'$ has only two 
rows, the existence of the sections $s_{\pm 1}$
implies that the spectral sequence collapses on the 
$E^2$-page. Thus the short exact sequence
\begin{equation*}
\xymatrix{
	0
	\ar[r]
	&
	E^\infty_{n,1}
	\ar[r]
	&
	H_{n+1} (\T^1\sslash V_1)
	\ar[r]
	&
	E^\infty_{n+1,0}
	\ar[r]
	&
	0
}
\end{equation*}
together with the description
\eqref{comp:pidash-umkehr}
for $(\pi_1')^!$ give a short exact sequence
\begin{equation}
\label{ses:pidash}
\xymatrix{
	0
	\ar[r]
	&
	H_n(\pt\sslash V_1)
	\ar[r]^-{(\pi_1')^!}
	&
	H_{n+1} (\T^1\sslash V_1)
	\ar[r]^-{\pi'_\ast}
	&
	H_{n+1} (\pt \sslash V_1)
	\ar[r]
	&
	0.
}
\end{equation}

Let 
$U_{-1} = \T^1 \setminus \{1\}$ and 
$U_{1} = \T^1 \setminus \{-1\}$,
and observe that $V_1$ acts freely on the intersection 
$U_{-1} \cap U_1$ and that $U_1$ and $U_{-1}$ admit
$V_1$-equivariant deformation retractions onto $\{1\} \subset \T^1$
and $\{-1\} \subset \T^1$, respectively.
From the Mayer--Vietoris sequence for 
$U_{-1}\sslash V_1$ and $U_1\sslash V_1$ we now conclude that the map
\[
	(s_{-1})_\ast + (s_1)_\ast 
	\colon H_n (\pt \sslash V_1) \oplus H_n (\pt \sslash V_1)
	\longto
	H_n(\T^1 \sslash V_1)
\]
is an isomorphism for $n>0$.
Since the maps $s_{\pm 1}$ are sections of $\pi'$, from 
the short exact sequence \eqref{ses:pidash}
we can deduce that under this isomorphism,
the map $(\pi_1')^!$ corresponds to the map sending  
$x^{[n]} \in H_n(\pt\sslash V_1)$
to the diagonal element
$(x^{[n+1]},x^{[n+1]}) 
\in H_{n+1}(\pt\sslash V_1)\oplus H_{n+1}(\pt\sslash V_1)$.

Observe now that the two maps
\[
	\tilde s_1, \tilde s_{-1}
	\colon 
	\pt \sslash V_1 \times (\T^1)^{\{p\}}
	\longto 
	(\T^1)^{V_1}\sslash V_1 \times (\T^1)^{\{p,q\}}
\]
defined by the assignments 
$\pt \mapsto (1,1),(1,-1)\in (\T^1)^{V_1}$
and the homomorphisms
\[
	 V_1 \times (\T^1)^{\{p\}} \longto V_1 \times (\T^1)^{\{p,q\}},
	 \quad
	 (\varepsilon,z_p) \longmapsto (\varepsilon,z_p,z_p),
	 								(\varepsilon,z_p,\varepsilon z_p),
\]
provide lifts of the respective maps
\[
	s_1 \times 1, s_{-1} \times 1
	\colon
	(\pt \sslash V_1) \times (\pt \sslash (\T^1)^{\{p\}})
	\longto
	(\T^1 \sslash V_1) \times (\pt \sslash (\T^1)^{\{p\}})
\]
through $\eta$ in the sense that the diagram
\[\xymatrix{
	\pt\sslash V_1 \times \pt \sslash (\T^1)^{\{p\}}
	\ar[d]_{s_{\pm 1} \times 1}
	&
	\pt \sslash  V_1 \times (\T^1)^{\{p\}}
	\ar[l]_-{\isom}
	\ar[dr]^{\tilde s_{\pm 1}}
	\\
	\T^1 \sslash V_1 \times \pt \sslash (\T^1)^{\{p\}}
	&
	\T^1 \sslash  V_1 \times (\T^1)^{\{p\}}
	\ar[l]_-{\isom}
	&
	(\T^1)^{V_1}\sslash V_1 \times (\T^1)^{\{p,q\}}
	\ar[l]_-\eta^-\homot
}\]
commutes for both choices of the sign.
The composite 
\[
	\pi_2 \circ \tilde s_{\pm 1}
	\colon \
	\pt \sslash V_1 \times (\T^1)^{\{p\}}
	\longto
	\pt\sslash (\T^1)^{\{q\}}
\]
is the map induced by the projection 
$\pr\colon V_1 \times \T^1 \to \T^1$
in the case of $\tilde s_1$ and the map 
induced by the composite 
\[\xymatrix{
	V_1 \times \T^1
	\ar[r]^-{\beta \times 1}
	&
	\T^1 \times \T^1
	\ar[r]^-{\mu}
	&
	\T^1
}\]
in the case of $\tilde s_{-1}$.
Here $\mu$ is the multiplication map of $\T^1$.
From the above description of the umkehr map $(\pi_1')^!$
we now deduce that the map $\alpha^{\T^1}_1$ is 
the sum of the two maps
$H_\ast(BV_1) \tensor H_\ast(B\T^1)	
 \to
H_{\ast+1}(B\T^1)$
given by 
\[
	x^{[n]} \tensor b \longmapsto \pr_\ast( x^{[n+1]} \times b)
	\quad\text{and}\quad
	x^{[n]} \tensor b 
	\longmapsto 
	\mu_\ast(\beta\times 1)_\ast (x^{[n+1]} \times b).
\]
The claim follows.
\end{proof}

The rest of this section is dedicated to the proof of 
Theorem~\ref{thm:tcalc2}.
Again, from Proposition~\ref{prop:alphagk-comp} we deduce that
we may compute $\alpha^{\T^1}_2$ by a push-pull construction
in the following special case of 
diagram~\eqref{diag:push-pull-for-alphagk}:
\begin{equation}
\label{diag:push-pull-for-alphat1-2}
\vcenter{\xymatrix@!0@C=3.5em@R=9.3ex{
	&
	(\T^1)^{V_2}/\Delta \T^1 \sslash V_2 \times (\T^1)^{\{p\}}
	\ar[dl]_{!}^{\pi_1}
	&&&&
	(\T^1)^{V_2}\sslash V_2 \times (\T^1)^{\{p,q\}}
	\ar[llll]^-\homot_-\kappa
	\ar[dr]_{\pi_2}
	\\
	\pt \sslash V_2 \times (\T^1)^{\{p\}}
	&&&&&&
	\pt \sslash (\T^1)^{\{q\}}
}} 
\end{equation}
From now on, for the sake of brevity,
let us write $T^3$ for the space 
$(\T^1)^{V_2}/\Delta \T^1$ 
(which is a 3-dimensional torus).
The map $\pi_1$ fits into the commutative square
\begin{equation}
\label{diag:pi1decomp}
\vcenter{\xymatrix{
	T^3 \sslash V_2 \times (\T^1)^{\{p\}}
	\ar[r]^-\isom
	\ar[d]_{\pi_1}
	&
	(T^3 \sslash V_2) 
		\times 
	(\pt\sslash(\T^1)^{\{p\}})	
	\ar[d]^{\pi'_1\times 1}
	\\
	\pt \sslash V_2 \times (\T^1)^{\{p\}}
	\ar[r]^-\isom
	&
	(\pt \sslash V_2) \times (\pt\sslash (\T^1)^{\{p\}})
}}
\end{equation}
where the $V_2$-action on $T^3$
is given by the formula
\[
	u \cdot (z_v)_{v\in V_2} \Delta \T^1 
	= 
	(z_{u+v})_{v\in V_2} \Delta \T^1,
\]
the map 
$\pi'_1$ is induced by the projection $T^3\to \pt$,
and the horizontal maps are given by 
the evident homeomorphisms.
Thus to compute the umkehr map $\pi_1^!$, it suffices
to understand the umkehr map $(\pi'_1)^!$.

To understand the map $(\pi'_1)^!$, we will 
study in detail the $V_2$-space $T^3$.
For a homomorphism $\alpha\colon V_2 \to \Z^\times$,
write $\R(\alpha)$ for the associated 1-dimensional
real representation.
From representation theory we know that the 
maps 
\[
	\R[V_2] 
	\longto
	\;\prod_{\mathclap{\alpha\colon V_2 \to \Z^\times}}\;\R(\alpha),
	\quad
	v 
	\longmapsto 
	\,\sum_{\mathclap{\alpha\colon V_2 \to \Z^\times}}\; \alpha(v) e_\alpha
\]
and
\[
	\prod_{\mathclap{\alpha\colon V_2 \to \Z^\times}}\;\R(\alpha)
	\longto 
	\R[V_2],
	\quad 
	e_\alpha \longmapsto \frac{1}{4} \sum_{v\in V_2} \alpha(v)v
\]
are inverse isomorphisms of $V_2$-representations.
Here $e_\alpha$ denotes the unit basis vector
associated to the $\R(\alpha)$-factor in $\prod_\alpha \R(\alpha)$.
These isomorphisms induce a $V_2$-equivariant homeomorphism
\[
	\prod_{\alpha\neq 1} \R(\alpha) / L 
	\xto{\ \isom\ }
	(\R[V_2]/\Delta\R) / (\Z[V_2] / \Delta \Z),
	\quad
	e_\alpha + L 
	\mapsto 
	\left[\frac{1}{4} \sum_{v\in V_2} \alpha(v)v\right]
\]
where $1$ denotes the trivial homomorphism $V_2 \to \Z^\times$,
$L \subset \prod_{\alpha\neq 1} \R(\alpha)$ is the lattice
generated the vectors $\sum_{\alpha\neq 1} \alpha(v)e_\alpha$
for $v\in V_2$, $v\neq 0$, and $\Delta \R$ and $\Delta\Z$
denote the diagonal subgroups of $\R[V_2]$ and $\Z[V_2]$,
respectively. 
The target of the above is canonically $V_2$-homeomorphic to 
$(\R[V_2]/\Z[V_2]) / (\Delta \R /\Delta \Z)$,
which in turn is in an evident way $V_2$-homeomorphic
to the space $T^3 = (\T^1)^{V_2} / \Delta \T^1$.
Thus we obtain a $V_2$-equivariant homeomorphism
\begin{equation}
\label{iso:t3desc} 
	\prod_{\alpha\neq 1} \R(\alpha) / L 
	\xto{\ \isom\ }
	T^3.
\end{equation}


\newcommand{\emkpos}{0.83} 
\newcommand{\emklen}{4.7pt} 
\newcommand{\emkwid}{2.2pt} 
\newcommand{\emkgap}{1.05pt} 
\newcommand{\emkadjc}{0.7} 
\tikzset{
	-e1-/.style={
		decoration={
    		markings, 
			mark=at position \emkpos with {
				\draw[fill=black] (0,0) -- (-\emklen,-\emkwid) -- (-\emklen,\emkwid) -- cycle;
			}
		},
		postaction={decorate}
	},
	-e2-/.style={
		decoration={
    		markings, 
			mark=at position \emkpos with {
				\draw[fill=black] (0,0) -- (-\emklen,-\emkwid) -- (-\emklen,\emkwid) -- cycle;
				\draw (-\emklen-\emkgap-\emkadjc\pgflinewidth,-\emkwid-0.5\pgflinewidth) 
						-- +(0pt,2*\emkwid + \pgflinewidth);
			}
		},
		postaction={decorate}
	},
	-e3-/.style={
		decoration={
    		markings, 
			mark=at position \emkpos with {
				\draw[fill=black] (0,0) -- (-\emklen,-\emkwid) -- (-\emklen,\emkwid) -- cycle;
				\draw (-\emklen-\emkgap-\emkadjc\pgflinewidth,-\emkwid-0.5\pgflinewidth) 
						-- +(0pt,2*\emkwid + \pgflinewidth);
				\draw (-\emklen-2*\emkgap-2*\emkadjc\pgflinewidth,-\emkwid-0.5\pgflinewidth) 
						-- +(0pt,2*\emkwid + \pgflinewidth);
			}
		},
		postaction={decorate}
	},
	-e1'-/.style={
		decoration={
    		markings, 
			mark=at position \emkpos with {
				\draw (0,0) -- (-\emklen,-\emkwid) -- (-\emklen,\emkwid) -- cycle;
			}
		},
		postaction={decorate}
	},
	-e2'-/.style={
		decoration={
    		markings, 
			mark=at position \emkpos with {
				\draw (0,0) -- (-\emklen,-\emkwid) -- (-\emklen,\emkwid) -- cycle;
				\draw (-\emklen-\emkgap-\emkadjc\pgflinewidth,-\emkwid-0.5\pgflinewidth) 
						-- +(0pt,2*\emkwid + \pgflinewidth);
			}
		},
		postaction={decorate}
	},
	-e3'-/.style={
		decoration={
    		markings, 
			mark=at position \emkpos with {
				\draw (0,0) -- (-\emklen,-\emkwid) -- (-\emklen,\emkwid) -- cycle;
				\draw (-\emklen-\emkgap-\emkadjc\pgflinewidth,-\emkwid-0.5\pgflinewidth) 
						-- +(0pt,2*\emkwid + \pgflinewidth);
				\draw (-\emklen-2*\emkgap-2*\emkadjc\pgflinewidth,-\emkwid-0.5\pgflinewidth) 
						-- +(0pt,2*\emkwid + \pgflinewidth);
			}
		},
		postaction={decorate}
	},
}
\newcommand{\elbllen}{11pt}
\newcommand*{\lblei}
	{\begin{tikzpicture}\draw [-e1-] (0,0)--(\elbllen,0); \end{tikzpicture}}
\newcommand*{\lbleii}
	{\begin{tikzpicture}\draw [-e2-] (0,0)--(\elbllen,0); \end{tikzpicture}}
\newcommand*{\lbleiii}
	{\begin{tikzpicture}\draw [-e3-] (0,0)--(\elbllen,0); \end{tikzpicture}}
\newcommand*{\lbleip}
	{\begin{tikzpicture}\draw [-e1'-] (0,0)--(\elbllen,0); \end{tikzpicture}}
\newcommand*{\lbleiip}
	{\begin{tikzpicture}\draw [-e2'-] (0,0)--(\elbllen,0); \end{tikzpicture}}
\newcommand*{\lbleiiip}
	{\begin{tikzpicture}\draw [-e3'-] (0,0)--(\elbllen,0); \end{tikzpicture}}

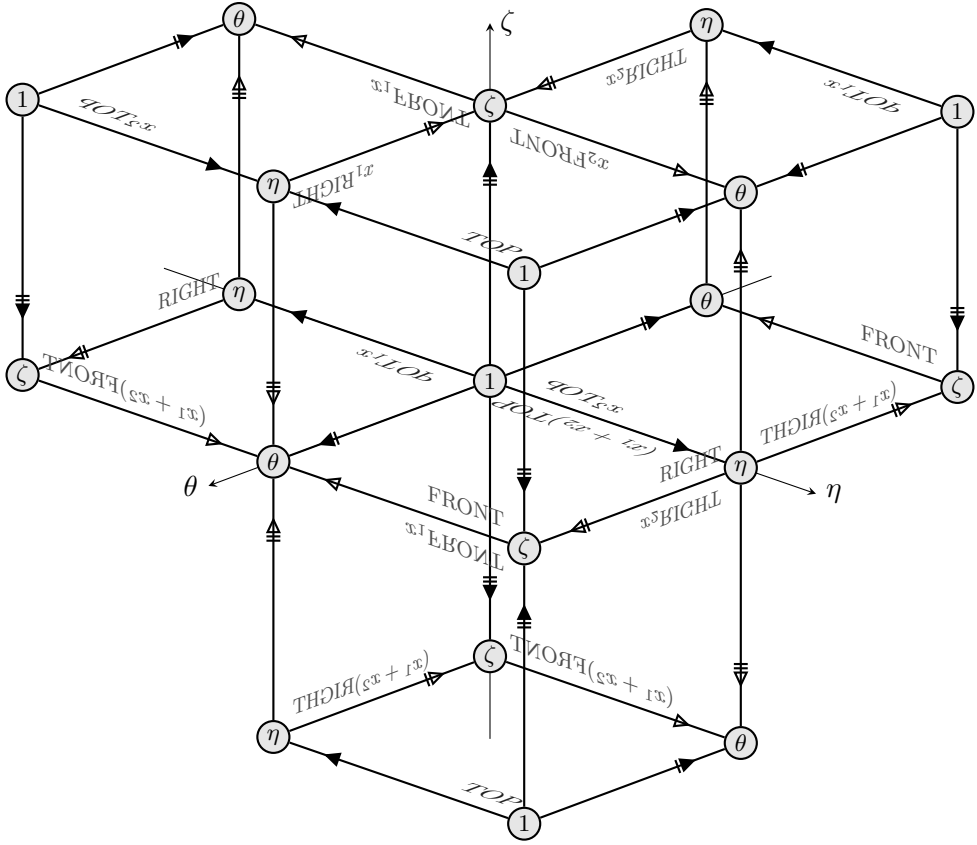
\begin{figure}
\begin{tikzpicture}
	[
		x={(3.3cm,-1.15cm)}, 
		y={(-2.85cm,-1.07cm)}, 
		z={(0cm,3.65cm)}, 		
	    axis/.style={-stealth},
	    vertex/.style={
	    	circle,
			thick,
			draw=black,
			fill=gray!20,
			inner sep=0.0pt,
	    	font=\fontsize{9pt}{9pt}\selectfont,
			minimum size=12pt
		},
		edge/.style={thick},
		topfacelabel/.style={
			font=\fontsize{10pt}{10pt}\selectfont,
			color=black!87
		},
		facelabel/.style={
			font=\fontsize{8pt}{8pt}\selectfont,
			color=black!70
		},		
	]
	
	\newcommand*{\scalefactor}{1.0} 
		
	\pgfresetboundingbox
	\pgfpathrectangle
		{\pgfpoint{-6.5cm*\scalefactor}{-6.2cm*\scalefactor}}
		{\pgfpoint{13cm*\scalefactor}{11.3cm*\scalefactor}}
	\pgfusepath{use as bounding box}
	
	\pgflowlevel{\pgftransformscale{\scalefactor}}	

    \draw[axis] (-1.3,0,0) -- (1.3,0,0) node [right] {$\eta$};
    \draw[axis] (0,-1.3,0) -- (0,1.3,0) node [left]	 {$\theta$};
    \draw[axis] (0,0,-1.3) -- (0,0,1.3) node [right] {$\zeta$};

    \node[vertex] (v000) at (0,0,0)		{$1$};
    \node[vertex] (vppp) at (1,1,1)		{$1$};
    \node[vertex] (vppm) at (1,1,-1)	{$1$};
    \node[vertex] (vpmp) at (1,-1,1)	{$1$};
    \node[vertex] (vmpp) at (-1,1,1)	{$1$};
    \node[vertex] (vp00) at (1,0,0) 	{$\eta$};
    \node[vertex] (vm00) at (-1,0,0)	{$\eta$};
	\node[vertex] (v0pp) at (0,1,1)		{$\eta$};
	\node[vertex] (v0pm) at (0,1,-1)	{$\eta$};
	\node[vertex] (v0mp) at (0,-1,1)	{$\eta$};
    \node[vertex] (v0p0) at (0,1,0)		{$\theta$};
    \node[vertex] (v0m0) at (0,-1,0)	{$\theta$};
    \node[vertex] (vp0p) at (1,0,1)		{$\theta$};
    \node[vertex] (vp0m) at (1,0,-1)	{$\theta$};
    \node[vertex] (vm0p) at (-1,0,1)	{$\theta$};
    \node[vertex] (v00p) at (0,0,1)		{$\zeta$};
    \node[vertex] (v00m) at (0,0,-1)	{$\zeta$};
    \node[vertex] (vpp0) at (1,1,0)		{$\zeta$};
    \node[vertex] (vpm0) at (1,-1,0)	{$\zeta$};
    \node[vertex] (vmp0) at (-1,1,0)	{$\zeta$};
	
	\pgflowlevelobj{
		\pgftransformtriangle
		{\pgfpointxyz{0}{-1}{0}}
		{\pgfpointxyz{1}{-1}{0}}
		{\pgfpointxyz{0}{-1}{1}}
		\pgftransformscale{0.01}
	}
	{\node [left,facelabel] at (96pt,8pt) {FRONT};}
		
	\pgflowlevelobj{
		\pgftransformtriangle
		{\pgfpointxyz{-1}{0}{1}}
		{\pgfpointxyz{0}{0}{1}}
		{\pgfpointxyz{-1}{0}{0}}
		\pgftransformscale{0.01}
	}
	{\node [left,facelabel] at (96pt,8pt) {$\smash{x_1}$FRONT};}
	
	\pgflowlevelobj{
		\pgftransformtriangle
		{\pgfpointxyz{1}{0}{1}}
		{\pgfpointxyz{0}{0}{1}}
		{\pgfpointxyz{1}{0}{0}}
		\pgftransformscale{0.01}
	}
	{\node [left,facelabel] at (96pt,8pt) {$\smash{x_2}$FRONT};}
	
	\pgflowlevelobj{
		\pgftransformtriangle
		{\pgfpointxyz{1}{0}{-1}}
		{\pgfpointxyz{0}{0}{-1}}
		{\pgfpointxyz{1}{0}{0}}
		\pgftransformscale{0.01}
	}
	{\node [left,facelabel] at (96pt,8pt) {$\smash{(x_1+x_2)}$FRONT};}
	
	\pgflowlevelobj{
		\pgftransformtriangle
		{\pgfpointxyz{0}{1}{-1}}
		{\pgfpointxyz{1}{1}{-1}}
		{\pgfpointxyz{0}{0}{-1}}
		\pgftransformscale{0.01}
	}
	{\node [left,topfacelabel] at (92pt,11.5pt) {TOP};}
	
	\pgflowlevelobj{
		\pgftransformtriangle
		{\pgfpointxyz{0}{-1}{1}}
		{\pgfpointxyz{1}{-1}{1}}
		{\pgfpointxyz{0}{0}{1}}
		\pgftransformscale{0.01}
	}
	{\node [left,topfacelabel] at (92pt,11.5pt) {$\smash{x_1}$TOP};}
	
	\pgflowlevelobj{
		\pgftransformtriangle
		{\pgfpointxyz{-1}{0}{0}}
		{\pgfpointxyz{0}{0}{0}}
		{\pgfpointxyz{-1}{1}{0}}
		\pgftransformscale{0.01}
	}
	{\node [left,topfacelabel] at (92pt,11.5pt) {$\smash{x_1}$TOP};}
	
	\pgflowlevelobj{
		\pgftransformtriangle
		{\pgfpointxyz{0}{1}{1}}
		{\pgfpointxyz{-1}{1}{1}}
		{\pgfpointxyz{0}{0}{1}}
		\pgftransformscale{0.01}
	}
	{\node [left,topfacelabel] at (92pt,11.5pt) {$\smash{x_2}$TOP};}
	
	\pgflowlevelobj{
		\pgftransformtriangle
		{\pgfpointxyz{1}{0}{0}}
		{\pgfpointxyz{0}{0}{0}}
		{\pgfpointxyz{1}{-1}{0}}
		\pgftransformscale{0.01}
	}
	{\node [left,topfacelabel] at (92pt,11.5pt) {$\smash{x_2}$TOP};}
		
	\pgflowlevelobj{
		\pgftransformtriangle
		{\pgfpointxyz{1}{0}{0}}
		{\pgfpointxyz{0}{0}{0}}
		{\pgfpointxyz{1}{1}{0}}
		\pgftransformscale{0.01}
	}
	{\node [left,topfacelabel] at (92pt,11.5pt) {$\smash{(x_1+x_2)}$TOP};}
	
	\pgflowlevelobj{
		\pgftransformtriangle
		{\pgfpointxyz{-1}{1}{0}}
		{\pgfpointxyz{-1}{0}{0}}
		{\pgfpointxyz{-1}{1}{1}}
		\pgftransformscale{0.01}
	}
	{\node [left,facelabel] at (95pt,8pt) {RIGHT};}	
	
	\pgflowlevelobj{
		\pgftransformtriangle
		{\pgfpointxyz{0}{0}{1}}
		{\pgfpointxyz{0}{1}{1}}
		{\pgfpointxyz{0}{0}{0}}
		\pgftransformscale{0.01}
	}
	{\node [left,facelabel] at (95pt,8pt) {$\smash{x_1}$RIGHT};}	
	
	\pgflowlevelobj{
		\pgftransformtriangle
		{\pgfpointxyz{1}{1}{0}}
		{\pgfpointxyz{1}{0}{0}}
		{\pgfpointxyz{1}{1}{-1}}
		\pgftransformscale{0.01}
	}
	{\node [left,facelabel] at (95pt,8pt) {$\smash{x_2}$RIGHT};}	
	
	\pgflowlevelobj{
		\pgftransformtriangle
		{\pgfpointxyz{0}{0}{1}}
		{\pgfpointxyz{0}{-1}{1}}
		{\pgfpointxyz{0}{0}{0}}
		\pgftransformscale{0.01}
	}
	{\node [left,facelabel] at (95pt,8pt) {$\smash{x_2}$RIGHT};}	
	
	\pgflowlevelobj{
		\pgftransformtriangle
		{\pgfpointxyz{1}{-1}{0}}
		{\pgfpointxyz{1}{0}{0}}
		{\pgfpointxyz{1}{-1}{1}}
		\pgftransformscale{0.01}
	}
	{\node [left,facelabel] at (95pt,8pt) {$\smash{(x_1+x_2)}$RIGHT};}	
	
	\pgflowlevelobj{
		\pgftransformtriangle
		{\pgfpointxyz{0}{0}{-1}}
		{\pgfpointxyz{0}{1}{-1}}
		{\pgfpointxyz{0}{0}{0}}
		\pgftransformscale{0.01}
	}
	{\node [left,facelabel] at (95pt,8pt) {$\smash{(x_1+x_2)}$RIGHT};}	

    \draw [edge,-e1-] (v000) -- (vp00);
    \draw [edge,-e1-] (v000) -- (vm00);
    \draw [edge,-e1-] (vppp) -- (v0pp);
    \draw [edge,-e1-] (vppm) -- (v0pm);
    \draw [edge,-e1-] (vpmp) -- (v0mp);
    \draw [edge,-e1-] (vmpp) -- (v0pp);
    
    \draw [edge,-e1'-] (v00p) -- (vp0p);
    \draw [edge,-e1'-] (v00p) -- (vm0p);
    \draw [edge,-e1'-] (vpp0) -- (v0p0);
    \draw [edge,-e1'-] (vmp0) -- (v0p0);
    \draw [edge,-e1'-] (vpm0) -- (v0m0);
    \draw [edge,-e1'-] (v00m) -- (vp0m);

	\draw [edge,-e2-] (v000) -- (v0p0);
	\draw [edge,-e2-] (v000) -- (v0m0);
	\draw [edge,-e2-] (vppp) -- (vp0p);
	\draw [edge,-e2-] (vpmp) -- (vp0p);
	\draw [edge,-e2-] (vmpp) -- (vm0p);
	\draw [edge,-e2-] (vppm) -- (vp0m);
	
	\draw [edge,-e2'-] (v0pp) -- (v00p);
	\draw [edge,-e2'-] (v0pm) -- (v00m);
	\draw [edge,-e2'-] (v0mp) -- (v00p);
	\draw [edge,-e2'-] (vp00) -- (vpp0);
	\draw [edge,-e2'-] (vp00) -- (vpm0);
	\draw [edge,-e2'-] (vm00) -- (vmp0);

	\draw [edge,-e3-] (v000) -- (v00p);
	\draw [edge,-e3-] (v000) -- (v00m);
	\draw [edge,-e3-] (vppp) -- (vpp0);
	\draw [edge,-e3-] (vppm) -- (vpp0);
	\draw [edge,-e3-] (vmpp) -- (vmp0);
	\draw [edge,-e3-] (vpmp) -- (vpm0);

	\draw [edge,-e3'-] (vp00) -- (vp0p);
	\draw [edge,-e3'-] (vp00) -- (vp0m);
	\draw [edge,-e3'-] (v0pp) -- (v0p0);
	\draw [edge,-e3'-] (v0pm) -- (v0p0);
	\draw [edge,-e3'-] (vm00) -- (vm0p);
	\draw [edge,-e3'-] (v0m0) -- (v0mp);

	\pgflowlevelobj{
		\pgftransformtriangle
		{\pgfpointxyz{0}{1}{1}}
		{\pgfpointxyz{1}{1}{1}}
		{\pgfpointxyz{0}{0}{1}}
		\pgftransformscale{0.01}
	}
	{\node [left,topfacelabel] at (92pt,11.5pt) {TOP};}
	
	\pgflowlevelobj{
		\pgftransformtriangle
		{\pgfpointxyz{1}{1}{0}}
		{\pgfpointxyz{1}{0}{0}}
		{\pgfpointxyz{1}{1}{1}}
		\pgftransformscale{0.01}
	}
	{\node [left,facelabel] at (95pt,8pt) {RIGHT};}	
	
	\pgflowlevelobj{
		\pgftransformtriangle
		{\pgfpointxyz{0}{1}{0}}
		{\pgfpointxyz{1}{1}{0}}
		{\pgfpointxyz{0}{1}{1}}
		\pgftransformscale{0.01}
	}
	{\node [left,facelabel] at (96pt,8pt) {FRONT};}

	\pgflowlevelobj{
		\pgftransformtriangle
		{\pgfpointxyz{0}{1}{0}}
		{\pgfpointxyz{1}{1}{0}}
		{\pgfpointxyz{0}{1}{-1}}
		\pgftransformscale{0.01}
	}
	{\node [left,facelabel] at (96pt,8pt) {$\smash{x_1}$FRONT};}

	\pgflowlevelobj{
		\pgftransformtriangle
		{\pgfpointxyz{0}{1}{0}}
		{\pgfpointxyz{-1}{1}{0}}
		{\pgfpointxyz{0}{1}{1}}
		\pgftransformscale{0.01}
	}
	{\node [left,facelabel] at (96pt,8pt) {$\smash{(x_1+x_2)}$FRONT};}	

\end{tikzpicture}
\caption{The $V_2$-CW structure on
$\prod_{\alpha\neq 1} \R(\alpha) / L \isom T^3$.
Acting by $x_1$, $x_2$ and $x_1+x_2$ sends the cube in the middle
to the cube on the left, on the right, and at the bottom, respectively.
Faces and vertices are identified as indicated. 
Edges with the same label and orientation are also identified.
Edges with the same label and opposite orientation are
\emph{not} identified, but belong to the same $V_2$-orbit.
Vertices labeled by $1$ correspond to elements in $L$, and
vertices labeled by $\eta$, $\theta$ and $\zeta$ correspond to 
the respective basis vectors of $\prod_{\alpha\neq 1} \R(\alpha)$.}
\label{fig:cubestr}
\end{figure}

Write $\eta$, $\theta$ and $\zeta$ for the unique surjective homomorphisms
$V_2 \to \Z^\times$ sending $x_1$, $x_2$ and $x_1+x_2$ to the 
identity element, respectively,
and continue to write
$1$ for the constant homomorphism $V_2 \to \Z^\times$.
The lattice
$\prod_{\alpha\neq 1} \Z(\alpha) \subset \prod_{\alpha\neq 1} \R(\alpha)$
induces a $V_2$-equivariant cubical complex structure on 
$\prod_{\alpha\neq 1}\R(\alpha)$, and this structure descends to the
quotient $\prod_{\alpha\neq 1} \R(\alpha) / L$
to give it and hence, via \eqref{iso:t3desc}, the space $T^3$ the
$V_2$-CW structure depicted in Figure~\ref{fig:cubestr}.
Thus we obtain the cellular chain complex
\begin{equation}
\label{eq:ct3} 
	C_\ast = C_\ast(T^3) = ( 0\to C_3 \xto{d_3} C_2 \xto{d_2} C_1 \xto{d_1} C_0 \to 0)
\end{equation}
for $T^3$
where 
\begin{align*}
	C_3 &= \F_2[V_2]e_\mathrm{cube}
	\\
	C_2 &= \F_2[V_2]\{e_\mathrm{top}, e_\mathrm{front}, e_\mathrm{right}\}
	\\
	C_1 &= \F_2[V_2/\langle x_1\rangle]\{e_{\lblei}, e_{\lbleip}\}
			\oplus
	       \F_2[V_2/\langle x_2\rangle]\{e_{\lbleii}, e_{\lbleiip}\}\\
	       &\qquad\qquad
			\oplus
	       \F_2[V_2/\langle x_1+x_2\rangle]\{e_{\lbleiii}, e_{\lbleiiip}\}
	\\
	C_0 &= \F_2\{e_1,e_\eta,e_\theta,e_\zeta\}
\end{align*}
and
\begin{equation}
\label{eq:bdrymap}
\begin{aligned}
	d_3 (e_\mathrm{cube}) &= (1+[x_1+x_2]) e_\mathrm{top}
								+(1+[x_2]) e_\mathrm{front}
							    +(1+[x_1]) e_\mathrm{right}
	\\			
	d_2 (e_\mathrm{top}) &= [x_2]e_{\lblei} + e_{\lbleip}
								 + [x_1] e_{\lbleii} + [x_1] e_{\lbleiip}
	\\
	d_2 (e_\mathrm{front}) &= [x_2]e_{\lblei} + [x_2]e_{\lbleip}
								 + [x_1] e_{\lbleiii} + [x_1] e_{\lbleiiip}
	\\
	d_2 (e_\mathrm{right}) &= [x_1]e_{\lbleii} + e_{\lbleiip}
								 + [x_1] e_{\lbleiii} +  e_{\lbleiiip}
	\\
	d_1 (e_{\lblei}) &= e_1 + e_\eta
	\\
	d_1 (e_{\lbleip}) &= e_\theta+e_\zeta
	\\
	d_1 (e_{\lbleii}) &= e_1 + e_\theta
	\\
	d_1 (e_{\lbleiip}) &= e_\eta+e_\zeta
	\\
	d_1 (e_{\lbleiii}) &= e_1 + e_\zeta
	\\
	d_1 (e_{\lbleiiip}) &= e_\eta+e_\theta
\end{aligned}
\end{equation}
Here $e_\mathrm{cube}$ refers to the middle cube in 
Figure~\ref{fig:cubestr},
$e_{\lblei}$ etc.\ refer to the edge with the corresponding label 
whose orientation agrees with that of a coordinate axis,
and $e_\mathrm{top}$, $e_\mathrm{front}$, $e_\mathrm{right}$ and
$e_1$, $e_\eta$, $e_\theta$, $e_\zeta$ refer to the 
faces and vertices with corresponding labels. 
For clarity we have written $[v]$ for the element of $\F_2[V_2]$
corresponding to $v\in V_2$.

The $0$-cells of $T^3$  induce sections 
\[
	s_1, s_\eta, s_\theta, s_\zeta \colon \pt\sslash V_2 \longto T^3\sslash V_2
\]
of $\pi'_1$; alternatively, $s_\lambda$ can be described as the map
defined by the assignment 
\[
	\pt 
	\longmapsto 
	(\lambda(v))_{v_\in V_2} \Delta \T^1
	\in T^3 
	= 
	(\T^1)^{V_2}/\Delta \T^1.
\]
where we interpret $\Z^\times =\{\pm 1\} \subset \T^1$.
The following lemma is the key ingredient in the proof of 
Theorem~\ref{thm:tcalc2}.
\begin{lemma}
\label{lm:tcalc2}
The umkehr map 
$(\pi'_1)^! \colon H_\ast (\pt\sslash V_2) \to H_{\ast+3}(T^3\sslash V_2)$
is given by 
\[
	(\pi'_1)^! (x_1^{[n_1]} x_2^{[n_2]}) 
	= 
	[(s_1)_\ast + (s_\eta)_\ast + (s_\theta)_\ast + (s_\zeta)_\ast]
	\left(\sum_{i=0}^{n_1+1} x_1^{[i]} x_2^{[n_1+n_2+3-i]}\right)
\]
for all $n_1,n_2\geq 0$.
\end{lemma}
\begin{proof}
Let $F_\ast$ be the divided power algebra over the ring $\F_2[V_2]$
generated by variables $X_1$ and $X_2$ of degree 1,
and equip $F_\ast$ with the differential 
\[
	d (X_1^{[k_1]} X_2^{[k_2]}) 
	= 
	t_1 X_1^{[k_1-1]} X_2^{[k_2]}
	+
	t_2 X_1^{[k_1]} X_2^{[k_2-1]}.
\]
Here  $t_i = 1+x_i$, $i=1,2$; notice that 
$\F_2[V_2]$ is simply the exterior algebra $\Lambda(t_1,t_2)$
over $\F_2$.
Equipped with the augmentation 
\[
	\varepsilon\colon F_0 = \Lambda(t_1,t_2) \longto \F_2
\]
given by the $\F_2$-algebra homomorphism sending $t_1$ and 
$t_2$ to $0$, the chain complex $F_\ast$ becomes a free 
$\F_2[V_2]$-resolution of $\F_2$.
The multiplication on $F_\ast$ is an admissible product 
in the sense of \cite[section V.5]{Brown}, 
and hence models the Pontryagin product
on $H_\ast (\pt\sslash V_2) = H_\ast (F_\ast \tensor_{V_2} \F_2)$.
Concretely, the product
$x_1^{[n_1]} x_2 ^{[n_2]}$ is represented by the cycle
$X_1^{[n_1]} X_2^{[n_2]}\tensor_{V_2} 1$ of $F_\ast \tensor_{V_2} \F_2$.

Consider now the double complex $F_\ast \tensor_{V_2} C_\ast$.
Its total complex is a chain model for the space $T^3 \sslash V_2$,
and the spectral sequence associated with $F_\ast\tensor_{V_2} C_\ast$
obtained by first taking homology in the $C_\ast$-direction 
agrees (from the $E^2$-page onwards) with the Serre spectral
sequence of the fibration 
$\pi'_1 \colon T^3\sslash V_2 \to \pt \sslash V_2$. 
As before, by \cite[Lemma~8.4]{HL}, the umkehr map
$(\pi_1')^!$ 
is equal to the composite
\begin{equation}
	H_\ast(\pt\sslash V_2) 
	\xto{\ \isom\ }
	H_\ast(\pt\sslash V_2;\,H_3 T^3)
	=
	E^2_{n,3}
	\longtwoheadto
	E^\infty_{n,3}
	\longincl
	H_{n+3} (T^3\sslash V_2)
\end{equation}
where $E^2$ and $E^\infty$ refer to pages in the Serre
spectral sequence of $\pi_1'$, the first map 
is induced by the fundamental class of $T^3$,
and the last two maps arise from the fact
that we are working with the top row of the
spectral sequence. Observing that the fundamental class of
$T^3$ is represented by the cycle $t_1 t_2 e_\mathrm{cube} \in C_3$,
we obtain the following recipe for computing $(\pi'_1)^!$:
for any $n_1,n_2\geq 0$, the element
$(\pi'_1)^!(x_1^{[n_1]} x_2^{[n_2]}) \in H_\ast(T^3\sslash V_2)$
is the homology class of the element 
$X_1^{[n_1]} X_2^{[n_2]} \tensor_{V_2} t_1 t_2 e_\mathrm{cube}$
considered as a cycle in $\Tot(F_\ast \tensor_{V_2} C_\ast)$.

Let
\begin{align*}
 	c_1 &= X_1^{[n_1+1]} X_2^{[n_2]} \tensor_{V_2} t_2 e_\mathrm{cube}
	\\
	c_2 &= X_1^{[n_1+2]} X_2^{[n_2]} 
				\tensor_{V_2} 
			t_2 (e_\mathrm{top}+ e_\mathrm{right})
\end{align*}
and
\[
	c_3 = X_1^{[n_1+2]} X_2^{[n_2+1]} 
				\tensor_{V_2}
			(e_{\lblei} + e_{\lbleip})
		  + \sum_{i=0}^{n_1+2} X_1^{[i]}X_2^{[n_1+n_2+3-i]}
		  		\tensor_{V_2}
			(e_{\lbleiii} + e_{\lbleiiip}).
\]
Using the formulas \eqref{eq:bdrymap} for the boundary map
in $C_\ast$, it is easy to compute that 
\begin{multline*}
	d(c_1 + c_2 + c_3)\\
	= 
	X_1^{[n_1]} X_2^{[n_2]} \tensor_{V_2} t_1 t_2 e_\mathrm{cube} 
	+ 
	\sum_{i=0}^{n_1+1} X_1^{[i]} X_2^{[n_1+n_2+3-i]} 
		\tensor_{V_2} 
	(e_1 + e_\eta + e_\theta + e_\zeta)
\end{multline*}
in $\Tot(F_\ast \tensor_{V_2} C_\ast)$,
showing that the two summands on the right hand side are homologous. 
The claim follows.
\end{proof}

\begin{proof}[Proof of Theorem~\ref{thm:tcalc2}]
For $\lambda \colon V_2 \to \{\pm1\} \subset \T^1$ a homomorphism,
the assignment $\pt \mapsto (\lambda(v))_{v\in V_2}\in (\T^1)^{V_2}$
and the homomorphism
\[
	V_2\times (\T^1)^{\{p\}} 
	\longto
	V_2\times (\T^1)^{\{p,q\}},
	\quad
	(v,z_p) \longmapsto (v,z_p,\lambda(v)z_p)
\]
define a map
\[
	\tilde{s}_\lambda 
	\colon 
	\pt \sslash V_2\times (\T^1)^{\{p\}} 
	\longto
	(\T^1)^{V_2}\sslash	V_2\times (\T^1)^{\{p,q\}}
\]
lifting $s_\lambda \times 1$ 
through the map $\kappa$ of diagram
\eqref{diag:push-pull-for-alphat1-2}
in the sense that the following diagram commutes:
\[\xymatrix{
	(\pt\sslash V_2) \times (\pt \sslash (\T^1)^{\{p\}})
	\ar[d]_{s_{\lambda} \times 1}
	&
	\pt \sslash  V_2 \times (\T^1)^{\{p\}}
	\ar[l]_-{\isom}
	\ar[dr]^{\tilde s_{\lambda}}
	\\
	(T^3 \sslash V_2) \times (\pt \sslash (\T^1)^{\{p\}})
	&
	T^3 \sslash  V_2 \times (\T^1)^{\{p\}}
	\ar[l]_-{\isom}
	&
	(\T^1)^{V_2}\sslash V_2 \times (\T^1)^{\{p,q\}}
	\ar[l]_-\kappa^-\homot
}\]
Here the homeomorphisms on the left are as in 
diagram \eqref{diag:pi1decomp}.
Notice that the composite of $\tilde s_\lambda$
with the map $\pi_2$ of 
\eqref{diag:push-pull-for-alphat1-2}
is the map 
$\pt \sslash V_2 \times \T^1
	\longto
\pt \sslash \T^1$
induced by the composite homomorphism
\[
	V_2\times \T^1 
	\xto{\ \lambda \times 1\ }
	V_1 \times \T^1
	\xto{\ \beta \times 1\ }
	\T^1\times \T^1
	\xto{\ \mu\ }
	\T^1
\]
where we interpret $\lambda$ as a homomorphism into $V_1$
and $\mu$ denotes the multiplication map of $\T^1$.
Using Lemma~\ref{lm:tcalc2},
we now obtain the formula
\begin{align*}
	&\,\alpha_2^{\T^1} \big(x_1^{[n_1]} x_2^{[n_2]}\tensor b\big)
	\\
	=&\,\beta_\ast \,\sum_{\mathclap{\lambda\colon V_2\to V_1}}\,
		\lambda_\ast 
		\left(\sum_{i=0}^{n_1+1} x_1^{[i]} x_2^{[n_1+n_2+3-i]}\right)b
	\\
	=&\,\beta_\ast \left(
		\sum_{{\lambda\colon V_2\to V_1}} 
			\lambda_\ast \big(x_2^{[n_1+n_2+3]}\big)
		+
		\sum_{i=1}^{n_1+1} 
		\sum_{{\lambda\colon V_2\to V_1}} 
			\lambda_\ast\big(x_1^{[i]}\big)
			\lambda_\ast\big(x_2^{[n_1+n_2+3-i]}\big)
		\right)b
	\\
	=&\,\beta_\ast\left(
			\sum_{i=1}^{n_1+1} x^{[i]} x^{[n_1+n_2+3-i]} 
		\right)b
	\\
	=&\,\beta_\ast \left(
			\sum_{i=1}^{n_1+1} {{n_1+n_2+3}\choose{i}} x^{[n_1+n_2+3]}
		\right)b
	\\
	=&\,\beta_\ast \left(
			\sum_{i=1}^{n_1+1} 
				\left[
					{{n_1+n_2+2}\choose{i-1}} 
					+ 
					{{n_1+n_2+2}\choose{i}} 
				\right]
				x^{[n_1+n_2+3]}
		\right)b
	\\
	=&\,\left(1+{{n_1+n_2+2}\choose{n_1+1}}\right)
	\beta_\ast\big(x^{[n_1+n_2+3]}\big) b.\qedhere
\end{align*}
\end{proof}

\section{Computations for $SU(2)$}
\label{sec:su2calc}

The purpose of this section is to compute the operation
$\alpha^{SU(2)}_1$.
Interpret $V_1$ as $\pm 1$, and let $\nu$ be the map
\[
	\nu \colon V_1 \times SU(2) \longto SU(2),
	\quad
	(\varepsilon,A) \longmapsto \varepsilon A.
\]
It is easy to check that the map induced by $\nu$ 
makes $H_\ast(BSU(2))$ into a module over 
the ring $H_\ast(BV_1)$, which we continue
to interpret as the divided power algebra
over a single generator $x$ of degree 1.
Our aim is to prove the following theorem.
\begin{theorem}
\label{thm:su2calc}
The map 
\[
	\alpha^{SU(2)}_1 
	\colon
	H_{\ast-3}( BV_1) \tensor H_\ast(BSU(2)) \longto H_\ast (BSU(2))
\]
is given by 
\[
	x^{[n]} \tensor b \longmapsto x^{[n+3]}\cdot b
\]
where $\cdot$ refers to the module structure induced by $\nu$.
\end{theorem}

Before proving the theorem, let us spend a moment analyzing the 
$H_\ast BV_1$-module structure of $H_\ast BSU(2)$.
Recall that $H_\ast BSU(2)$ is isomorphic to $\Z/2$ in degrees
divisible by 4 and vanishes in other degrees. Let $i$ denote the inclusion
$i\colon V_1 \to SU(2)$, $\varepsilon \mapsto \varepsilon I$
where $I\in SU(2)$ is the identity matrix. 
We claim that the map $i_\ast \colon H_\ast BV_1 \to H_\ast BSU(2)$
is an isomorphism in degrees divisible by 4 and zero in other degrees.
One way to see this is to recall that the cohomology of 
$BSU(2)$ is a polynomial algebra generated by the Euler class
$e(\gamma_2^\C)$ of the tautological complex 2-plane bundle $\gamma_2^\C$
over $BSU(2)$, and observe that the pullback 
$(Bi)^\ast \gamma_2^\C$ is isomorphic to the direct sum of four 
copies of the tautological real line bundle $\gamma_1^\R$ over $BV_1$.
Thus the map 
$i^\ast\colon H^\ast BSU(2) \to H^\ast BV_1$ 
sends the generator
$e(\gamma_2^\C) = w_4^{}(\gamma_2^\C)$ to 
$w_4^{}(\bigoplus^4 \gamma_1^\R) = w_1^{}(\gamma_1^\R)^4 = x^4$,
and the claim follows by passing to duals.

With the behaviour of the map $i_\ast$ understood, we can now read off 
the $H_\ast BV_1$-module structure of $H_\ast BSU(2)$
from the commutative diagram
\[\xymatrix{
	H_\ast (BV_1) \tensor H_\ast (BV_1)
	\ar[r]^-\cdot
	\ar[d]_{1\tensor i_\ast}
	&
	H_\ast (BV_1)
	\ar[d]^{i_\ast}
	\\
	H_\ast (BV_1) \tensor H_\ast (BSU(2))
	\ar[r]^-\cdot
	&
	H_\ast (BSU(2))
}\]
where the upper $\cdot$ refers to the product on $H_\ast(BV_1)$
and the lower one to the module structure. Explicitly,
$H_\ast BSU(2)$ is isomorphic as a $H_\ast BV_1$-module 
to the quotient of $H_\ast BV_1$ by the ideal generated by 
$x$ and $x^{[2]}$. In particular, we see that the operation
$\big(\alpha^{SU(2)}_1\big)^\sharp\big(x^{[n]}\big)$ 
is non-trivial precisely when $n \equiv 1$ mod 4.

\begin{proof}[Proof of Theorem~\ref{thm:su2calc}]
The proof shares many similarities with the slightly simpler
proof of Theorem~\ref{thm:tcalc} above.
Again, Proposition~\ref{prop:alphagk-comp} implies that
we may compute $\alpha^{SU(2)}_1$ by  a push-pull construction
in the diagram
\begin{equation}
\label{diag:push-pull-for-alpha-su2-1}
\vcenter{\xymatrix@!0@C=3.4em@R=9.3ex{
	&
	SU(2) \sslash V_1 \times SU(2)^{\{p\}}
	\ar[dl]_{!}^{\pi_1}
	&&&&
	SU(2)^{V_1}\sslash V_1 \times SU(2)^{\{p,q\}}
	\ar[llll]^-\homot_-\eta
	\ar[dr]_{\pi_2}
	\\
	\pt \sslash V_1 \times SU(2)^{\{p\}}
	&&&&&&
	\pt \sslash SU(2)^{\{q\}}
}} 
\end{equation}
where the $V_1\times SU(2)^{\{p\}}$-action on $SU(2)$ is given by
\[
	(\varepsilon, g_p) \cdot g = g_p^{} g^\varepsilon g_p^{-1},
\]
where the $V_1\times SU(2)^{\{p,q\}}$-action on $SU(2)^{V_1}$
is given by 
\[
	(\varepsilon,g_p,g_q) \cdot (g_{-1},g_{1}) 
	= 
	(g_p^{} g_{-\varepsilon}^{} g_q^{-1},
	 g_p^{} g_\varepsilon^{} g_q^{-1}),
\]
where $\eta$ is induced by the map
\[
	SU(2)^{V_1} \longto SU(2),
	\quad
	(g_{-1},g_1) \longmapsto g_{-1}^{} g_1^{-1}
\]
and the projection
$V_1\times SU(2)^{\{p,q\}} \to V_1\times SU(2)^{\{p\}}$,
and where $\pi_1$ and $\pi_2$ are induced by the 
evident projections of spaces and groups.
Again, diagram~\eqref{diag:push-pull-for-alphat1}
is obtained from diagram~\eqref{diag:push-pull-for-alphagk}
by taking $G = SU(2)$ and $k=1$ and replacing the
space $SU(2)^{V_1}/\Delta SU(2)$
by $SU(2)$ using the homeomorphism
\[
	SU(2)^{V_1}/\Delta SU(2) \xto{\ \isom\ } SU(2),
	\quad
	(g_{-1},g_1)\Delta SU(2) \longmapsto g_{-1}^{} g_1^{-1}.
\]

By \cite[Lemma~8.4]{HL}, we may compute the umkehr map
$\pi_1^!$ using the Serre spectral sequence for $\pi_1$.
More precisely, the map $\pi_1^!$ is given by the composite
\begin{multline}
\label{comp:pi1-umkehr}
	H_n(\pt\sslash V_1\times SU(2)^{\{p\}})
	\xto{\ \isom\ }
	H_n(\pt\sslash V_1\times SU(2)^{\{p\}};\,H_3 SU(2))
	=
	E^2_{n,3}
	\\
	\longtwoheadto
	E^\infty_{n,3}
	\longincl
	H_{n+3} (SU(2)\sslash V_1\times SU(2)^{\{p\}}).
\end{multline}
Here $E^2$ and $E^\infty$ refer to pages in the Serre
spectral sequence of $\pi_1$, the first map 
is induced by the fundamental class of $SU(2)$,
and the last two maps arise from the fact
that we are working with the top row of the
spectral sequence.

The assignments $\pt \mapsto \pm I \in SU(2)$ define sections 
\[
	s_{\pm 1} 
	\colon 
	\pt \sslash V_1\times SU(2)^{\{p\}}
	\longto
	SU(2)\sslash V_1\times SU(2)^{\{p\}}.
\]
Since the spectral sequence has only two non-zero rows,
the existence of these sections implies that the Serre spectral
sequence for $\pi_1$ collapses at the $E^2$-page. Thus 
the short exact sequence
\[\xymatrix{
	0
	\ar[r]
	&
	E^\infty_{n,3}
	\ar[r]
	&
	H_{n+3} (SU(2)\sslash V_1\times SU(2)^{\{p\}})
	\ar[r]
	&
	E^\infty_{n+3,0}
	\ar[r]
	&
	0
}\]
and the description \eqref{comp:pi1-umkehr} for $\pi_1^!$ 
give the following short exact sequence:
\begin{multline}
\label{ses:pi1umk}
	0
	\longto
	H_{n} (\pt \sslash V_1\times SU(2)^{\{p\}})
	\xto{\ \pi_1^!\ }
	H_{n+3} ( SU(2)\sslash V_1\times SU(2)^{\{p\}})
	\\
	\xto{\ (\pi_1)_\ast\,}
	H_{n+3} (\pt \sslash V_1\times SU(2)^{\{p\}})
	\longto
	0.
\end{multline}

To proceed with our analysis of $\pi_1^!$, it is
helpful to dualize and work with the dual map
\[
	(\pi_1)_! 
	\colon 
	H^{*} ( SU(2)\sslash V_1\times SU(2)^{\{p\}})
	\longto
	H^{*-3} ( \pt \sslash V_1\times SU(2)^{\{p\}})
\]
instead. In general, 
\newcounter{dummy}\refstepcounter{dummy}\label{p:modstr}
if $\pi\colon M\to B$ is a closed
fibrewise manifold of fibre dimension $d$, then the dual
$\pi_!\colon H^\ast M \to H^{*-d} B$ of the umkehr
map $\pi^! \colon H_\ast B \to H_{\ast+d} M$ 
is a map of $H^\ast B$-modules where $H^\ast M$ is made
into a $H^\ast B$-module via the map $\pi^\ast$.
This follows from the pullback diagram
\[\xymatrix{
	M
	\ar[r]^-{(\pi,1)}
	\ar[d]_\pi
	&
	B\times M
	\ar[d]^{1\times \pi}
	\\
	B
	\ar[r]^-\Delta
	&
	B\times B
}\]
by using the compatibility of the umkehr maps with pullbacks
and direct products.
Dualizing the short exact sequence \eqref{ses:pi1umk},
we therefore obtain a short exact sequence
\begin{multline}
\label{ses:pi1umk2}
	0
	\longto
	H^{\ast} (\pt \sslash V_1\times SU(2)^{\{p\}})
	\xto{\ \pi_1^*\ }
	H^{\ast} ( SU(2)\sslash V_1\times SU(2)^{\{p\}})
	\\
	\xto{\ (\pi_1)_!\, }
	H^{*-3} (\pt \sslash V_1\times SU(2)^{\{p\}})
	\longto
	0.
\end{multline}
of $H^\ast(\pt\sslash V_1\times SU(2)^{\{p\}})$-modules.

The map 
\begin{equation}
\label{map:sp1psm1} 
	s_1^\ast + s_{-1}^\ast 
	\colon 
	H^{\ast} ( SU(2)\sslash V_1\times SU(2)^{\{p\}})
	\longto
	H^{\ast} ( \pt\sslash V_1\times SU(2)^{\{p\}})
\end{equation}
is a map of $H^\ast(\pt\sslash V_1\times SU(2)^{\{p\}})$-modules
satisfying 
$(s_1^\ast + s_{-1}^\ast) \circ \pi_1^\ast = \id+\id =0$,
so from the short exact sequence \eqref{ses:pi1umk2}
we can deduce that 
\[
	s_1^\ast + s_{-1}^\ast = \phi \circ (\pi_1)_!
\]
for some $H^\ast(\pt\sslash V_1\times SU(2)^{\{p\}})$-linear map 
\[
	\phi 
	\colon 
	H^\ast(\pt\sslash V_1\times SU(2)^{\{p\}})
	\longto
	H^{\ast+3}(\pt\sslash V_1\times SU(2)^{\{p\}}).
\]
Making use of the evident homeomorphism
\[
	\kappa 
	\colon 
	BV_1 \times BSU(2) 
	\xto{\ \isom\ }
	\pt \sslash V_1\times SU(2)^{\{p\}},
\]
we see that the cohomology 
$H^\ast(\pt\sslash V_1\times SU(2)^{\{p\}})$ is isomorphic
to a polynomial algebra generated by a class 
$w_1$ of degree 1 coming from the $BV_1$ factor 
and a class $u_4$ of degree 4 coming from the $BSU(2)$ factor.
Since the map $\phi$ is 
${H^\ast(\pt\sslash V_1\times SU(2)^{\{p\}})}$-linear, 
we have 
\[
	\phi(z) = z \phi(1)
\]
for all $z\in {H^\ast(\pt\sslash V_1\times SU(2)^{\{p\}})}$.
Lemma~\ref{lm:sp1psm1} below implies that the map $\phi$
is non-trivial, so we must have $\phi(1) = w_1^3$,
since $w_1^3$ is the only non-trivial element in 
$H^3(\pt\sslash V_1\times SU(2)^{\{p\}})$.

Let $\psi$ be the ``division by $w_1^3$ map''
\[
\begin{gathered}
 	\psi
	\colon 
	H^\ast(\pt\sslash V_1\times SU(2)^{\{p\}})
	\longto
	H^{\ast-3}(\pt\sslash V_1\times SU(2)^{\{p\}}),
	\\
	\qquad\qquad\,
	w_1^k u_4^l 
	\longmapsto \begin{cases}
					w_1^{k-3} u_4^l & \text{if $k \geq 3$}
					\\
					0 & \text{otherwise.} 
				\end{cases}
\end{gathered}
\]
Then $\psi\phi = \id$, and we obtain the formula
$(\pi_1)_! = \psi\phi (\pi_1)_! = \psi (s_1^\ast + s_{-1}^\ast).$
Dualizing, we see that the composite map 
\[\xymatrix@C-0.8em{
	H_\ast(BV_1\times BSU(2))
	\ar[r]^-{\kappa_\ast}_-\isom
	&
	H_\ast(\pt \sslash V_1\times SU(2)^{\{p\}})
	\ar[r]^-{\pi_1^!}
	&
	H_{\ast+3} ( SU(2)\sslash V_1\times SU(2)^{\{p\}})
}\]
is given by 
\begin{equation}
\label{eq:pi1umk-formula}
	\pi_1^!\kappa_\ast (x^{[n]} \times b) 
	=
	((s_1)_\ast + (s_{-1})_\ast)\kappa_\ast (x^{[n+3]}\times b).
\end{equation}

The maps
\[
	\tilde s_1, \tilde s_{-1} 
	\colon 
	\pt \sslash V_1 \times SU(2)^{\{p\}}
	\longto
	SU(2)^{V_1} \sslash V_1 \times SU(2)^{\{p,q\}}
\]
induced by the assignments
$\pt \mapsto (I,I), (I,-I)\in SU(2)^{V_2}$
and the homomorphisms
\[
	V_1 \times SU(2)^{\{p\}} \longto V_1 \times SU(2)^{\{p,q\}},
	\quad
	(\varepsilon,g_p) 
	\longmapsto 
	(\varepsilon,g_p,g_p), (\varepsilon,g_p, \varepsilon g_p)
\]
give lifts of the respective maps $s_1$ and $s_{-1}$
through the map $\eta$. The composite 
\[
	\pi_2\circ \tilde s_{\pm 1}
	\colon
	\pt \sslash V_1 \times SU(2)^{\{p\}}
	\longto
	\pt \sslash SU(2)^{\{q\}}
\]
is the map induced by the projection 
$\pr \colon V_1 \times SU(2) \to SU(2)$
in the case of $\tilde s_1$
and the map induced by $\nu$ 
in the case of $\tilde s_{-1}$.
Using the formula \eqref{eq:pi1umk-formula}
for $\pi_1^!$, we conclude that 
the operation $\alpha^{SU(2)}_1$
is the sum of the maps
\[
	H_\ast (BV_1) \tensor H_{\ast} (BSU(2))
	\longto
	H_{\ast+3} (BSU(2))
\]
given by 
\[	
	x^{[n]} \tensor b \longmapsto \pr_\ast(x^{[n+3]} \times b)
	\quad\text{and}\quad
	x^{[n]} \tensor b \longmapsto \nu_\ast(x^{[n+3]} \times b).
\]
The claim follows.
\end{proof}

In the proof of Theorem~\ref{thm:su2calc}, we made use
of the following lemma.
\begin{lemma}
\label{lm:sp1psm1}
The map
\[
	s_1^\ast + s_{-1}^\ast 
	\colon 
	H^{\ast} ( SU(2)\sslash V_1\times SU(2)^{\{p\}})
	\longto
	H^{\ast} ( \pt\sslash V_1\times SU(2)^{\{p\}})
\]
in the proof of Theorem~\ref{thm:su2calc} is non-zero.
\end{lemma}
\begin{proof}
Let us interpret $SU(2)$ as the group of unit quaternions,
and define 
\[
\begin{aligned}
	D^3_+ &= \{q\in SU(2)\,|\, \Re q > -1/2\}
	\\
	D^3_- &= \{q\in SU(2)\,|\, \Re q < 1/2\}.
\end{aligned}
\]
Then $D^3_+ \cup D^3_- = SU(2)$, the spaces $D^3_\pm$
are invariant under the $V_1\times SU(2)^{\{p\}}$-action
on $SU(2)$, and there exist
$V_1\times SU(2)^{\{p\}}$-equivariant
deformation retractions of
$D^3_+$ onto $\{1\} \subset SU(2)$,
$D^3_-$ onto $\{-1\} \subset SU(2)$
and
$D^3_+\cap D^3_-$ onto the subspace
\[
	S^2_0 =  \{q\in SU(2)\,|\, \Re q = 0\} \isom S^2
\]
of $SU(2)$.
From the Mayer--Vietoris sequence of 
the subspaces
$D^3_+\sslash V_1 \times SU(2)^{\{p\}}$ 
and 
$D^3_-\sslash V_1 \times SU(2)^{\{p\}}$ 
of 
$SU(2) \sslash V_1 \times SU(2)^{\{p\}}$ 
we therefore obtain an exact sequence
\begin{multline}
\label{eq:sp1sm2-seq}
	 H^3(SU(2)\sslash V_1 \times SU(2)^{\{p\}})
	 \xto{\ {(s_1^\ast, s_{-1}^\ast)}\ }
	 H^3(\pt\sslash V_1 \times SU(2)^{\{p\}})^{\oplus 2}
	 \\
	 \xto{\qquad\ \ }
	 H^3(S^2_0\sslash V_1 \times SU(2)^{\{p\}}).
\end{multline}
Observe that 
\[
	S^2_0\sslash V_1 \times SU(2)^{\{p\}}
	\isom
	(S^2_0\sslash V_1)\sslash SU(2)^{\{p\}}.
\]
The $V_1$-action on $S^2_0$ is given by the map $q\mapsto -q$,
so we have $S^2_0\sslash V_1 \homot \R P^2$.
From the Serre spectral sequence of the fibration
$(S^2_0\sslash V_1)\sslash SU(2)^{\{p\}} \to \pt \sslash SU(2)^{\{p\}}$
we now deduce that 
\[
	H^3 (S^2_0\sslash V_1 \times SU(2)^{\{p\}})
	\isom 
	H^3 ((S^2_0\sslash V_1)\sslash SU(2)^{\{p\}}) 
	=
	0
\]
Thus the map
${(s_1^\ast, s_{-1}^\ast)}$
in the exact sequence
\eqref{eq:sp1sm2-seq}
is an epimorphism. Consequently the map
\[
	s_1^\ast + s_{-1}^\ast 
	\colon 
	H^3 ( SU(2)\sslash V_1\times SU(2)^{\{p\}})
	\longto
	H^3 ( \pt\sslash V_1\times SU(2)^{\{p\}})
\]
is also an epimorphism. 
But the group $H^3 ( \pt\sslash V_1\times SU(2)^{\{p\}})$
is non-trivial: in the notation of 
the proof of Theorem~\ref{thm:su2calc},
it contains the non-trivial element $w_1^3$.
Thus  the claim follows.
\end{proof}

\section{Applications}
\label{sec:applications}

The computations made in the preceding sections show the 
existence of a large number of classes $a\in H_\ast(B\Sigma_n)$
for which the operation
\[
	\Phi^G(S_n/B\Sigma_n)^\sharp(a) 
	\colon 
	H_\ast(BG)
	\longto 
	H_{\ast + |a| + \dim(G)(n-1)}(BG)
\]
is non-trivial for some compact Lie group $G$. More generally,
the computations provide lots of examples of classes 
$a = a_1 \times \cdots \times a_r 
\in H_\ast(B\Sigma_{n_1} \times\cdots \times B\Sigma_{n_r})$
for which the operation
\begin{equation}
\label{eq:compop}
 	\Phi^G(S_{n_1,\ldots,n_r}/B\Sigma_{n_1,\ldots,n_r})^\sharp(a)
	= 
	\Phi^G(S_{n_1}/B\Sigma_{n_1})^\sharp(a_1)
	\circ \cdots \circ
	\Phi^G(S_{n_r}/B\Sigma_{n_r})^\sharp(a_r)
\end{equation}
is non-trivial for some compact Lie group $G$. Here 
$\Sigma_{n_1,\ldots,n_r} = \Sigma_{n_1} \times \cdots \times \Sigma_{n_r}$
and 
$S_{n_1,\ldots,n_r}
= 
S_{n_1} \ucirc \cdots \ucirc S_{n_r}$.
For example, we have the following result for $G=\Z/2$.
\begin{proposition}
\label{prop:example-of-non-triv-op}
Let $u_1,\ldots, u_k$ be positive integers no two of which
have a $1$ in common in their binary expansions, and let 
$f \colon \{1,\ldots,k\} \to \{1,\ldots,r\}$
be a surjective function. Let
\[
	n_i = 2^{|f^{-1}(i)|} 
	\quad
	\text{ and }
	\quad
	a_i 
	= 
	\bigcirc_{j \in f^{-1}(i)} E_{u_j} \in H_*(B\Sigma_{n_i})
\]
for $i=1,\ldots,r$, and let $a = a_1 \times \cdots \times a_r$.
Then the operation \eqref{eq:compop} is non-trivial for $G=\Z/2$. 
\end{proposition}
\begin{proof}
Let $x$ be the generator of $H_1(B\Z/2)$.
Recalling Definition~\ref{def:alphagk}
and using Propositions~\ref{prop:iotaformula}
and \ref{prop:z2-nontrivial-ops},
we see that the operation
\eqref{eq:compop}
is simply the map $H_\ast(B\Z/2) \to H_{\ast+u_1+\cdots +u_k}(B\Z/2)$
given by multiplication by $x^{[u_1+\cdots+u_k]}$.
\end{proof}
The purpose of this section is to show 
that classes $a$ for which 
the operation \eqref{eq:compop}
is non-trivial for a suitable 
compact Lie group $G$
give rise to non-trivial classes in 
the homology of certain highly interesting groups:
first, the ordinary
homology of the holomorph $\Hol(F_N) = F_N\rtimes \Aut(F_N)$;
second, the twisted homology $H_\ast(B\Aut(F_N);\,\tilde \F_2^N)$
of $\Aut(F_N)$
where $\tilde \F_2^N$ denotes $\F_2^N$ equipped with the tautological
$\Aut(F_N)$-action;
and third, the ordinary homology of the affine groups 
$\Aff_N(\Z) = \Z^N\rtimes GL_N(\Z)$
and 
$\Aff_N(\F_2) = \F_2^N\rtimes GL_N(\F_2)$.
Here and for the rest of the section $N = \sum_{i=1}^r (n_i -1)$.
The families of non-trivial homology classes we construct
are exhibited in 
Corollaries~\ref{cor:hol-app},
\ref{cor:non-stab},
\ref{cor:aut-app},
\ref{cor:hol-fam2},
\ref{cor:aff-app},
and
\ref{cor:aff-app2} below.

\subsection{Stability results}
\label{subsec:stability-results}

In preparation for the discussion of the applications of our 
computations to the homology of $\Aut(F_n)$, $\Hol(F_n)$
and $\Aff_n(R) = R^n \rtimes GL_n(R)$
for $R=\Z$ and $\F_2$,
we will now recall stability results concerning the 
homology of these groups.
We will start with the result relevant to the groups $\Aut(F_n)$
and $\Hol(F_n)$, which we will consider in the context of 
a larger family of groups $A(\Gamma,L)$ defined below.

 Let $\Gamma$ be a connected 
finite graph, by which we mean a connected finite CW complex
of dimension $\leq 1$, and let $u\colon L \to \Gamma$
be an injective map from a finite set $L$. We think of $u$ 
as providing labels for a finite set of distinguished points
on $\Gamma$, namely the points in the image of $u$. Omitting
$u$ from notation, we denote by $H(\Gamma,L)$ the 
space of self homotopy equivalences of $\Gamma$ fixing the 
distinguished points, and set $A(\Gamma,L) = \pi_0 H(\Gamma,L)$.
Composition of maps makes $A(\Gamma,L)$ into a group.
The group $\Aut(F_n)$ is realized in this way as 
$A(\bigvee^n S^1,\pt)$ where
the basepoint of $\bigvee^n S^1$ is the 
sole distinguished point;
explicitly, an isomorphim
\[
	\Aut(F_n) \xto{\ \isom\ } A(\textstyle\bigvee^n S^1,\pt)
\]
is given by the map sending an automorphism 
$\theta$ to the component of the map
$\bigvee^n S^1 \to \bigvee^n S^1$ which on the $i$-th
wedge summand is given by $\theta(x_i)|_{l_1,\ldots,l_n}$.
Here $x_1,\ldots, x_n$ denotes the basis of $F_n$, $l_j$
denotes the inclusion of the $j$-th  wedge summand into $\bigvee^n S^1$,
and $v|_{l_1,\ldots,l_n}$ for $v\in F_n$ denotes the result of 
substituting $l_j$ for $x_j$ in $v$ for all $1\leq j\leq n$.
Similarly, the group $\Hol(F_n)$ is realized as $A(T_n,\{p,q\})$
where $T_n$ is the graph depicted below.
\begin{equation}
\label{pic:tndef}
T_n = 
\left(
\begin{tikzpicture}[
		baseline=15,
		scale=0.03,
		->-/.style={	decoration={	markings,
		  	mark=at position 0.5 with {\arrow{>}}
			},postaction={decorate}},
		->--/.style={	decoration={	markings,
		  	mark=at position 0.46 with {\arrow{>}}
			},postaction={decorate}},
		-->-/.style={	decoration={	markings,
		  	mark=at position 0.58 with {\arrow{>}}
			},postaction={decorate}}
	]
	\clip (-50, -20) rectangle (130, 70);
	\path[ARC, fill=black] (0,0) circle (3);
	\path[ARC, fill=black] (120,0) circle (3);
	\path[ARC,->-] (0,0) -- (120,0) node [midway, above] {$c$};
	\node [below] at (0,0) {$p$};
	\node [below] at (120,0) {$q$}; 
	\path[ARC,->--] (0,0) .. controls (-20,80) and (-80,45) 
						.. (0,0) node [midway, above] {$l_1$};
	\node at (1.5,35) {$\cdots$};
	\path[ARC,-->-] (0,0) .. controls (80,45) and (20,80)
						.. (0,0)  node [midway, above] {$l_n$};
\end{tikzpicture}
\right)
\end{equation}
We will often regard $T_n$ as an h-graph cobordism
$T_n\colon \pt \hto \pt$, where the incoming point is 
given by $p$ and the outgoing point by $q$.
An isomorphism 
\begin{equation}
\label{iso:psi}
	\psi \colon\Hol(F_n) \xto{\ \isom\ } A(T_n,\{p,q\}) = \pi_0\hAut(T_n)
\end{equation}
is given by the map sending an element $(w,\theta)\in F_n\rtimes \Aut(F_n)$
to the component of the map $T_n \to T_n$ whose restriction 
to the edge $l_i$ in \eqref{pic:tndef} is the path
$\theta(x_i)|_{l_1,\ldots,l_n}$
and whose restriction to the edge $c$ in \eqref{pic:tndef} 
is the path  $c \cdot w|_{l_1,\ldots,l_n}$.

By the \emph{rank} of $\Gamma$ we mean the first Betti number of $\Gamma$.
If $\Gamma'$ is another connected finite graph equipped with 
and injection $L \to \Gamma'$ and $\rank(\Gamma) = \rank(\Gamma')$,
then $\Gamma$ and $\Gamma'$ are homotopy equivalent under $L$.
Conjugation by a fixed homotopy equivalence then yields an isomorphism
$A(\Gamma,L) \isom A(\Gamma',L)$ which up to conjugacy is independent
of the homotopy equivalence chosen. In particular, we obtain a
well-defined isomorphism $H_\ast BA(\Gamma,L) \isom H_\ast BA(\Gamma',L)$
which is independent of any choice.

We will consider the following
three types of maps between the groups $A(\Gamma,L)$.
\begin{enumerate}
\label{list:map-types}
\renewcommand{\theenumi}{\Alph{enumi}}
\item \label{type:forget} 
	Given $l\in L$, we have a map 
	$A(\Gamma,L) \to A(\Gamma,L{\setminus}\{l\})$
	which forgets the distinguished point labeled by $l$.
	If $l$ labels a leaf, we will also consider a variant (A')
	which follows the above map with the isomorphism
	given by collapsing the edge corresponding to $l$.
\item \label{type:whisker}  
	Given $l \in L$, we have a map
	$A(\Gamma,L) \to A(\Gamma \cup_{\{l\}} I, L\sqcup \pt)$
	which attaches a whisker to $\Gamma$ at the point labeled
	by $l$. The free endpoint of the whisker provides 
	a new distinguished point.
\item \label{type:join} 
	Given  distinct elements $l_1,l_2 \in L$, we have a map
	$A(\Gamma,L) \to A(\Gamma/{\sim}, L/{\sim})$
	which identifies the points of $\Gamma$ labeled by 
	$l_1$ and $l_2$ and identifies these two labels.
\end{enumerate}
Maps of the aforementioned types
are known to give isomorphisms on homology
in a range of degrees, the \emph{stable range}, which
depends on the rank of $\Gamma$.
The formulation of the stability result given below is
close to the one given in \cite[Theorem~1.4]{Galatius}.
\begin{theorem}[\cite{HV-stab,HV-stab-erratum}]
\label{thm:stability}
Let $A$ be an abelian group, and 
let $n$ be the rank of $\Gamma$. Then maps of type 
\eqref{type:whisker} and \eqref{type:join}
induce isomorphisms upon application of $H_k(-;\,A)$ 
whenever $n > 2k+1$,
while maps of type \eqref{type:forget}
induce isomorphisms whenever $n > 2k+1$
if $|L| > 1$ and whenever $n > 2k+3$ if $|L|=1$. \qed
\end{theorem}

In particular, the preceding theorem implies that the
following maps induce an isomorphism on $k$-th homology 
groups if $n>2k+1$:
\begin{enumerate}
\item the inclusion $\Aut(F_n) \incl \Aut(F_{n+1})$
	sending an automorphism $\varphi$ of $F_n$ to the 
	automorphism of $F_{n+1}$ given by $\varphi$ on the
	first $n$ basis elements and fixing the last basis element;
\item the inclusion $\Hol(F_n) \incl \Hol(F_{n+1})$
	induced by the inclusions $F_n \incl F_{n+1}$
	and $\Aut(F_n) \incl \Aut(F_{n+1})$;
\item the inclusion $\Aut(F_n) \incl \Hol(F_n)$; and 
\item the quotient homomorphism $\Hol(F_n) \to \Aut(F_n)$.
\end{enumerate}
Indeed, on the level of graphs, the first map
is realized by 
the map $A(\bigvee^n S^1,\pt)\to A(\bigvee^{n+1} S^1,\pt)$
given by attaching a new copy of $S^1$, which amounts to 
the composite of a map of type \eqref{type:whisker}
and a map of type \eqref{type:join}.
The second map is realized similarly as a composite of maps 
of type \eqref{type:whisker}
and \eqref{type:join},
while the third map is 
realized by a map 
of type \eqref{type:whisker} and the 
fourth map
by a map of type (\ref{type:forget}').
We call the maps 
$H_\ast B\Aut(F_n) \to H_\ast B\Aut(F_{n+1})$
and 
$H_\ast B\Hol(F_n) \to H_\ast B\Hol(F_{n+1})$
induced by the  inclusions \emph{stabilization maps}.

In the stable range $n > 2k+1$, the homology of $\Aut(F_n)$ 
and $\Hol(F_n)$ is completely understood:
\begin{theorem}[{\cite[Theorem~1.1]{Galatius}}]
\label{thm:galatius}
For any abelian group $A$,
the inclusion of $\Sigma_n$ into $\Aut(F_n)$ as the 
automorphisms 
permuting the basis elements of $F_n$ induces an isomorphism
upon application of $H_k(-;\,A)$ for $n>2k+1$. \qed
\end{theorem}
\noindent
On the other hand, outside the above stable range, 
the homology of $\Aut(F_n)$ and $\Hol(F_n)$ 
remains poorly understood. 
For an up-to-date summary of what is known about the 
homology of $\Aut(F_n)$ with rational coefficients,
we refer the reader to the introduction 
of \cite{CHKV}. Some information about torsion 
in the homology of $\Aut(F_n)$ is  
available via computations of the $p$-torsion
in the integral Farrell cohomology of $\Aut(F_n)$ for 
$n\in\{p-1,p,2(p-1)\}$ for primes $p\geq 3$ 
and for $n\in\{p+1,p+2\}$ for primes $p \geq 5$.
See \cite{GloverMislinVoon,Chen,JensenAut}.
The $k$-th homology groups of the holomorphs $\Hol(F_n)$
have been computed by 
Jensen \cite{JensenHol} 
with mod $p$ coefficients for $p$ an odd prime for $k \leq 2$
and with rational coefficients for $k \leq 5$.
Jensen also supplies a computation of 
the twisted homology groups
$H_k (B\Aut(F_n);\, \tilde R^n)$
for $R = \Q$ and $k\leq 4$,
while Satoh \cite{Satoh1,Satoh2}
has computed these groups when $R = \Z$, $k=1$ and $n\geq 2$
and when $R = \Z[1/2]$, $k=2$ and $n\geq 6$.
Here $\tilde R^n$ for a ring $R$ 
denotes  $R^n$ with the action of $\Aut(F_n)$
given by the homomorphism $\Aut(F_n) \to GL_n(\Z) \to GL_n(R)$.

We now turn to the stability result relevant to the 
homology of affine groups.
The following result can be read off from \cite[Theorem~4.8]{linStability};
in the notation of \cite{linStability}, we have $\Aff_n(R) = G^{PQ}_{n+1}$
where $P=\emptyset$ and $Q = \{1\}$.
\begin{theorem}
\label{thm:glstab}
Let $A$ be an abelian group, and let $R$ be a principal ideal
domain. Then the inclusions 
\[
	GL_n(R) \longincl GL_{n+1}(R)
	\quad\text{and}\quad
	GL_n(R) \longincl \Aff_n(R)
\]
and the inclusion
\[
	\Aff_n(R) \longincl \Aff_{n+1}(R)
\]
induced by the inclusions $R^n \incl R^{n+1}$ and 
$GL_n(R) \incl GL_{n+1}(R)$
induce isomorphisms upon application of $H_k(-;\,A)$
whenever $n \geq 2k+1$. \qed
\end{theorem}
\noindent
In particular, as $n$ tends to infinity, 
the homology of $\Aff_n(R)$ and 
the homology of $GL_n(R)$ stabilize to a common
value, namely the homology of the infinite general linear
group $GL(R)=\colim_n GL_n(R)$.
Again, we call the maps induced by the 
inclusions
$\Aff_n(R) \incl \Aff_{n+1}(R)$
and
$GL_n(R) \incl GL_{n+1}(R)$
on homology \emph{stabilization maps}.
The following two theorems describe the 
stable homology in the cases 
$A= \F_2$ and $R=\Z$ or $\F_2$.
\begin{theorem}[{\cite{MitchellGL}, \cite[Theorem~1]{AMNY}}]
The homology 
$H_\ast(BGL(\Z);\,\F_2)$
is isomorphic to $H_\ast(BO;\,\F_2) \tensor H_\ast(SU;\,\F_2)$
as a Hopf algebra.\qed
\end{theorem}
\begin{theorem}[{\cite[section~11, Corollary~2]{QuillenFiniteFields}}]
\label{thm:quillen-calc}
$H_\ast(BGL(\F_2);\,\F_2) = \F_2$. \qed
\end{theorem}
Again, 
in contrast with the complete  information 
we have concerning the stable values,
outside the stable range the mod 2 homology of $\Aff_n(\Z)$
and $\Aff_n(\F_2)$ remains poorly understood.
Indeed, the same is true of the closely related groups 
$GL_n(\Z)$ and $GL_n(\F_2)$.
Quillen's computations \cite{QuillenFiniteFields} 
give complete information on the cohomology
of $GL_n(k)$ with mod $p$ coefficients when 
$k$ is a finite field of characteristic prime to $p$,
but yield little information when $k=\F_2$ and $p=2$.
For computations of the integral cohomology of $GL_n(\Z)$
and $SL_n(\Z)$ for small $n$ and modulo small primes, see 
\cite[Exercise II.7.3]{Brown},
\cite{SouleSL3},
\cite{LeeSzczarba}, and
\cite{EVGS}.
The mod p Farrell cohomology of $GL_n(\Z)$ for an odd prime $p$
has been computed for $p-1 \leq n \leq 2p-3$ 
by Ash \cite{AshFarrellCohom} 
and studied for $n=2p-2$ by Manjrekar 
\cite{Manjrekar}, who has also 
computed the mod 3 cohomology of $GL_4(\Z)$
in dimensions $>3$ \cite{ManjrekarGL4}.
For an in-depth discussion of the homology of 
linear groups in general, we refer the reader to the monograph
\cite{Knudson}.

\subsection{Other preliminaries}

Our aim in this subsection is to relate the family 
$S_{n_1,\ldots,n_r}$
of h-graph cobordisms 
over $B\Sigma_{n_1,\ldots,n_r}$
to a certain family
$U_N$
of h-graph cobordisms over $B\Hol(F_N)$
and to obtain a useful description of the 
operation $\Phi^G(U_N/B\Hol(F_N))$.

\begin{definition}
We define $U_n/B\Hol(F_n)\colon \pt\hto\pt$ to be the 
family of h-graph cobordisms obtained by performing the
mapping cylinder construction of 
Remark~\ref{rk:families-from-borel-construction}
to the Borel construction 
$E\Hol(F_n) \times_{\Hol(F_n)} B\Pi_1(T_n,\{p,q\})$,
where
$T_n$ is as depicted in \eqref{pic:tndef}
and $\Hol(F_n)$ acts on the h-graph cobordism
$B\Pi_1(T_n,\{p,q\}) \colon \{p\} \hto \{q\}$
via the isomorphism $\psi$ of \eqref{iso:psi}.
Observe that the family  $U_n/B\Hol(F_n)$ is then 
obtained from $B\Pi_1(T_n,\{p,q\})$ by 
the Borel construction in the sense of
Definition~\ref{def:families-from-borel-construction}.
\end{definition}

The h-graph cobordism 
\begin{equation}
\label{hgcob:sn1nr} 
	\hat S_{n_1,\ldots,n_r} 
	= 
	\hat S_{n_1}\circ \cdots \circ \hat S_{n_r}
	\colon 
	\pt\longhto \pt
\end{equation}
is homotopy equivalent (relative to the two endpoints)
to $T_N$. Conjugation by a fixed homotopy equivalence
$\hat S_{n_1,\ldots,n_r} \to T_N$ therefore
gives us an isomorphism
\begin{equation}
\label{iso:phi]}
	\phi 
	\colon 
	\pi_0\hAut(\hat S_{n_1,\ldots,n_r})
	\xto{\ \isom\ }
	\pi_0\hAut(T_N)
\end{equation}
which up to conjugacy is independent of the 
homotopy equivalence chosen.
Let 
\[
	U'_N/B\pi_0\hAut(\hat S_{n_1,\ldots,n_r}) 
	\colon
	\pt \longhto \pt
\]
be the pullback of the family $U_N/B\Hol(F_N)$
along the homeomorphism
\begin{equation}
\label{homeom:basespaces}
	B\pi_0\hAut(\hat S_{n_1,\ldots,n_r})\xto{\ \isom\ } B\Hol(F_N)
\end{equation}
induced by $\phi$ and the isomorphism $\psi$
of \eqref{iso:psi}.
We then have a 2-cell \cite[Definition 2.10]{HL}
\begin{equation}
\label{2cell:1} 
\xymatrix{
	U'_N/B\pi_0\hAut(\hat S_{n_1,\ldots,n_r}) 
	\ar@{=>}[r]
	&
	U_N/B\Hol(F_N)
}
\end{equation}
which on base spaces is given by the homeomorphism 
\eqref{homeom:basespaces}.
Observe that the 
family $S_{n_1,\ldots,n_r}/B\Sigma_{n_1,\ldots,n_r}$
is obtained by the Borel construction from the 
h-graph cobordism $\hat S_{n_1,\ldots,n_r}$ of 
\eqref{hgcob:sn1nr} equipped with the evident 
$\Sigma_{n_1,\ldots,n_r}$-action
derived from the $\Sigma_{n_i}$-actions on the factors $\hat S_{n_i}$.
Furthermore, observe that the family of h-graph cobordisms
$U'_N/B\pi_0\hAut(\hat S_{n_1,\ldots,n_r})$ 
is obtained from the h-graph 
cobordism
\[
   B\Pi_1(\hat S_{n_1,\ldots,n_r}, \{p,q\}) 
   \colon 
   \{p\} \longhto \{q\}
\]
by the Borel construction with respect to 
the evident action of 
$\pi_0\hAut(\hat S_{n_1,\ldots,n_r})$
on 
$B\Pi_1(\hat S_{n_1,\ldots,n_r}, \{p,q\})$.
Here $p$ and $q$ denote the endpoints of $\hat S_{n_1,\ldots,n_r}$.
Thus the following proposition provides a 2-cell
\begin{equation}
\label{2cell:2}
\xymatrix{
	S_{n_1,\ldots,n_r}/B\Sigma_{n_1,\ldots,n_r}
	\ar@{=>}[r]
	&
	U'_N/B\pi_0\hAut(\hat S_{n_1,\ldots,n_r})
}
\end{equation}
which on base spaces is induced by the 
homomorphism
$\Sigma_{n_1,\ldots,n_r}
\to
\pi_0\hAut(\hat S_{n_1,\ldots,n_r})$
obtained from the 
$\Sigma_{n_1,\ldots,n_r}$-action on 
$\hat S_{n_1,\ldots,n_r}$.

\begin{proposition}
\label{prop:2cell}
Suppose $\hat S \colon P \hto Q$ is a positive h-graph cobordism
such that $\hat S$ 
is a finite CW complex of dimension $\leq 1$ and 
$P$ and $Q$ consist of 0-cells of~$\hat S$.
Assume further that each 1-cell of $\hat S$ is equipped with 
the choice of a characteristic map $I\to \hat S$.
Let $\Gamma$ be a discrete group acting on $\hat S$
by permuting the cells. Suppose this action 
respects the chosen characteristic maps and keeps $P$ and $Q$ pointwise fixed.
Assume $S/B\Gamma \colon P \hto Q$ is obtained from $\hat S$
by the Borel construction in the sense of 
Definition~\ref{def:families-from-borel-construction}.
Let $U/B\pi_0 \hAut(\hat S) \colon P \hto Q$ be a 
family of h-graph cobordisms obtained by the Borel construction
from the h-graph cobordism $B\Pi_1(\hat S, P\sqcup Q) \colon P\hto Q$
equipped with the evident $\pi_0 \hAut(\hat S)$-action.
Then there exists a 2-cell \cite[Definition 2.10]{HL}
$S/B\Gamma \Rightarrow U/B\pi_0 \hAut(\hat S)$
which is the identity on $P$ and $Q$ and 
which on base spaces is the map induced by the 
homomorphism $\Gamma \to \pi_0 \hAut(\hat S)$
given by the $\Gamma$-action on $\hat S$.
\end{proposition}
\begin{proof}
Suppose $X$ is a finite CW complex of dimension $\leq 1$,
and let $X^{(0)}$ denote the 0-skeleton of $X$.
Choosing characteristic maps for the 1-cells
of $X$, we obtain  a map $\beta_X\colon X \to B\Pi_1(X,X^{(0)})$
sending the 0-cells of $X$ to the corresponding
0-cells in $B\Pi_1(X,X^{(0)})$ and sending 
each 1-cell $c$ to the 1-cell of $B\Pi_1(X,X^{(0)})$
corresponding to the morphism of 
$\Pi_1(X,X^{(0)})$ represented by  
the characteristic map of $c$. 
The finite free groupoid
$\Pi_1(X,X^{(0)})$ has a basis consisting of
the homotopy classes of the characteristic maps 
for the 1-cells of $X$, as can be shown by 
induction on the number of 1-cells in $X$
using the compatibility of $\Pi_1$ with 
pushouts \cite[Theorem 17']{Higgins}.
Thus it follows from \cite[Lemma 7.38]{HL}
that the map $\beta_X$ is a homotopy equivalence.
Clearly $\beta_X$ is natural with respect to cellular
maps preserving characteristic maps for 1-cells.

In the context of the h-graph cobordism $\hat S$,
we obtain the zigzag
\[\xymatrix{
	\hat S 
	\ar[r]^-{\beta_{\hat S}}_-\homot
	&
	B\Pi_1(\hat S, \hat S^{(0)})
	&
	\ar[l]^-\homot
	B\Pi_1(\hat S, P\sqcup Q)
}\]
of homotopy equivalences which are 
$\Gamma$-equivariant maps under $P\sqcup Q$.
Here the second map is induced by 
the inclusion of $P\sqcup Q$ into $\hat S^{(0)}$.
Performing the Borel construction, we obtain 
the zigzag 
\begin{equation}
\label{zigzag:borel}
\xymatrix{
	E\Gamma \times_\Gamma \hat S 
	\ar[r]
	&
	E\Gamma \times_\Gamma B\Pi_1(\hat S, \hat S^{(0)})
	&
	\ar[l]
	E\Gamma \times_\Gamma B\Pi_1(\hat S, P\sqcup Q)
}
\end{equation}
of maps over $B\Gamma$ and under $(P\sqcup Q) \times B\Gamma$
which restrict to homotopy equivalences on fibres.

For a space $Z$ over $B$ and under $(P \sqcup Q)\times B$,
let us denote by $Z'$ the mapping cylinder of 
the map of $(P\sqcup Q) \times B$ into $Z$.
Then the inclusion of $(P\sqcup Q) \times B$
into $Z'$ is a closed fibrewise cofibration over $B$.
Moreover, by \cite[Proposition 1.3]{Clapp}
the map $Z' \to B$ is a fibration if 
the map $Z\to B$ is.
By \cite[Remark B.3]{HL}, the zigzag
\[\xymatrix{
	(E\Gamma \times_\Gamma \hat S)'
	\ar[r]
	&
	(E\Gamma \times_\Gamma B\Pi_1(\hat S, \hat S^{(0)}))'
	&
	\ar[l]
	(E\Gamma \times_\Gamma B\Pi_1(\hat S, P\sqcup Q))'
}\]
obtained from \eqref{zigzag:borel}
consists of homotopy equivalences
over $B\Gamma$ and under $(P\sqcup Q) \times B\Gamma$.
Choosing a homotopy inverse over $B\Gamma$ and 
under $(P\sqcup Q) \times B\Gamma$ for the 
second map, we obtain a 2-cell 
$(E\Gamma \times_\Gamma \hat S)'/B\Gamma 
\Rightarrow
(E\Gamma \times_\Gamma B\Pi_1(\hat S, P\sqcup Q))'/B\Gamma$
which is the identity on $P$ and $Q$ and base spaces.
Furthermore, the map $\Gamma \to \pi_0 \hAut(\hat S)$
given by the $\Gamma$-action induces a 2-cell 
\[\xymatrix{
	(E\Gamma \times_\Gamma B\Pi_1(\hat S, P\sqcup Q))'/B\Gamma
	\ar@{=>}[r] 
	&
	(E\pi_0 H\times_{\pi_0 H} B\Pi_1(\hat S, P\sqcup Q))'/B\pi_0 H
}\]
which is the identity on $P$ and $Q$. Here $H$ denotes $\hAut(\hat S)$.
As observed in Remark~\ref{rk:families-from-borel-construction},
$S$ is homotopy equivalent over $B\Gamma$ and under 
$(P \sqcup Q)\times B\Gamma$
to $(E\Gamma \times_\Gamma \hat S)'$, and similarly for
$U$ and 
$(E\pi_0 H\times_{\pi_0 H} B\Pi_1(\hat S, P\sqcup Q))'$,
so the claim follows.
\end{proof}

Composing the 2-cells \eqref{2cell:2} and \eqref{2cell:1},
we obtain a 2-cell
\[\xymatrix{
	S_{n_1,\ldots,n_r}/B\Sigma_{n_1,\ldots,n_r}
	\ar@{=>}[r]
	&
	U_N/B\Hol(F_N)
}\]
which on base spaces is induced by the composite homomorphism
\[\xymatrix{
	\zeta 
	\colon
	\Sigma_{n_1,\ldots,n_r}
	\ar[r]
	&
	\pi_0\hAut(\hat S_{n_1,\ldots,n_r})
	\ar[r]^-\phi_-\isom
	&
	\pi_0\hAut(T_N)
	\ar[r]^-{\psi^{-1}}_-\isom
	&
	\Hol(F_N)
}\]
where the first map is given by 
the action of 
$\Sigma_{n_1,\ldots,n_r}$
on 
$\hat S_{n_1,\ldots,n_r}$.
Thus the base change axiom of HHGFTs implies the 
following result.
\begin{theorem}
\label{thm:hol-app}
For any compact Lie group $G$,
the map 
\[
	\Phi^G(S_{n_1,\ldots,{n_r}}
			/B\Sigma_{n_1,\ldots ,n_r})^\sharp 
	\colon 
	H_{\ast-\mathrm{shift}}(B\Sigma_{n_1,\ldots,n_r}) \longto
	\Hom_\ast(H_\ast(BG),H_\ast(BG))
\]
factorizes as the composite
\newcommand{\entry}{H_{\ast-\mathrm{shift}}(B\Sigma_{n_1,\ldots,n_r})\;}
\[\xymatrix@!0@C=7.4em{
	*!R{\entry}
	\ar[r]^{\zeta_\ast}
	&
	*!L{\;H_{\ast-\mathrm{shift}}(B\Hol(F_N))}
	\\
	*!R{\phantom\entry}
	\ar[r]^{\Phi^G(U_N/B\Hol(F_N))^\sharp}
	&
	*!L{\;\Hom_\ast(H_\ast(BG),H_\ast(BG)).}
}\]
Here `shift' equals $\dim(G)(N-1)$.\qed
\end{theorem}
\noindent
Observe that the homomorphism $\zeta$ is, up to conjugacy,
independent of the choice of a homotopy equivalence
$\hat S_{n_1,\ldots,n_r} \to T_N$ involved in the 
construction of $\phi$, and hence in particular the 
induced map $\zeta_\ast$ on homology is independent of this choice.

In the remainder of this subsection, 
our goal is to 
use Proposition~\ref{prop:borel-opdesc} 
to develop a description of the operation 
$\Phi^G(U_n/B\Hol(F_n))$. 
For this end, we would like to
obtain more concrete descriptions of 
the spaces of functors 
\[
	\fun(\Pi_1(B\Pi_1(T_n,\{p,q\}),\{p,q\}),G)
	\quad\text{and}\quad
	\fun(\Pi_1(B\Pi_1(T_n,\{p,q\}),\{p\}),G).
\]
By \cite[Lemma~7.27]{HL}, there is
a natural isomorphism of finite free groupoids
\begin{equation}
\label{isom:pi1red1}
	\Pi_1(T_n,\{p,q\})
	\xto{\ \isom\ } 
	\Pi_1(B\Pi_1(T_n,\{p,q\}),\{p,q\})
\end{equation}
which is the identity on objects and sends each morphism 
to the homotopy class of the path defined by the 
corresponding 1-simplex of $B\Pi_1(T_n,\{p,q\})$.
This isomorphism restricts to give an isomorphism
\begin{equation}
\label{isom:pi1red2}
	\Pi_1(T_n,\{p\})
	\xto{\ \isom\ } 
	\Pi_1(B\Pi_1(T_n,\{p,q\}),\{p\}).
\end{equation}
Moreover, observe that the loops $l_1,\ldots,l_n$
in \eqref{pic:tndef} give a basis for the finite
free groupoid $\Pi_1(T_n,\{p\})$
and that $l_1,\ldots,l_n$ together with the arc $c$
give a basis for $\Pi_1(T_n,\{p,q\})$. Evaluation
against these basis elements gives isomorphisms
\begin{equation}
\label{isos:fun-prod}
	\fun(\Pi_1(T_n,\{p\}),G) 
	\xto{\ \isom\ }
	G^n
	\quad\text{and}\quad
	\fun(\Pi_1(T_n,\{p,q\}),G) 
	\xto{\ \isom\ }
	G^n\times G.
\end{equation}
Under these isomorphisms, the $\Hol(F_n)\times G^{\{p\}}$-action
on $\fun(\Pi_1(T_n,\{p\}),G)$ corresponds to the 
$\Hol(F_n)\times G^{\{p\}}$-action on $G^n$ given by 
\begin{equation}
\label{eq:holgp-action}
	((w,\theta),g_p) \cdot (g_i)_{1\leq i \leq n}
	= 
	(g_p \theta^{-1}(x_i)|_{g_1,\ldots,g_n} g_p^{-1})_{1\leq i \leq n}
\end{equation}
and the  $\Hol(F_n)\times G^{\{p,q\}}$-action
on the space $\fun(\Pi_1(T_n,\{p,q\}),G)$ corresponds to the 
 $\Hol(F_n)\times G^{\{p,q\}}$-action on $G^n\times G$ given by 
\begin{multline}
\label{eq:holgpq-action}
 	((w,\theta),g_p,g_q) \cdot ((g_i)_{1\leq i \leq n},g_c) 
	\\
	= 
	\big((g_p \theta^{-1}(x_i)|_{g_1,\ldots,g_n} g_p^{-1})_{1\leq i \leq n},
		 g_q g_c\theta^{-1}(w^{-1})|_{g_1,\ldots,g_n} g_p^{-1}\big).
\end{multline}
Here $(w,\theta)\in \Hol(F_n) = F_n\rtimes \Aut(F_n)$,
we use $x_i$ to denote the $i$-th basis element of $F_n$,
and $v|_{g_1,\ldots,g_n}$ for $v\in F_n$ denotes the 
result of substituting $g_i$ for $x_i$ in $v$
for every $i=1,\ldots, n$.

Making use of the isomorphisms
\eqref{isom:pi1red1}, \eqref{isom:pi1red2} and
\eqref{isos:fun-prod}
to simplify the spaces
\[
	\fun(\Pi_1(B\Pi_1(T_n,\{p,q\}),\{p,q\}),G)
	\quad\text{and}\quad
	\fun(\Pi_1(B\Pi_1(T_n,\{p,q\}),\{p\}),G),
\]
Proposition~\ref{prop:borel-opdesc}
implies that we may compute 
the operation $\Phi^G(U_n/B\Hol(F_n))$
by a push-pull construction in the diagram
\begin{equation}
\label{diag:push-pull}
\vcenter{\xymatrix@!0@C=4.5em@R=12ex{
	&
	G^n\sslash \Hol(F_n)\times G^{\{p\}}
	\ar[dl]_{!}
	&&&
	\ar[lll]_-\homot
	G^n\times G \sslash \Hol(F_n)\times G^{\{p,q\}}
	\ar[dr]
	\\
	\pt\sslash \Hol(F_n)\times G^{\{p\}}
	&&&&&
	\pt\sslash G^{\{q\}}
}}
\end{equation}
where the left-hand diagonal map is induced by the projection $G^n \to \pt$,
the horizontal map is induced by the projections $G^n \times G \to G^n$
and $\Hol(F_n)\times G^{\{p,q\}} \to \Hol(F_n)\times G^{\{p\}}$,
and the right-hand diagonal map is induced by the projection
$\Hol(F_n)\times G^{\{p,q\}} \to G^{\{q\}}$.
More precisely, we obtain the following result
\begin{lemma}
\label{lm:uopcomp}
The operation $\Phi^G(U_n/B\Hol(F_n))$
agrees with the composite
\newcommand{\entry}{H_\ast( B\Hol(F_n))\tensor H_\ast(BG)\;}
\[\vcenter{\xymatrix@!0@C=4em{
	*!R{\entry}
	\ar[r]^\times_\isom
	&
	*!L{\; H_\ast (B\Hol(F_n) \times BG) }
	\\
	*!R{\phantom{\entry}}
	\ar[r]^{(a)}_\isom
	&
	*!L{\;H_\ast (\pt \sslash \Hol(F_n) \times G^{\{p\}}) }
	\\
	*!R{\phantom{\entry}}
	\ar[r]^{!}
	&
	*!L{\;H_\ast ( G^n \sslash \Hol(F_n) \times G^{\{p\}})}
	\\
	*!R{\phantom{\entry}}
	\ar[r]^{(c)}_\isom
	&
	*!L{\;H_\ast ( G^n\times G \sslash \Hol(F_n) \times G^{\{p,q\}})}
	\\
	*!R{\phantom{\entry}}
	\ar[r]^{(d)}
	&
	*!L{\;H_\ast (\pt \sslash G^{\{q\}})}
	\\
	*!R{\phantom{\entry}}
	\ar[r]^{(e)}_\isom
	&
	*!L{\;H_\ast (BG)}
}}\]
where \textup{(a)} and \textup{(e)}
are induced by the evident homeomorphisms; 
where the map labeled by 
\textup{$!$}
is the umkehr map \cite[section~7.2]{HL}
associated to the map 
\textup{$!$} of \eqref{diag:push-pull}
considered as a map of fibrewise manifolds 
over $\pt \sslash \Hol(F_n) \times G^{\{p\}}$;
where the isomorphism 
\textup{(c)} is induced by the homotopy inverse of
the horizontal map in \eqref{diag:push-pull};
and where the map \textup{(d)} is induced by the right-hand
diagonal map in \eqref{diag:push-pull}. \qed
\end{lemma}

\subsection{Applications to holomorphs of free groups}

We will consider two families of non-trivial elements
in the homology of $\Hol(F_N)$. 
In this section, we will focus on the first one of 
these families, which consists of the non-trivial 
elements produced by the following immediate corollary of 
Theorem~\ref{thm:hol-app}.
\begin{corollary}
\label{cor:hol-app}
Let $a\in H_\ast(B\Sigma_{n_1,\ldots,n_r})$
be such that 
$\Phi^G(S_{n_1,\ldots,{n_r}}/B\Sigma_{n_1,\ldots ,n_r})^\sharp (a)$
is non-trivial for some compact Lie group $G$.
Then the element $\zeta_\ast(a)\in H_\ast B\Hol(F_N)$
is non-zero.\qed
\end{corollary}
\noindent
Our main results about these elements are
Theorem~\ref{thm:zetastab} below, which determines
their images in the stable range 
under iteration of the stabilization map
$H_\ast B\Hol(F_{n}) \to H_\ast B\Hol(F_{n+1})$,
and Corollary~\ref{cor:non-stab}, which 
strengthens the conlusion of
Corollary~\ref{cor:hol-app} to the statement
that the class
$\zeta_\ast(a)$ is not in the image of the stabilization map.
As argued in Remark~\ref{rk:stab} below, 
all examples in this first family of elements
that arise from the computations in this paper
are in fact stable in the sense that they survive
under arbitrary iterations of the stabilization map. 
The second family of elements of $H_\ast B\Hol(F_N)$,
constructed in Corollary~\ref{cor:hol-fam2}
in the next subsection, consists of unstable classes.

\begin{theorem}
\label{thm:zetastab}
The following diagram commutes when $N+r+L > 2k+1$.
\begin{equation}
\label{diag:zetastab}
\vcenter{\xymatrix{
	H_k(B\Sigma_{n_1,\ldots,n_r};\,A)
	\ar[r]^-{\mu_\ast}
	\ar[dd]_{\zeta_\ast}
	&
	H_k (B\Sigma_{N+r};\,A)
	\ar[r]
	&
	H_k (B\Sigma_{N+r+L};\,A)
	\ar[d]^\isom
	\\
	&&
	H_k (B\Aut(F_{N+r+L});\,A)
	\ar[d]^\isom
	\\
	H_k (B\Hol(F_N);\,A)
	\ar[rr]
	&&
	H_k (B\Hol(F_{N+r+L});\,A)
}}
\end{equation}
Here $A$ is an arbitrary abelian group and
all maps except $\zeta_\ast$
are induced by the standard inclusions of groups.
In particular, $\mu$ denotes the inclusion
\[
	\mu \colon \Sigma_{n_1,\ldots,n_r}
	\longincl 
	\Sigma_{n_1+\cdots+n_r} = \Sigma_{N+r}.
\]
\end{theorem}

\begin{remark}
Arguably, the most obvious map connecting $\Sigma_{n_1,\ldots,n_r}$
to a holomorph of a free group is the composite
\[
	\Sigma_{n_1,\ldots,n_r} 
	\xto{\ \mu\ }
	\Sigma_{N+r}
	\longto
	\Aut(F_{N+r})
	\longto 
	\Hol(F_{N+r})
\] 
of inclusions homomorphisms. This map is certainly
different from our homomorphism 
$\zeta\colon \Sigma_{n_1,\ldots,n_r} \to \Hol(F_N)$;
indeed, 
the two maps do not even have the same target.
Theorem~\ref{thm:zetastab} implies that the  
two maps nevertheless agree on homology after sufficient stabilization.
\end{remark}

\begin{proof}[Proof of Theorem~\ref{thm:zetastab}]
For the duration of the proof, we will work with homology
with coefficients in $A$, and will omit the coefficients
from the notation. After the proof is complete,
we will resume our convention of defaulting to $\F_2$-coefficients.

Diagram~\eqref{diag:zetastab} embeds as the top left-hand rectangle
into the diagram
\begin{equation}
\label{diag:enclosing}
\vcenter{\xymatrix@C-1.08em@R-0.95ex{
	H_k(B\Sigma_{n_1,\ldots,n_r})
	\ar[r]^-{\mu_\ast}
	\ar[dd]_{\zeta_\ast}
	&
	H_k (B\Sigma_{N+r})
	\ar[r]
	&
	H_k (B\Sigma_{N+r+L})
	\ar[d]^\isom
	\ar[r]^\isom
	&
	H_k(B\Sigma_{N+2r+L})
	\ar[ddd]^\isom
	\\
	&&
	H_k B\Aut(F_{N+r+L})
	\ar[d]^\isom
	\\
	H_k B\Hol(F_N)
	\ar[rr]
	\ar[d]
	&&
	H_k B\Hol(F_{N+r+L})
	\ar[d]^\isom
	\\
	H_k B\Aut(F_N)
	\ar[rr]
	&&
	H_kB\Aut(F_{N+r+L})
	\ar[r]^-\isom
	&
	H_k B\Aut(F_{N+2r+L})	
}}
\end{equation}
Here the vertical maps in the bottom left-hand square
are induced by the quotient maps from the holomorphs 
to the automorphism groups, and all remaining maps 
except for $\zeta_\ast$ are induced by inclusions.
That the maps so indicated are isomorphisms
follows from Theorems~\ref{thm:stability} and \ref{thm:galatius}.
The bottom left-hand square and the right-hand 
rectangle commute, so to prove the claim, it is enough
to show that the outer rectangle commutes.

Let us denote $R_n = \bigvee^n S^1 \vee \bigvee^r I$,
where the basepoint of $I$ is one of the endpoints.
We label the basepoint of $R_n$ with $p_0$ and
the free endpoints of the $r$ intervals with $p_1,\ldots,p_r$.
We also use $p_0,\ldots,p_r$ to label the 0-cells of 
$\hat S_{n_1,\ldots,n_r}$, with $p_0$ corresponding to the incoming point
and $p_r$ to the outgoing point,
and label the point $p$ of $T_N$ in \eqref{pic:tndef}
by $p_0$ and the point $q$ with $p_r$.
Fix a homotopy equivalence
$h \colon \hat S_{n_1,\ldots,n_r} \to R_N$ under $\{p_0,\ldots,p_r\}$,
and let $h' \colon \hat S_{n_1,\ldots,n_r} \to T_N$
be the composite of $h$ with the homotopy equivalence
that collapses the edges of $R_N$ corresponding to $p_1,\ldots,p_{r-1}$.
We then have the following diagram.
\[\xymatrix{
	\Sigma_{n_1,\ldots,n_r}
	\ar[r]^\mu
	\ar[d]_{\textrm{act}}
	&
	\Sigma_{N+r}
	\ar[r]^(0.46){\textrm{incl}}
	\ar[d]_{\textrm{perm}}
	&
	\Sigma_{N+2r+L}
	\ar[d]^{\textrm{perm}}
	\\
	A(\hat S_{n_1,\ldots,n_r},\{p_0,\ldots,p_r\})
	\ar[r]^-{\textrm{(C)}}
	\ar[d]_{\textrm{(A)}}
	\ar[dr]_\isom^{h_\ast}
	&
	A(\bigvee^{N+r} S^1,\pt)
	\ar[r]^(0.46){\textrm{stab}}
	&
	A(\bigvee^{N+2r+L} S^1,\pt)
	\\
	A(\hat S_{n_1,\ldots,n_r},\{p_0,p_r\})
	\ar[dr]_\isom^{h'_\ast}
	&
	A(R_N,\{p_0,\ldots,p_r\})
	\ar[r]^(0.46){\textrm{stab}}
	\ar[u]^{\textrm{(C)}}
	\ar[d]_{\textrm{(A')}}
	&
	A(R_{N+r+L},\{p_0,\ldots,p_r\})
	\ar[u]_{\textrm{(C)}}
	\ar[dd]^{\textrm{(A')}}
	\\
	&
	A(T_N,\{p_0,p_r\})
	\ar[d]_{\textrm{(A')}}
	\\
	&
	A(\bigvee^N S^1,\{p_0\})
	\ar[r]^(0.46){\textrm{stab}}
	&
	A(\bigvee^{N+r+L} S^1,\{p_0\})
}\]
Here the map `$\mathrm{incl}$' is the inclusion;
the map `$\mathrm{act}$' is given by the  
$\Sigma_{n_1,\ldots,n_r}$-action on 
$\hat S_{n_1,\ldots,n_r}$;
the maps labeled `$\mathrm{perm}$'
are inclusions of the symmetric groups
as permutations of wedge summands;
the maps labeled `$\mathrm{stab}$'
are given by attaching copies of $S^1$;
$h_\ast$ and $h'_\ast$ are
the maps induced by $h$ and $h'$, respectively;
and the maps labeled by 
(A),
(A')
and
(C)
are iterates of maps of the respective types
defined on page~\pageref{list:map-types}.

All squares and the triangle in the above diagram 
commute up to conjugacy, and hence they yield a strictly commutative
diagram upon application of $H_k$. By assumption on $N+r+L$,
the right-hand vertical map labeled by (A')
induces an isomorphism on $H_k$. The inverse to this isomorphism
is induced by an iterate of maps of 
type (B), and the composite of this inverse with the 
right-hand vertical map labeled (C) amounts to an $r$-fold
iteration of the map that attaches a new copy of $S^1$.
Now, on $H_k$,
starting at $\Sigma_{n_1,\ldots,n_r}$ and proceeding counterclockwise
along the outer edge of the diagram, 
the composite of `$\mathrm{act}$', (A), and
$h'_\ast$ yields the map $\zeta_\ast$; the map (A')
gives the lower left-hand vertical map in \eqref{diag:enclosing};
and the composite of `$\mathrm{stab}$', inverse of (A'), and (C)
gives the bottom row in \eqref{diag:enclosing}.
On the other hand, the composite along the top row 
yields the composite along the top row in \eqref{diag:enclosing},
and the right-hand map labeled `$\textrm{perm}$'
induces the right-hand vertical map in \eqref{diag:enclosing}.
Thus the outer rectangle in diagram \eqref{diag:enclosing} commutes, 
as desired.
\end{proof}

\begin{remark}
\label{rk:stab}
If $A$ is a ring, the map $\mu_\ast$ in 
\eqref{diag:zetastab} amounts to iterated multiplication
in the ring $H_\ast(\bigsqcup_{n \geq 0} B\Sigma_n;\,A)$.
By \cite[Theorem 5.8]{NakaokaDecomposition}),
the top map on the left in \eqref{diag:zetastab}
is an injection, so Theorem~\ref{thm:zetastab}
implies that $\zeta_\ast (a)\in H_\ast (B\Hol(F_N);\,A)$ 
is a stable class whenever
$a \in H_\ast(B\Sigma_{n_1,\ldots,n_k};\,A)$
multiplies to a non-trivial element in the 
ring $H_\ast(\bigsqcup_{n \geq 0} B\Sigma_n;\,A)$.
In the case $A=\F_2$, it turns out that 
every positive-dimensional
class $a$ for which our computations 
show the non-triviality of 
$\Phi^G(S_{n_1,\ldots,{n_r}}/B\Sigma_{n_1,\ldots ,n_r})^\sharp (a)$
has this property.
To see this, one makes use of 
Theorems~\ref{thm:turnerdesc} and \ref{thm:vanishing}
and the observation that 
in every case we have computed,
$\Phi^G(S_n/B\Sigma_n)^\sharp(b)$ 
and 
$\Phi^G(S_m/B\Sigma_m)^\sharp(c)$
commute
for all $m$, $n$, $b$ and $c$.
\end{remark}

\begin{remark}
Similarly, using diagram \eqref{diag:enclosing}
in place of \eqref{diag:zetastab},
we obtain a non-trivial stable element in 
$H_\ast (\Aut(F_N);\,A)$
for every 
$a \in H_\ast(B\Sigma_{n_1,\ldots,n_k};\,A)$
multiplying to a non-trivial element in the 
ring $H_\ast(\bigsqcup_{n \geq 0} B\Sigma_n;\,A)$.
\end{remark}

\begin{theorem}
\label{thm:non-stab}
Suppose $b \in H_\ast B\Hol(F_N)$ is 
a positive-dimensional class which is in the image of the 
stabilization map $H_\ast B\Hol(F_{N-1}) \to H_\ast B\Hol(F_N)$. 
Then the operation $\Phi^G(U_N/B\Hol(F_N))^\sharp(b)$
vanishes for every compact Lie group $G$.
\end{theorem}
\begin{proof}
If $G$ is a finite group of odd order, then 
$\Phi^G(U_N/B\Hol(F_N))^\sharp(b) = 0$
since $b$ is positive-dimensional and 
$H_\ast BG$ is concentrated in degree zero.
For the remainder of the proof, let us assume that $G$
is a positive-dimensional compact Lie group or a 
finite group of even order.

Let $i\colon \Hol(F_{N-1}) \incl \Hol(F_N)$ be 
the inclusion, and let $i^\ast U_N/B\Hol(F_{N-1})$
denote the pullback of $U_N/B\Hol(F_N)$ along the map 
$B\Hol(F_{N-1}) \to B\Hol(F_N)$
induced by~$i$. To prove the claim, it is enough to show that
the operation $\Phi^G(i^\ast U_N/B\Hol(F_{N-1}))$ is zero.
We have
\[
	i^\ast U_N 
	= 
	(E\Hol(F_{N-1}) \times_{\Hol(F_{N-1})} B\Pi_1(T_N,\{p,q\}))'
\]
where $(-)'$ denotes the mapping cylinder construction of 
Remark~\ref{rk:families-from-borel-construction}.
Here the group $\Hol(F_{N-1})$ acts on 
$B\Pi_1(T_N,\{p,q\})$ via the composite
\[\xymatrix{
	\Hol(F_{N-1}) 
	\ar[r]^-i
	&
	\Hol(F_N)
	\ar[r]^-\psi_-\isom
	&
	\pi_0 \hAut(T_N)
}\]
where $\psi$ is the isomorphism \eqref{iso:psi}.
Notice that the above composite is 
the same as the composite
\[\xymatrix{
	\Hol(F_{N-1}) 
	\ar[r]^-\psi_-\isom
	&
	\pi_0 \hAut(T_{N-1})
	\ar[r]
	&
	\pi_0 \hAut(T_N)
}\]
where the latter map is given by attaching a copy of $S^1$
and extending by identity.
Our strategy is to use 
this observation to obtain a decomposition of
the family $i^\ast U_N$ which features the h-graph 
cobordism $\mu$ of \eqref{eq:basic-cobs}, and 
then deduce the claim from Proposition~\ref{prop:vanishing}.

Let $T'_{N-1}$ and $C$ be the h-graph cobordisms pictured below.
\[
	T'_{N-1} = 
	\left(
	\begin{tikzpicture}[
			baseline=15,
			scale=0.03,
			->-/.style={	decoration={	markings,
			  	mark=at position 0.5 with {\arrow{>}}
				},postaction={decorate}},
			->--/.style={	decoration={	markings,
			  	mark=at position 0.46 with {\arrow{>}}
				},postaction={decorate}},
			-->-/.style={	decoration={	markings,
			  	mark=at position 0.58 with {\arrow{>}}
				},postaction={decorate}}
		]
		\clip (-47, -20) rectangle (130, 70);
		\path[ARC,->-] (0,0) -- (120,0) node [midway, below] {$c$};
		\path[ARC, fill=black] (0,0) circle (3);
		\node [below] at (0,0) {$p$};
		\path[ARC, fill=black] (120,0) circle (3);
		\node [below] at (120,0) {$q_0$}; 
		\path[ARC] (0,0) .. controls (30,5) and (90,25) .. (120,25);
		\path[ARC, fill=black] (120,25) circle (3);
		\node [below] at (120,25) {$q_1$}; 
		\path[ARC] (0,0) .. controls (30,10) and (90,50) .. (120,50);
		\path[ARC, fill=black] (120,50) circle (3);
		\node [below] at (120,50) {$q_2$}; 
		\path[ARC,->--] (0,0) .. controls (-20,80) and (-80,45) 
							.. (0,0) node [midway, above] {$l_1$};
		\node at (1.5,35) {$\cdots$};
		\path[ARC,-->-] (0,0) .. controls (80,45) and (20,80)
							.. (0,0)  node [midway, above] {$l_{N-1}$};
	\end{tikzpicture}
	\right)
	\colon \pt \longhto \textstyle\bigsqcup^3\pt
\]
\[
	C = 
	\left(
	\begin{tikzpicture}[
			baseline=15,
			scale=0.03,
			->-/.style={	decoration={	markings,
			  	mark=at position 0.5 with {\arrow{>}}
				},postaction={decorate}},
			->--/.style={	decoration={	markings,
			  	mark=at position 0.46 with {\arrow{>}}
				},postaction={decorate}},
			-->-/.style={	decoration={	markings,
			  	mark=at position 0.58 with {\arrow{>}}
				},postaction={decorate}}
		]
		\path[ARC, fill=black] (0,0) circle (3);
		\node [below] at (0,0) {$q_0$};
		\path[ARC, fill=black] (0,25) circle (3);
		\node [below] at (0,25) {$q_1$}; 
		\path[ARC, fill=black] (0,50) circle (3);
		\node [below] at (0,50) {$q_2$}; 
		\path[ARC, fill=black] (55,0) circle (3);
		\node [below] at (55,0) {$q$}; 
		\path[ARC] (0,0) -- (55,0);
		\path[ARC] (0,25) .. controls (30,25) and (30,50) .. (0,50);
	\end{tikzpicture}
	\right)
	\colon \textstyle\bigsqcup^3\pt \longhto \pt
\]
We have an evident homeomorphism $T_N \isom C\circ T'_{N-1}$,
and this homeomorphism induces a homotopy equivalence
\begin{equation}
\label{he:tndecom}
	B\Pi_1(T_N,\{p,q\}) 
	\xto{\ \homot\ }
	B\Pi_1(C\circ T'_{N-1},\{p,q_0,q_1,q_2,q\}).
\end{equation}
under $\{p,q\}$.
From $T'_{N-1}$ and $C$ we obtain h-graph cobordisms
\begin{align}
	\label{eq:hcob1}
	B\Pi_1(T'_{N-1},\{p,q_0,q_1,q_2\}) 
	&\colon 
	\pt \longhto {\textstyle\bigsqcup^3 \pt}
	\qquad\text{and}
	\\
	\label{eq:hcob2}
	B\Pi_1(C,\{q_0,q_1,q_2,q\}) 
	&\colon 
	{\textstyle\bigsqcup^3 \pt}  \longhto \pt,
\end{align}
and we would like to compare the composite of these
h-graph cobordisms to the target in \eqref{he:tndecom}.

By \cite[Proposition 7.26]{HL},
the underlying pushout square of spaces of the 
following square of h-graphs with basepoints
\begin{equation}
\label{sq:comp1}
\vcenter{\xymatrix{
	(\{q_0,q_1,q_2\},\{q_0,q_1,q_2\})
	\ar[r]
	\ar[d]
	&
	(C,\{q_0,q_1,q_2,q\})
	\ar[d]
	\\
	(T'_{N-1},\{p,q_0,q_1,q_2\})
	\ar[r]
	&
	(C\circ T'_{N-1},\{p,q_0,q_1,q_2,q\})
}}
\end{equation}
and the square
\begin{equation}
\label{sq:comp2}
\vcenter{\xymatrix{
	B\Pi_1(\{q_0,q_1,q_2\},\{q_0,q_1,q_2\})
	\ar[r]
	\ar[d]
	&
	B\Pi_1(C,\{q_0,q_1,q_2,q\})
	\ar[d]
	\\
	B\Pi_1(T'_{N-1},\{p,q_0,q_1,q_2\})
	\ar[r]
	&
	B\Pi_1(C\circ T'_{N-1},\{p,q_0,q_1,q_2,q\})
}}
\end{equation}
obtained from \eqref{sq:comp1} by applying $B\Pi_1$
are connected by a zigzag of natural homotopy equivalences.
Since the underlying square of spaces of the diagram \eqref{sq:comp1} 
is a homotopy cofibre square,
it follows that square \eqref{sq:comp2} also is.
Observing that the space in the top left corner of \eqref{sq:comp2}
is just the three-point space $\{q_0,q_1,q_2\}$, we see
that the pushout of the top and the left arrows in \eqref{sq:comp2}
is the composite 
of \eqref{eq:hcob1} and \eqref{eq:hcob2}.
Since the maps out of the top left-hand corner 
in \eqref{sq:comp2} are cofibrations, it follows that the map
\begin{multline}
\label{he:decomp}
	\quad
	B\Pi_1(C,\{q_0,q_1,q_2,q\}) \circ B\Pi_1(T'_{N-1},\{p,q_0,q_1,q_2\}) 
	\\
	\xto{\quad \homot\quad}
	B\Pi_1(C\circ T'_{N-1},\{p,q_0,q_1,q_2,q\}).
	\quad
\end{multline}
induced by the square \eqref{sq:comp2} is a homotopy equivalence.

We let $\Hol(F_{N-1})$ act on the sources and targets of 
\eqref{he:tndecom} and \eqref{he:decomp}
as follows. On $B\Pi_1(T_N,\{p,q\})$ the action has already been 
specified. Notice that $T'_{N-1}$ is obtained from $T_{N-1}$
by attaching the two edges joining $p$ to $q_1$ and $q_2$,
and write $j$ for the resulting map
\[
	j\colon \pi_0\hAut(T_{N-1}) \longto \pi_0\hAut(T'_{N-1})
\]
given by extension by identity along the new edges.
The action on the space 
$B\Pi_1(C\circ T'_{N-1},\{p,q_0,q_1,q_2,q\})$ is then given by the 
composite
\[\xymatrix{
	\Hol(F_{N-1}) 
	\ar[r]^-\psi_-\isom
	&
	\pi_0\hAut(T_{N-1})
	\ar[r]^-j
	&
	\pi_0\hAut(T'_{N-1})
	\ar[r]
	&
	\pi_0\hAut(C\circ T'_{N-1}),
}\]
where the last map is given by attaching $C$ and extending by 
identity. Finally, the action on 
$B\Pi_1(C,\{q_0,q_1,q_2,q\}) \circ B\Pi_1(T'_{N-1},\{p,q_0,q_1,q_2\})$
is induced
by the trivial action on the $B\Pi_1(C,\{q_0,q_1,q_2,q\})$-factor
and the action given by the composite
\[\xymatrix{
	\Hol(F_{N-1}) 
	\ar[r]^-\psi
	&
	\pi_0\hAut(T_{N-1})
	\ar[r]^-j
	&
	\pi_0\hAut(T'_{N-1})
}\]
on the 
$B\Pi_1(T'_{N-1},\{p,q_0,q_1,q_2\})$-factor.

The maps \eqref{he:tndecom} and \eqref{he:decomp}
are $\Hol(F_{N-1})$-equivariant
with respect to the aforementioned actions,
and induce a zigzag of 2-cells,
all given by homeomorphisms on base spaces,
connecting $i^\ast U_N/B\Hol(F_{N-1})$
to the composite of the families
\[
	(E\Hol(F_{N-1}) 
	\times_{\Hol(F_{N-1})} 
	B\Pi_1(T'_{N-1},\{p,q_0,q_1,q_2\}))'/B\Hol(F_{N-1})
	\colon
	\pt \longhto {\textstyle\bigsqcup^3 \pt}
\]
and 
\[
	B\Pi_1(C,\{q_0,q_1,q_2,q\})/\pt
	\colon
	{\textstyle\bigsqcup^3 \pt} \longhto \pt.
\]
Here again $(-)'$ denotes the mapping cylinder construction of 
Remark~\ref{rk:families-from-borel-construction}.
The h-graph cobordism $B\Pi_1(C,\{q_0,q_1,q_2,q\})$ 
is homotopy equivalent relative to the incoming and 
outgoing points to the h-graph cobordism $C$,
which, using the notation of \eqref{eq:basic-cobs},
decomposes as 
\[	
	C \homot (\varepsilon \circ \mu) \sqcup I.
\]
Thus the claim follows from Proposition~\ref{prop:vanishing}.
\end{proof}

Theorems~\ref{thm:hol-app} and \ref{thm:non-stab}
have the following corollary.
\begin{corollary}
\label{cor:non-stab}
Suppose $a\in H_\ast(B\Sigma_{n_1,\ldots,n_r})$
is a positive-dimensional class such that 
$\Phi^G(S_{n_1,\ldots,{n_r}}/B\Sigma_{n_1,\ldots ,n_r})^\sharp (a)$
is non-trivial for some compact Lie group $G$.
Then the element $\zeta_\ast(a)\in H_\ast B\Hol(F_N)$
is not in the image of the stabilization map
$H_\ast B\Hol(F_{N-1}) \to H_\ast B\Hol(F_N)$. \qed
\end{corollary}

\subsection{Applications to automorphism groups of free groups}

We now proceed to 
discuss the promised applications to the twisted homology
$H_\ast(B\Aut(F_N);\,\tilde \F_2^N)$.
Here $\tilde \F_2^N$ denotes the vector space $\F_2^N$
equipped with the $\Aut(F_N)$-action given by the
composite homomorphism $\Aut(F_N) \to GL_N(\Z) \to GL_N(\F_2)$
where the first map is induced by abelianization of $F_N$
and the second map is given by reduction mod 2.
Considering the Hochschild--Serre spectral
sequence associated to the 
split short exact sequence
\[\xymatrix{
	1
	\ar[r]
	&
	F_N
	\ar[r]
	&
	\Hol(F_N)
	\ar[r]^\pi
	&
	\Aut(F_N)
	\ar[r]
	&
	1,
}\]
we see that there is a short exact sequence
\begin{equation}
\label{ses:holfn}
\xymatrix@C-0.5em{
	0
	\ar[r]
	&
	H_{\ast-1} (B\Aut(F_N);\,\tilde \F_2^N) 
	\ar[r]^-j
	&
	H_\ast B\Hol(F_N)
	\ar[r]^-{\pi_\ast}
	&
	H_\ast B\Aut(F_N)
	\ar[r]
	&
	0.
}
\end{equation}
Using the inclusion $\Aut(F_N) \to \Hol(F_N)$ to
split this short exact sequence, we obtain 
a decomposition of the
homology of $\Hol(F_N)$  
as a direct sum
\begin{equation}
\label{eq:holfn-dirsum}
	H_\ast B\Hol(F_N) 
	\isom 
	H_{\ast-1} (B\Aut(F_N);\,\tilde \F_2^N) 
	\oplus
	H_{\ast} B\Aut(F_N). 
\end{equation}
Let 
$\rho \colon H_\ast B\Hol(F_N) \to H_{\ast-1} (B\Aut(F_N);\,\tilde \F_2^N)$
be the projection onto the first factor in this 
decomposition. Our aim is to prove the following theorem.
\begin{theorem}
\label{thm:aut-app}
Let $G$ be a positive-dimensional compact Lie group 
or a finite group of even order. Then the map 
\[
	\Phi^G(U_N/B\Hol(F_N))^\sharp
	\colon 
	H_{\ast-\mathrm{shift}}(B\Hol(F_N)) \longto
	\Hom_\ast(H_\ast(BG),H_\ast(BG))
\]
factors through the map 
\[
	\rho 
	\colon 
	H_{\ast-\mathrm{shift}} ( B\Hol(F_N) )
	\longto 
	H_{\ast-\mathrm{shift}-1} (B\Aut(F_N);\,\tilde \F_2^N).
\]
Here `shift' equals $\dim(G)(N-1)$.
\end{theorem}
Theorems~\ref{thm:hol-app} and \ref{thm:aut-app}
have the following corollary.
\begin{corollary}
\label{cor:aut-app}
If $a\in H_\ast(B\Sigma_{n_1,\ldots,n_r})$
is such that 
$\Phi^G(S_{n_1,\ldots,{n_r}}/B\Sigma_{n_1,\ldots ,n_r})^\sharp (a)$
is non-trivial for some compact Lie group $G$ which is 
either positive-dimensional or finite of even order,
then the element $\rho\zeta_\ast(a)\in H_{\ast-1}(B\Aut(F_N);\,\tilde \F_2^N)$
is non-zero. \qed
\end{corollary}

The inclusions $\Aut(F_n) \incl \Aut(F_{n+1})$ and 
$\F_2^n \incl \F_2^{n+1}$
induce a stabilization map 
\[
	H_{\ast} (B\Aut(F_n);\,\tilde \F_2^n)
	\longto
	H_{\ast} (B\Aut(F_{n+1});\,\tilde \F_2^{n+1}).
\]
It follows from Theorem~\ref{thm:stability}
that the map $\pi_\ast$ in \eqref{ses:holfn}
is an isomorphism on $H_k$ for $N>2k+1$,
so we see that $H_k(B\Aut(F_n);\, \tilde \F_2^n) = 0$
for $n> 2k+3$. Thus all elements in 
$H_*(B\Aut(F_N);\, \tilde \F_2^N)$
are unstable in the sense that they 
are annihilated by an iteration of the 
stabilization map.
As the decomposition
\eqref{eq:holfn-dirsum} is compatible with 
the stabilization maps on both sides,
it follows that all elements in the image of 
the map $j$ of \eqref{ses:holfn}
are also unstable. Thus Theorems~\ref{thm:hol-app},
\ref{thm:aut-app} and \ref{thm:non-stab}
imply the following corollary,
where we have written 
$i$ for the inclusion
$i\colon \Aut(F_N) \incl \Hol(F_N)$.
\begin{corollary}
\label{cor:hol-fam2}
Let $a\in H_\ast(B\Sigma_{n_1,\ldots,n_r})$
be such that 
$\Phi^G(S_{n_1,\ldots,{n_r}}/B\Sigma_{n_1,\ldots ,n_r})^\sharp (a)$
is non-trivial for some compact Lie group $G$ which is 
either positive-dimensional or finite of even order.
Then 
\[
	j\rho\zeta_\ast(a) 
	= 
	\zeta_\ast(a) + i_\ast \pi_\ast \zeta_\ast (a) 
	\in 
	H_\ast B\Hol(F_N)
\]
is an unstable element which is not in the image of the stabilization 
map $H_\ast B\Hol(F_{N-1}) \to H_\ast B\Hol(F_N)$ \qed
\end{corollary}

Before proving Theorem~\ref{thm:aut-app}, we need 
to establish an auxiliary result.

\begin{lemma}
\label{lm:umk-ind}
Suppose $\pi\colon M\to B$ is a fibrewise closed manifold 
in the sense of \cite[Definition~7.1]{HL} and suppose
the fibres of $\pi$ are either positive-dimensional
or finite of even cardinality. Then the composite
\[\xymatrix{
	H_\ast (B)
	\ar[r]^-{\pi^!}
	&
	H_{\ast+d} (M)
	\ar[r]^-{\pi_\ast}
	&
	H_{\ast+d} (B)
}\]
is zero. Here $\pi^!$ denotes the umkehr map
of \cite[section~7.2]{HL} associated to 
the map $\pi$
(considered as a map of fibrewise manifolds over $B$)
and $d$ denotes the fibre dimension of $M$.
\end{lemma}
\begin{proof}
If the fibres of $\pi$ are finite, the umkehr map
$\pi^!$ is just the transfer map \cite[Lemma~8.6]{HL}, and 
the composite $\pi_\ast \pi^!$ amounts to multiplication
by the cardinality of the fibre. By assumption this
is an even number, so $\pi_\ast \pi^!$ vanishes
since we are working with $\F_2$ coefficients. 
Let us assume that the fibres of $\pi$ are 
positive-dimensional. Dualizing, it suffices 
to show that the composite
\[\xymatrix{
	H^\ast (B)
	\ar[r]^-{\pi^\ast}
	&
	H^{\ast} (M)
	\ar[r]^-{\pi_!}
	&
	H^{\ast-d} (B)
}\]
is zero, where $\pi_!$ again denotes the linear dual of $\pi^!$.
Regarding (as before) $H^{\ast} (M)$ as an 
$H^\ast(B)$-module via the map $\pi^\ast$,
both maps in the above composite are
$H^\ast(B)$-linear: the linearity of $\pi^\ast$ is immediate,
and we have already observed the linearity of $\pi_!$ on 
page~\pageref{p:modstr}. The claim now follows for degree reasons.
\end{proof}

\begin{proof}[Proof of Theorem~\ref{thm:aut-app}]
It is enough to show that the map 
$\Phi^G(U_N/B\Hol(F_N))^\sharp$ vanishes
on classes which are in the image of the map
\[
	i_\ast \colon H_\ast B\Aut(F_N) \longto H_\ast B\Hol(F_N)
\]
induced by the inclusion $i\colon \Aut(F_N) \incl \Hol(F_N)$.
The diagram \eqref{diag:push-pull} 
for computing $\Phi^G(U_N/B\Hol(F_N))$
fits into a commutative diagram
\begin{equation}
\label{diag:push-pull2}
\vcenter{\xymatrix@!0@C=6.4em@R=5ex{
	&
	G^N\sslash \Hol(F_N)\times G^{\{p\}}
	\ar[ddl]_{!}
	&&
	\ar[ll]_-\homot
	G^N\times G \sslash \Hol(F_N)\times G^{\{p,q\}}
	\ar[ddr]
	\\
	\\
	\pt\sslash \Hol(F_N)\times G^{\{p\}}
	&&&&
	\pt\sslash G^{\{q\}}
	\\
	&
	G^N\sslash \Aut(F_N)\times G^{\{p\}}
	\ar[ddl]_{!}^{\pi_1}
	\ar[uuu]
	&&
	\ar[ll]_-\homot^-\eta
	G^N\times G \sslash \Aut(F_N)\times G^{\{p,q\}}
	\ar[ddr]_{\pi_2}
	\ar[uuu]
	\\
	\\
	\pt\sslash \Aut(F_N)\times G^{\{p\}}
	\ar[uuu]
	&&&&
	\pt\sslash G^{\{q\}}	
	\ar[uuu]_\id
}}
\end{equation}
where the lower part is obtained from the upper part
by restriction along $i$ and the unlabeled vertical
maps are induced by $i$. Explicitly, the 
$\Aut(F_N) \times G^{\{p\}}$-action on $G^N$
is given by 
\begin{equation}
\label{eq:autgp-action}
	(\theta,g_p) \cdot (g_i)_{1\leq i \leq N}
	= 
	(g_p \theta^{-1}(x_i)|_{g_1,\ldots,g_N} g_p^{-1})_{1\leq i \leq N},
\end{equation}
the  $\Aut(F_N)\times G^{\{p,q\}}$-action
on $G^N\times G$ is given by 
\begin{equation}
\label{eq:autgpq-action}
 	(\theta,g_p,g_q) \cdot ((g_i)_{1\leq i \leq N},g_c) 
	= 
	\big((g_p \theta^{-1}(x_i)|_{g_1,\ldots,g_N} g_p^{-1})_{1\leq i \leq N},
		 g_q g_c g_p^{-1}\big),
\end{equation}
and the maps $\pi_1$, $\pi_2$ and $\eta$ are induced 
by the evident projection maps of groups and spaces.
Here $x_i$ again denotes the $i$-th basis element of $F_N$,
and $v|_{g_1,\ldots,g_N}$ for $v\in F_N$ denotes the 
result of substituting $g_i$ for $x_i$ in $v$
for every $i=1,\ldots, N$.

Observe that the left-hand parallelogram in \eqref{diag:push-pull2}
is a pullback square,
and gives rise to a commutative square on homology after taking
umkehr maps of the maps labeled by $!$ and 
induced maps of the vertical maps.
To obtain the umkehr map $\pi_1^!$, we should interpret $\pi_1$
as a map of fibrewise manifolds
over $\pt\sslash \Aut(F_N)\times G^{\{p\}}$.
Using Lemma~\ref{lm:uopcomp}, we deduce that the composite 
\newcommand{\entry}{H_{\ast-\mathrm{shift}} (B\Aut(F_N)) \tensor H_\ast(BG)\;}
\[\xymatrix@!0@C=6.7em{
	*!R{\entry}
	\ar[r]^-{i_\ast \tensor 1}
	&
	*!L{\;H_{\ast-\mathrm{shift}} (B\Hol(F_N)) \tensor H_\ast(BG)}
	\\
	*!R{\phantom{\entry}}
	\ar[r]^-{\Phi^G(U_N/B\Hol(F_N))}
	&
	*!L{\;H_\ast(BG)}
}\]
factors through the composite
$(\pi_2)_\ast \circ \eta_\ast^{-1} \circ \pi_1^!$.
The map 
\[
	\eta' 
	\colon 
	G^N\sslash \Aut(F_N)\times G^{\{p\}}
	\longto
	G^N\times G \sslash \Aut(F_N)\times G^{\{p,q\}}
\]
induced by the maps
\[
	G^N\longto G^N\times G,
	\quad 
	(g_i)_{1\leq i\leq N} \longmapsto ((g_i)_{1\leq i\leq N},e)
\]
and 
\[
	\Aut(F_N)\times G^{\{p\}} \longto \Aut(F_N)\times G^{\{p,q\}},
	\quad
	(\theta,g_p) \longmapsto (\theta,g_p,g_p)	
\]
is a right inverse to $\eta$. It follows that 
$\eta_\ast^{-1} = \eta'_\ast$.
But the composite $\pi_2 \circ \eta'$ equals 
$\pi_3 \circ \pi_1$, where 
$\pi_3 \colon \pt \sslash \Aut(F_N) \times G^{\{p\}} \to \pt\sslash G^{\{q\}}$
is induced by the projection away from $\Aut(F_N)$.
Thus 
$(\pi_2)_\ast \circ \eta_\ast^{-1} \circ \pi_1^! 
= 
(\pi_3)_\ast \circ (\pi_1)_\ast \circ \pi_1^!$,
and the claim follows from 
Lemma~\ref{lm:umk-ind}.
\end{proof}

\subsection{Applications to affine groups}

We now turn to the applications to the homology of the affine groups
$\Aff_N(\Z)$ and $\Aff_N(\F_2)$. For any $n$, 
the abelianization $F_n \to \Z^n$
induces a homomorphism 
\[
	\alpha_\Z\colon \Hol(F_n) \longto \Aff_n(\Z),
\]
which we may compose with the homomorphism $\Aff_n(\Z) \to \Aff_n(\F_2)$
given by reduction mod 2 to obtain a homomorphism 
$\alpha_{\F_2} \colon \Hol(F_n) \to \Aff_n(\F_2)$. Our aim is to prove
the following result.
\begin{theorem}
\label{thm:aff-app}
Let $G$ be an abelian compact Lie group.
Then the map 
\[
	\Phi^G(U_N/B\Hol(F_N))^\sharp
	\colon 
	H_{\ast-\mathrm{shift}}(B\Hol(F_N))
	\longto
	\Hom_\ast(H_\ast(BG),H_\ast(BG))
\]
factors through the map 
\[
	(\alpha_\Z)_\ast
	\colon 
	H_{\ast-\mathrm{shift}} ( B\Hol(F_N) )
	\longto 
	H_{\ast-\mathrm{shift}} (B\Aff_N(\Z)).
\]
If $G$ is an elementary abelian 2-group,
then $\Phi^G(U_N/B\Hol(F_N))^\sharp$
factors through the map 
\[
	(\alpha_{\F_2})_\ast
	\colon 
	H_{\ast-\mathrm{shift}} ( B\Hol(F_N) )
	\longto 
	H_{\ast-\mathrm{shift}} (B\Aff_N(\F_2)).
\]
Here `shift' equals $\dim(G)(N-1)$.
\end{theorem}
Theorems~\ref{thm:hol-app} and \ref{thm:aff-app}
together imply the following companion to 
Corollary~\ref{cor:hol-app}.
\begin{corollary}
\label{cor:aff-app}
If $a\in H_\ast(B\Sigma_{n_1,\ldots,n_r})$
is such that 
$\Phi^G(S_{n_1,\ldots,{n_r}}/B\Sigma_{n_1,\ldots ,n_r})^\sharp (a)$
is non-trivial for some abelian compact Lie group $G$, then
$(\alpha_{\Z})_\ast \zeta_\ast(a)\in H_\ast(B\Aff_N(\Z))$
is non-zero. 
If $G$ can be chosen to be an elementary abelian 2-group, then
$(\alpha_{\F_2})_\ast \zeta_\ast(a)\in H_\ast(B\Aff_N(\F_2))$
is non-zero.
\qed
\end{corollary}

By Theorem~\ref{thm:quillen-calc},
all positive-dimensional elements in the homology of $\Aff_N(\F_2)$
obtained in this way are unstable.
The following companion to 
Corollary~\ref{cor:hol-fam2}
exhibits unstable elements in the homology of $\Aff_N(\Z)$
as well. It follows from 
Theorems~\ref{thm:hol-app}, 
\ref{thm:aut-app} and
\ref{thm:aff-app},
Corollary~\ref{cor:hol-fam2},
and the observation that 
the maps $\alpha_\Z$ (resp.\ $\alpha_{\F_2}$) 
are compatible 
with the inclusions $\Hol(F_n) \incl \Hol(F_{n+1})$
and $\Aff_n(\Z) \incl \Aff_{n+1}(\Z)$
(resp.\  $\Aff_n(\F_2) \incl \Aff_{n+1}(\F_2)$).
\begin{corollary}
\label{cor:aff-app2}
Let $a\in H_\ast(B\Sigma_{n_1,\ldots,n_r})$
be such that 
$\Phi^G(S_{n_1,\ldots,{n_r}}/B\Sigma_{n_1,\ldots ,n_r})^\sharp (a)$
is non-trivial for some abelian compact Lie group $G$
which is positive dimensional or finite of even order.
Then 
\[
    (\alpha_{\Z})_\ast j \rho \zeta_\ast(a)
    = 
    (\alpha_{\Z})_\ast \zeta_\ast(a) + 
    (\alpha_{\Z})_\ast i_\ast\pi_\ast \zeta_\ast(a)
    \in 
    H_\ast(B\Aff_N(\Z))
\]
is a non-trivial unstable element.
If $G$ can be chosen to be an elementary abelian 2-group,
then 
\[
    (\alpha_{\F_2})_\ast j \rho \zeta_\ast(a)
    =
    (\alpha_{\F_2})_\ast \zeta_\ast(a) + 
    (\alpha_{\F_2})_\ast i_\ast\pi_\ast \zeta_\ast(a)
    \in 
    H_\ast(B\Aff_N(\F_2))
\]
is a non-trivial unstable element.
\qed
\end{corollary}

\begin{proof}[Proof of Theorem~\ref{thm:aff-app}]
Observe that for abelian $G$, the actions of 
$\Hol(F_N)\times G^{\{p\}}$ on $G^N$ 
and 
$\Hol(F_N)\times G^{\{p,q\}}$ on $G^N\times G$ given by 
\eqref{eq:holgp-action} and \eqref{eq:holgpq-action}
factor through the homomorphisms
\[
	\alpha_\Z \times 1 
	\colon
	\Hol(F_N)\times G^{\{p\}} 
	\longto 
	\Aff_N(\Z)\times G^{\{p\}}
\]
and
\[
	\alpha_\Z \times 1 
	\colon
	\Hol(F_N)\times G^{\{p,q\}} 
	\longto 
	\Aff_N(\Z)\times G^{\{p,q\}},
\]
respectively, inducing an action of
$\Aff_N(\Z)\times G^{\{p\}}$
on $G^N$ and an action of 
$\Aff_N(\Z)\times G^{\{p,q\}}$
on $G^N\times G$.
The diagram \eqref{diag:push-pull} 
for computing $\Phi^G(U_N/B\Hol(F_N))$
now fits into a commutative diagram
\begin{equation*}
\label{diag:push-pull3}
\vcenter{\xymatrix@!0@C=6.4em@R=5ex{
	&
	G^N\sslash \Hol(F_N)\times G^{\{p\}}
	\ar[ddl]_{!}
	\ar[ddd]
	&&
	\ar[ll]_-\homot
	G^N\times G \sslash \Hol(F_N)\times G^{\{p,q\}}
	\ar[ddr]
	\ar[ddd]
	\\
	\\
	\pt\sslash \Hol(F_N)\times G^{\{p\}}
	\ar[ddd]
	&&&&
	\pt\sslash G^{\{q\}}
	\ar[ddd]^\id
	\\
	&
	G^N\sslash \Aff_N(\Z)\times G^{\{p\}}
	\ar[ddl]_{!}
	&&
	\ar[ll]_-\homot
	G^N\times G \sslash \Aff_N(\Z)\times G^{\{p,q\}}
	\ar[ddr]
	\\
	\\
	\pt\sslash \Aff_N(\Z)\times G^{\{p\}}
	&&&&
	\pt\sslash G^{\{q\}}	
}}
\end{equation*}
where the maps in the bottom part of the diagram are
induced by the evident projection maps of groups and
spaces and where the unlabeled vertical arrows are induced by 
$\alpha_\Z$. The left-hand parallelogram is a pullback
square, and induces a commutative square after passing
to homology and taking umkehr maps of the maps labeled by $!$
and ordinary induced maps of the vertical maps.
When taking the umkehr maps, the lower map labeled by $!$
should be interpreted as a map of fibrewise manifolds 
over $\pt\sslash \Aff_N(\Z)\times G^{\{p\}}$.
The claim for abelian compact Lie groups now follows from
Lemma~\ref{lm:uopcomp}. The claim for elementary abelian 
2-groups follows in a similar way from the observation that
the actions 
\eqref{eq:holgp-action} and \eqref{eq:holgpq-action}
in this case factor through the homomorphisms
\[
	\alpha_{\F_2} \times 1 
	\colon
	\Hol(F_N)\times G^{\{p\}} 
	\longto 
	\Aff_N(\F_2)\times G^{\{p\}}
\]
and
\[
	\alpha_{\F_2} \times 1 
	\colon
	\Hol(F_N)\times G^{\{p,q\}} 
	\longto 
	\Aff_N(\F_2)\times G^{\{p,q\}},
\]
respectively.
\end{proof}


\bibliographystyle{amsalpha}
\bibliography{higher-operations.bib}

\def\cprime{$'$}
\providecommand{\bysame}{\leavevmode\hbox to3em{\hrulefill}\thinspace}
\providecommand{\MR}{\relax\ifhmode\unskip\space\fi MR }
\providecommand{\MRhref}[2]{%
  \href{http://www.ams.org/mathscinet-getitem?mr=#1}{#2}
}
\providecommand{\href}[2]{#2}
\begin{thebibliography}{AMNY99}

\bibitem[AMNY99]{AMNY}
Dominique Arlettaz, Mamoru Mimura, Koji Nakahata, and Nobuaki Yagita, \emph{The
  mod {$2$} cohomology of the linear groups over the ring of integers}, Proc.
  Amer. Math. Soc. \textbf{127} (1999), no.~8, 2199--2212. \MR{1646320
  (99j:20052)}

\bibitem[Ash89]{AshFarrellCohom}
Avner Ash, \emph{Farrell cohomology of {${\rm GL}(n,{\bf Z})$}}, Israel J.
  Math. \textbf{67} (1989), no.~3, 327--336. \MR{1029906 (91c:11022)}

\bibitem[BB15]{BB}
Alexander Berglund and Kaj B\"{o}rjeson, \emph{Free loop space homology of
  highly connected manifolds}, {\tt arXiv:1502.03356v1}, 2015.

\bibitem[BGNX12]{BGNX}
Kai Behrend, Gr{\'e}gory Ginot, Behrang Noohi, and Ping Xu, \emph{String
  topology for stacks}, Ast\'erisque (2012), no.~343, xiv+169. \MR{2977576}

\bibitem[Bro82]{Brown}
Kenneth~S. Brown, \emph{Cohomology of groups}, Graduate Texts in Mathematics,
  vol.~87, Springer-Verlag, New York-Berlin, 1982. \MR{672956 (83k:20002)}

\bibitem[Cha05]{Chataur-bordism}
David Chataur, \emph{A bordism approach to string topology}, Int. Math. Res.
  Not. (2005), no.~46, 2829--2875. \MR{2180465 (2007b:55009)}

\bibitem[Che97]{Chen}
Yu~Qing Chen, \emph{Farrell cohomology of automorphism groups of free groups of
  finite rank}, Ph.D. thesis, Ohio State University, 1997.

\bibitem[CHKV15]{CHKV}
James Conant, Allen Hatcher, Martin Kassabov, and Karen Vogtmann,
  \emph{Assembling homology classes in automorphism groups of free groups},
  {\tt arXiv:1501.02351v2}, 2015.

\bibitem[Cla81]{Clapp}
M{\'o}nica Clapp, \emph{Duality and transfer for parametrized spectra}, Arch.
  Math. (Basel) \textbf{37} (1981), no.~5, 462--472. \MR{643290 (83i:55010)}

\bibitem[CLM76]{IteratedLoopSpaces}
Frederick~R. Cohen, Thomas~J. Lada, and J.~Peter May, \emph{The homology of
  iterated loop spaces}, Lecture Notes in Mathematics, Vol. 533,
  Springer-Verlag, Berlin-New York, 1976. \MR{0436146 (55 \#9096)}

\bibitem[CM12]{ChataurMenichi}
David Chataur and Luc Menichi, \emph{String topology of classifying spaces}, J.
  Reine Angew. Math. \textbf{669} (2012), 1--45. \MR{2980450}

\bibitem[CS99]{ChasSullivan}
Moira Chas and Dennis Sullivan, \emph{String topology}, {\tt
  arXiv:math/9911159}, 1999.

\bibitem[EVGS13]{EVGS}
Philippe Elbaz-Vincent, Herbert Gangl, and Christophe Soul{\'e}, \emph{Perfect
  forms, {K}-theory and the cohomology of modular groups}, Adv. Math.
  \textbf{245} (2013), 587--624. \MR{3084439}

\bibitem[Gal11]{Galatius}
S{\o}ren Galatius, \emph{Stable homology of automorphism groups of free
  groups}, Ann. of Math. (2) \textbf{173} (2011), no.~2, 705--768. \MR{2784914
  (2012c:20149)}

\bibitem[GH09]{GoreskyHingston}
Mark Goresky and Nancy Hingston, \emph{Loop products and closed geodesics},
  Duke Math. J. \textbf{150} (2009), no.~1, 117--209. \MR{2560110
  (2010k:58021)}

\bibitem[GMV98]{GloverMislinVoon}
H.~H. Glover, G.~Mislin, and S.~N. Voon, \emph{The {$p$}-primary {F}arrell
  cohomology of {${\rm Out}(F_{p-1})$}}, Geometry and cohomology in group
  theory ({D}urham, 1994), London Math. Soc. Lecture Note Ser., vol. 252,
  Cambridge Univ. Press, Cambridge, 1998, pp.~161--169. \MR{1709957
  (2000g:20103)}

\bibitem[God07]{Godin}
V\'eronique Godin, \emph{Higher string topology operations}, {\tt
  arXiv:0711.4859v2}, 2007.

\bibitem[Hep09]{HepworthProjective}
Richard Hepworth, \emph{String topology for complex projective spaces}, {\tt
  arXiv:0908.1013}, 2009.

\bibitem[Hep10]{HepworthLie}
\bysame, \emph{String topology for {L}ie groups}, J. Topol. \textbf{3} (2010),
  no.~2, 424--442. \MR{2651366 (2011m:55009)}

\bibitem[Hig05]{Higgins}
P.~J. Higgins, \emph{Categories and groupoids}, Repr. Theory Appl. Categ.
  (2005), no.~7, 1--178, Reprint of the 1971 original [{Notes on categories and
  groupoids}, Van Nostrand Reinhold, London; MR0327946] with a new preface by
  the author. \MR{2118984 (2005k:20137)}

\bibitem[HL15]{HL}
Richard Hepworth and Anssi Lahtinen, \emph{On string topology of classifying
  spaces}, Adv. Math. \textbf{281} (2015), 394--507. \MR{3366844}

\bibitem[HV04]{HV-stab}
Allen Hatcher and Karen Vogtmann, \emph{Homology stability for outer
  automorphism groups of free groups}, Algebr. Geom. Topol. \textbf{4} (2004),
  1253--1272. \MR{2113904 (2005j:20038)}

\bibitem[HVW06]{HV-stab-erratum}
Allan Hatcher, Karen Vogtmann, and Natalie Wahl, \emph{Erratum to: ``{H}omology
  stability for outer automorphism groups of free groups''}, Algebr. Geom.
  Topol. \textbf{6} (2006), 573--579 (electronic). \MR{2220689 (2006k:20069)}

\bibitem[Jen01]{JensenAut}
Craig~A. Jensen, \emph{Cohomology of {${\rm Aut}(F_n)$} in the {$p$}-rank two
  case}, J. Pure Appl. Algebra \textbf{158} (2001), no.~1, 41--81. \MR{1815782
  (2003a:20079)}

\bibitem[Jen04]{JensenHol}
\bysame, \emph{Homology of holomorphs of free groups}, J. Algebra \textbf{271}
  (2004), no.~1, 281--294. \MR{2022483 (2004k:20111)}

\bibitem[Kau08]{Kaufmann}
Ralph~M. Kaufmann, \emph{Moduli space actions on the {H}ochschild co-chains of
  a {F}robenius algebra. {II}. {C}orrelators}, J. Noncommut. Geom. \textbf{2}
  (2008), no.~3, 283--332. \MR{2411420 (2009e:55015)}

\bibitem[Knu01]{Knudson}
Kevin~P. Knudson, \emph{Homology of linear groups}, Progress in Mathematics,
  vol. 193, Birkh\"auser Verlag, Basel, 2001. \MR{1807154 (2001j:20070)}

\bibitem[LS78]{LeeSzczarba}
Ronnie Lee and R.~H. Szczarba, \emph{On the torsion in {$K_{4}({\bf Z})$} and
  {$K_{5}({\bf Z})$}}, Duke Math. J. \textbf{45} (1978), no.~1, 101--129.
  \MR{0491893 (58 \#11074a)}

\bibitem[Man95]{ManjrekarGL4}
Rajesh Manjrekar, \emph{The mod-{$3$} cohomology of {${\rm GL}(4,{\bf Z})$} and
  {${\rm GL}(2,{\bf Z}[\omega])$}}, Comm. Algebra \textbf{23} (1995), no.~8,
  3099--3126. \MR{1332169 (96c:20081)}

\bibitem[Man96]{Manjrekar}
\bysame, \emph{The mod-{$p$} cohomology of {${\rm GL}(2p-2,{\bf Z})$}}, J.
  Algebra \textbf{181} (1996), no.~3, 697--726. \MR{1386576 (97g:20048)}

\bibitem[May75]{MayClassifying}
J.~Peter May, \emph{Classifying spaces and fibrations}, Mem. Amer. Math. Soc.
  \textbf{1} (1975), no.~1, 155, xiii+98. \MR{0370579 (51 \#6806)}

\bibitem[Men09]{MenichiSpheres}
Luc Menichi, \emph{String topology for spheres}, Comment. Math. Helv.
  \textbf{84} (2009), no.~1, 135--157, With an appendix by Gerald Gaudens and
  Menichi. \MR{2466078 (2009k:55017)}

\bibitem[Mit92]{MitchellGL}
Stephen~A. Mitchell, \emph{On the plus construction for {$B{\rm GL}\,{\bf
  Z}[\frac12]$} at the prime {$2$}}, Math. Z. \textbf{209} (1992), no.~2,
  205--222. \MR{1147814 (93b:55021)}

\bibitem[MM79]{MadsenMilgram}
Ib~Madsen and R.~James Milgram, \emph{The classifying spaces for surgery and
  cobordism of manifolds}, Annals of Mathematics Studies, vol.~92, Princeton
  University Press, Princeton, N.J.; University of Tokyo Press, Tokyo, 1979.
  \MR{548575 (81b:57014)}

\bibitem[MS06]{MaySigurdsson}
J.~P. May and J.~Sigurdsson, \emph{Parametrized homotopy theory}, Mathematical
  Surveys and Monographs, vol. 132, American Mathematical Society, Providence,
  RI, 2006. \MR{2271789 (2007k:55012)}

\bibitem[Nak60]{NakaokaDecomposition}
Minoru Nakaoka, \emph{Decomposition theorem for homology groups of symmetric
  groups}, Ann. of Math. (2) \textbf{71} (1960), 16--42. \MR{0112134 (22
  \#2989)}

\bibitem[PR11]{PoirierRounds}
Kate Poirier and Nathaniel Rounds, \emph{Compactifying string topology}, {\tt
  arXiv:1111.3635}, 2011.

\bibitem[Qui72]{QuillenFiniteFields}
Daniel Quillen, \emph{On the cohomology and {$K$}-theory of the general linear
  groups over a finite field}, Ann. of Math. (2) \textbf{96} (1972), 552--586.
  \MR{0315016 (47 \#3565)}

\bibitem[Sat06]{Satoh1}
Takao Satoh, \emph{Twisted first homology groups of the automorphism group of a
  free group}, J. Pure Appl. Algebra \textbf{204} (2006), no.~2, 334--348.
  \MR{2184815 (2006j:20079)}

\bibitem[Sat07]{Satoh2}
\bysame, \emph{Twisted second homology groups of the automorphism group of a
  free group}, J. Pure Appl. Algebra \textbf{211} (2007), no.~2, 547--565.
  \MR{2341270 (2008i:20063)}

\bibitem[Sou78]{SouleSL3}
Christophe Soul{\'e}, \emph{The cohomology of {${\rm SL}_{3}({\bf Z})$}},
  Topology \textbf{17} (1978), no.~1, 1--22. \MR{0470141 (57 \#9908)}

\bibitem[Tam06]{TamanoiStiefel}
Hirotaka Tamanoi, \emph{Batalin-{V}ilkovisky {L}ie algebra structure on the
  loop homology of complex {S}tiefel manifolds}, Int. Math. Res. Not. (2006),
  Art. ID 97193, 23. \MR{2211159 (2006m:55026)}

\bibitem[Tam09]{TamanoiStable}
\bysame, \emph{Stable string operations are trivial}, Int. Math. Res. Not. IMRN
  (2009), no.~24, 4642--4685. \MR{2564371 (2010k:55015)}

\bibitem[Tur97]{Turner}
Paul~R. Turner, \emph{Dickson coinvariants and the homology of {$QS^0$}}, Math.
  Z. \textbf{224} (1997), no.~2, 209--228. \MR{1431193 (98e:55023)}

\bibitem[vdK80]{linStability}
Wilberd van~der Kallen, \emph{Homology stability for linear groups}, Invent.
  Math. \textbf{60} (1980), no.~3, 269--295. \MR{586429 (82c:18011)}

\bibitem[Wah12]{Wahl}
Nathalie Wahl, \emph{Universal operations in {H}ochschild homology}, {\tt
  arXiv:1212.6498v2}, to appear in J. Reine Angew. Math., 2012.

\bibitem[WW16]{WahlWesterland}
Nathalie Wahl and Craig Westerland, \emph{Hochschild homology of structured
  algebras}, Adv. Math. \textbf{288} (2016), 240--307. \MR{3436386}

\bibitem[Yan13]{Yang}
Tian Yang, \emph{A {B}atalin-{V}ilkovisky algebra structure on the {H}ochschild
  cohomology of truncated polynomials}, Topology Appl. \textbf{160} (2013),
  no.~13, 1633--1651. \MR{3091339}

\end{thebibliography}

\end{document}